\documentclass[12pt]{amsart}
\usepackage{color}
\usepackage{amssymb}
\usepackage{comment}
\usepackage{pdfpages}
\usepackage{graphicx}
\usepackage{dcolumn}

\RequirePackage[numbers]{natbib}
\RequirePackage[colorlinks=true, pdfstartview=FitV, linkcolor=blue,
  citecolor=blue, urlcolor=blue]{hyperref}
\usepackage{paralist}

\usepackage{tikz} 
\usepackage{dcolumn}
\usetikzlibrary{arrows,positioning}
\tikzset{
    >=stealth',
    pil/.style={
           ->,
           thick,
           shorten <=2pt,
           shorten >=2pt,}
}

\newcommand{\cC}{{\mathcal{C}}}

\newcommand{\EvCo}[1]{\mathcal{A}_{#1}^{(\mathrm{coll})}}
\newcommand{\EvSp}[1]{\mathcal{A}_{#1}^{(\mathrm{sparse})}}
\newcommand{\EvGap}[1]{\mathcal{A}_{#1}^{(\mathrm{gap})}}
\newcommand{\EvReg}[1]{\mathcal{A}_{#1}^{(\mathrm{regular})}}
\newcommand{\EvCov}[1]{\mathcal{A}_{#1}^{(\mathrm{cover})}}

\newcommand{\colT}[1]{\mathcal{T}_{\mathrm{col}}(#1)}
\newcommand{\colM}[1]{\mathcal{M}_{\mathrm{col}}(#1)}

\newcommand{\tS}[1]{\tau_{\mathrm{start}}^{(#1)}}
\newcommand{\tTr}[1]{\tau_{\mathrm{triple}}^{(#1)}}
\newcommand{\tDe}[1]{\tau_{\mathrm{decoupling}}^{(#1)}}

\newcommand{\tCo}{\tau_{\mathrm{col}}}

\newcommand{\tRe}{\tau_{\mathrm{ret}}}
\newcommand{\tNe}[1]{\tau_{\mathrm{near},#1}}

\usepackage{graphicx}
\graphicspath{{./figures/}}
\usepackage{amsfonts}
\usepackage{amsmath}
\usepackage{amsthm}
\usepackage{amssymb}
\usepackage{amsbsy}
\usepackage{epsfig}
\usepackage{fullpage}
\usepackage{natbib, mathrsfs} 
\usepackage{verbatim}
\usepackage[latin1]{inputenc}
\usepackage{mhequ}
\usepackage{algorithm}
\usepackage{algorithmic}

\numberwithin{equation}{section}

\def \be{\begin{equs}}
\def \ee{\end{equs}}
\def \P{\mathbb{P}}
\def \E{\mathbb{E}}

\newcommand \tcb[1]{\textcolor{blue}{(#1)}}

\newcommand \TV{\mathrm{TV}}

\def \tmix{\tau_{\mathrm{mix}}}
\def \tmixR{\tau_{\mathrm{mix}}^{\mathrm{RW}}}

\def \TV{\mathrm{TV}}

\def \O{\mathrm{Occ}}
\def \tpi{\tilde{\pi}}
\def \LL {\Lambda(L,d)}
\def \cN{\mathcal{N}}
\def \T{\mathtt{T}}

\def \S{\mathcal{S}}
\def \ck{\mathcal{K}}
\def \zint{\zeta_\mathrm{int}}
\def \ztr{\zeta_\mathrm{triple}}
\def \Th{\Theta}
\def \hTh{\widehat{\Theta}}
\def \near{\mathcal{N}^\mathrm{near}}

\newtheorem{theorem}{Theorem}[section]
\newtheorem{lemma}[theorem]{Lemma}

\newtheorem{prop}[theorem]{Proposition}

\theoremstyle{plain}
\newtheorem{thm}{Theorem}
\newtheorem*{thm-non}{Theorem}
\newtheorem{cor}[thm]{Corollary}
\newtheorem{conj}[thm]{Conjecture}

\theoremstyle{definition}
\newtheorem{defn}[theorem]{Definition}
\newtheorem{remark}[theorem]{Remark}

\begin{document}

\title[Mixing times for a Constrained Ising Process on the torus at low density]
{Mixing times for a Constrained Ising Process on the torus at low density}


\author{Natesh S. Pillai$^{\ddag}$}
\thanks{$^{\ddag}$pillai@fas.harvard.edu, 
   Department of Statistics,
    Harvard University, 1 Oxford Street, Cambridge
    MA 02138, USA}

\author{Aaron Smith$^{\sharp}$}
\thanks{$^{\sharp}$smith.aaron.matthew@gmail.com, 
   Department of Mathematics and Statistics,
University of Ottawa, 585 King Edward Avenue, Ottawa
ON K1N 7N5, Canada}

\maketitle






\begin{abstract}
We study a kinetically constrained Ising process (KCIP) associated with a graph $G$ and density parameter $p$; this process is an interacting particle system with state space $\{ 0, 1 \}^{G}$. The stationary distribution of the KCIP Markov chain is the Binomial($|G|$, $p$) distribution on the number of particles, conditioned  on having at least one particle. The `constraint' in the name of the process refers to the rule that a vertex cannot change its state unless it has at least one neighbour in state `1'. The KCIP has been proposed by statistical physicists as a model for the glass transition, and more recently as a simple algorithm for data storage in computer networks. In this note, we study the mixing time of this process on the torus $G = \mathbb{Z}_{L}^{d}$, $d \geq 3$, in the low-density regime $p = \frac{c}{|G|}$ for arbitrary $0 < c < \infty$; this regime is the subject of a conjecture of Aldous and is natural in the context of computer networks. Our results provide a counterexample to Aldous' conjecture, suggest a natural modification of the conjecture, and show that this modification is correct up to logarithmic factors. The methods developed in this paper also provide a strategy for tackling Aldous' conjecture for other graphs.
\end{abstract}
\section{Introduction} \label{SecProbDesc}
The kinetically constrained Ising process (KCIP) refers to a class of interacting particle systems introduced by physicists in \cite{FrAn84, FrAn85} to study the glass transition. Versions of this process have accrued other names since then, including the kinetically constrained spin model, the east model \cite{AlDi02} and the north-east model  \cite{Vali04}. These models have attracted a great deal of interest recently, including applications to combinatorics, computer science, and other areas; \cite{ChMa13,CFM14b} have useful surveys of places that the KCIP has appeared outside of the physics literature. Recent mathematical progress has included new bounds on the mixing time of the KCIP in various different regimes \cite{KoLa06, CMRT09, MaTo13, ChMa13, CFM14, CFM14b, CFM15}. For a more complete review of recent progress on KCIP within the physics community, see the survey \cite{GST11} and the references therein. \par

In this note, we study a simple discrete-time version of this process, though our main result applies, after suitable time scaling, to standard continuous-time analogues as well. Fix a graph $G = (V,E)$ and a density parameter $0 < p < 1$.  For a set $S$, we denote by $\mathrm{Unif}(S)$ the uniform distribution on $S$. Define a reversible Markov chain $\{ X_t \}_{t \in \mathbb{N}}$ on the set of $\{0,1\}$-labellings of $G$ as follows.  To update the chain $X_{t}$, choose 
\be [EqCiRep]
v_{t} &\sim \mathrm{Unif}(V), \\
p_{t} &\sim \mathrm{Unif}([0,1]).
\ee
If there exists $u \in V$ such that $(u,v_{t}) \in E$ and $X_{t}[u] = 1$, set $X_{t+1}[v_{t}] = 1$ if $p_{t} \leq p$ and set $X_{t+1}[v_{t}] = 0$ if $p_{t} > p$. If no such $u \in V$ exists, set $X_{t+1}[v_t] = X_{t}[v_t]$. In either case, set $X_{t+1}[w] = X_{t}[w]$ for all $w \in V \backslash \{ v_t \}$. \par 
The state space for the KCIP $\{ X_t \}_{t \in \mathbb{N}}$ on a graph $G$ is $\Omega =  \{0,1\}^G$. Set $|V| = n$; for general points $x \in \{0,1,\ldots,|V|\}^{G}$, write $\vert x \vert = \sum_{v \in G} \textbf{1}_{x[v] \neq 0}$. Let $\pi$ denote the stationary distribution of $\{ X_t \}_{t \in \mathbb{N}}$. For $y \in \Omega$, this is given by
\be \label{eqn:pistat}
\pi(y) = {1 \over \mathcal{Z}_{\mathrm{KCIP}}}\, p^{|y|}(1-p)^{n-|y|}\, \textbf{1}_{|y| > 0},
\ee
where $\mathcal{Z}_{\mathrm{KCIP}} = 1 - (1-p)^n$ is the normalizing constant (see formulas \eqref{EqStatDistCheck1} $\&$ \eqref{EqStatDistCheck2} below). Thus $\pi(y)$ is proportional to the Binomial$(n,p)$ distribution on the number of non-zero labels in $y \in \Omega$, conditional on having at least one non-zero entry. \par

In this paper we study a conjecture stated by David Aldous in  \cite{Aldo12} about the mixing time of this process. To state Aldous' conjecture, we recall some standard notation that will be used throughout the paper. For sequences $x = x(n),y = y(n)$ indexed by $\mathbb{N}$, we write $y = O(x)$ for $\sup_{n} \frac {|y(n)|}{|x(n)|} \leq C < \infty$ and $y = o(x)$ for $\limsup_{n \rightarrow \infty} \frac {|y(n)|}{|x(n)|} = 0$. Recall that for distributions $\mu, \nu$ on a common measure space $(\Theta, \mathcal{A})$, the \textit{total variation} distance between $\mu$ and $\nu$ is given by
\be 
\| \mu - \nu \|_{\TV} = \sup_{A \in \mathcal{A}} (\mu(A) - \nu(A)).
\ee 
The  \textit{mixing profile} for the KCIP Markov chain $\{ X_{t} \}_{t \in \mathbb{N}}$ on $\Omega$ with stationary distribution $\pi$ is given by
\be 
\tau(\epsilon) = \inf \Big \{ t > 0 \, : \, \sup_{X_{0}  = x \in \Omega}\| \mathcal{L}(X_{t}) - \pi \|_{\TV} < \epsilon \Big \}
\ee 
for all $0 < \epsilon < 1$. As usual, the \textit{mixing time} is defined as $\tmix = \tau \big( \frac{1}{4} \big)$. Aldous' conjecture is \cite{Aldo12}:

\begin{conj}[Aldous] \label{ConjAldous}
The mixing time $\tmix$ of the constrained Ising process with parameter $p$ on graph $G$ is $O(p^{-1} |E| \tmixR)$, where $\tmixR$ is the mixing time of the $\frac{1}{2}$-lazy simple random walk on the graph $G$.
\end{conj}

Although Conjecture \ref{ConjAldous} is quite general, it was made in  the context of studying the KCIP on a sequence of graphs $\{ G_{n} \}_{n\in \mathbb{N}}$ with associated density $p = p_{n} = \frac{c}{\vert G_{n} \vert}$ for some fixed $0 < c < \infty$ \tcb{\cite{Aldo12}}. This scaling regime for $p_n$ is natural for studying the low-temperature limit of the physical process and has been referred to as the natural equilibrium scale \cite{LectureFM14}; however, its motivation in \cite{Aldo12} is as a model for data storage in computer networks rather than as a model for physical processes.  \par
For a positive integer $L \in \mathbb{N}$, let $\Lambda(L,d)$ denote the $d$-dimensional torus with $n = L^{d}$ points; this is a Cayley graph with vertex set, generating set and edge set given by 
\be 
V &= \mathbb{Z}_{L}^{d}, \\
\mathrm{Gen} &= \{ (1,0,0,\ldots,0), (0,1,0,\ldots,0), \ldots, (0,0,0,\ldots,1) \}, \\
E &= \Big \{ (u,v) \in V \times V  \, : u - v \in \pm \mathrm{Gen}  \}.
\ee 
Set
\be
 n = |\Lambda(L,d)|  = L^d.
\ee 
In this paper, we study the KCIP on a sequence of graphs $\{ \Lambda(L,d) \}_{L \in \mathbb{N}}$ with density 
\be \label{eqn:pcn}
p = p_{n} = {c \over n}
\ee
 for some fixed  constant $0 < c < \infty $ and fixed dimension $d \geq 3$; see Section \ref{SecConclusion} for a brief explanation of how our method applies when $d$ and $c$ are allowed to vary with $L$. \par
 The mixing time of the simple random walk on $G = \LL$ is known to be $\tmixR \approx n^{\frac{2}{d}}$ (see, \textit{e.g.}, Theorem 5.5 of \cite{LPW09}). Thus Aldous' conjecture for $G =\Lambda(L,d)$ suggests a mixing time of $\tmix = O \big( n^{2 + \frac{2}{d}} \big)$. This is correct for $d = 1$ (as shown in \cite{AlDi02}), but we show in Theorem \ref{ThmMainResult} below that this conjecture is incorrect for $d \geq 3$. \par The following is our main result, in which we prove a modified version of Conjecture \ref{ConjAldous} for the torus: 

\begin{thm} [Mixing of the Constrained Ising Process on the Torus] \label{ThmMainResult}
Fix $0 < c < \infty$ and $d \geq 3$. For $p=p_n$ defined in \eqref{eqn:pcn}, the mixing time of the KCIP on $\LL$ satisfies
\be 
C_{1} n^{3} \leq \tmix \leq C_{2} n^{3} \log(n)
\ee 
for some constants $C_{1}, C_{2}$ that may depend on $c,d$ but are independent of $n$. 
\end{thm}

\begin{remark}
As can seen from Theorem \ref{ThmLowerBoundOnMixingTime} below, the lower bound  for the mixing time of KCIP obtained in Theorem \ref{ThmMainResult} is valid for any $m$-regular graph ($m >1$) with no triangles.
\end{remark}
In the statement of Theorem \ref{ThmMainResult} and throughout the paper, we assume that both the quantities $0 < c < \infty$ and $3 \leq d \in \mathbb{N}$ are fixed; only $n$ grows. In particular, in Theorem \ref{ThmMainResult} and all other calculations, bounds that are `uniform' are implied to be uniform only in $n$ and other explicitly mentioned variables; they will generally not be uniform in $c$ or $d$. Throughout, we will denote by $C$ a generic constant, whose value may change from one occurrence to the next, but is independent of $n$. For $u,v \in \LL$, denote by $|u - v|$ the smallest number of edges needed to traverse from $u$ to $v$ via a connected path; this is the usual graph distance. Finally, let
\be \label{eqn:Ball}
\mathcal{B}_{\ell}(v) = \{ w \in \Lambda(L,d) \, : \, |v - w | \leq \ell \} 
\ee
be the ball of radius $\ell$ around $v$ in the graph distance.
\subsection{Relationship to previous work}
Although the KCIP we study was introduced in the physics literature \cite{FrAn84,FrAn85} and discussed in later work such as \cite{Aldo12}, most recent mathematical work on mixing bounds for KCIPs has focussed on different local constraints, which give rise to qualitatively different behaviour. Thus, even if we studied the same regime, our results would not imply (or be implied by) recent work in this area. However, our primary contribution to the literature is the fact that our work is in a new regime: we obtain good mixing bounds on a KCIP that apply in the regime of high dimension $d \geq 3$, low density $p_{n} \approx n^{-1}$, and under the strong metric $\| \cdot \|_{\TV}$. All three of these distinctions can make the problem harder than working in dimension $d=1$, at high density $p_{n} = p$, or in a weaker metric.\par
We briefly review some recent work on the mixing properties of related constrained Ising processes \cite{KoLa06, CMRT09, MaTo13, ChMa13, CFM14, CFM14b, CFM15}. However our results in the regime which interests us are novel. We contrast our work with specific papers. Many previous results, such as \cite{ChMa13, CFM15}, deal primarily with the regime in which $p$ is a constant, independently of $n$. In this regime, many KCIPs mix relatively quickly and the obstacles to mixing are quite different. Other results, such as \cite{AlDi02,GLM14}, study the small-$p$ regime, but only in one dimension. In particular, the methods employed in \cite{AlDi02} completely break down for $d > 1$ and thus are not applicable to our setting. The recent paper \cite{CFM14b} seems most similar to ours. In \cite{CFM14b} the authors study the mixing of a related KCIP on $\LL$ at density $p \approx \frac{1}{n}$ and obtain results in greater generality than ours. However, the authors of \cite{CFM14b} focus on bounding the relaxation time of the process, rather than the mixing time; the bounds obtained in \cite{CFM14b} cannot be used to obtain sharp estimates on the mixing time.
\subsection{Outline for the paper}
This paper is largely devoted to the proof of Theorem \ref{ThmMainResult}. In Section \ref{SecRoadmap}, we set up notation and give a proof sketch. In Section \ref{SecGenMix}, we give a general upper bound on the mixing time of Markov chains on finite state spaces; the rest of the proof of the upper bound in Theorem \ref{ThmMainResult} consists of estimating the constants in this general bound. In Section \ref{sec:lowerbd}, we give the lower bound in Theorem \ref{ThmMainResult}. Sections \ref{SecCompToExc} to \ref{SecExcLength} contain most of the work required to prove the upper bound in Theorem \ref{ThmMainResult}. In Section \ref{SecCompToExc}, we study the behaviour of the KCIP at low density by comparing it to the simple exclusion process. In Sections \ref{SecDriftCond} and \ref{SecExcLength}, we detail the behaviour of the KCIP at high density by comparing it to the coalescence process. In Section \ref{SecProofThm}, we combine the results obtained in earlier sections and give the proof of Theorem \ref{ThmMainResult}. Finally, in Section \ref{SecConclusion} we discuss related problems and the extent to which our methods apply to them. 

\section{A Roadmap for the Proof} \label{SecRoadmap}
We first explain the heuristic arguments that make Conjecture \ref{ConjAldous} plausible and point out the key time scales involved. For the reader's convenience, we also give a proof sketch highlighting  all the major steps involved in the proof of Theorem \ref{ThmMainResult}.
\subsection{Heuristics and key time scales} 
We discuss heuristics for the time scales on which important changes to the KCIP occur. 

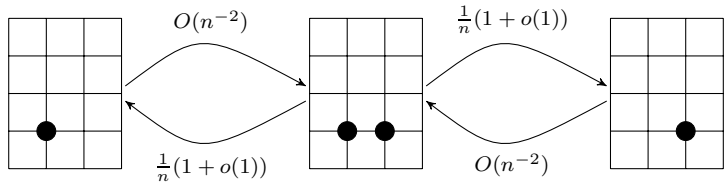
\begin{figure} 
\begin{tikzpicture}[scale=0.5]
\draw (0 , 0) grid (3, 4);
\draw  (8,0) grid (11,4);
\draw  (16,0) grid (19,4);
\draw [fill] (1,1) circle [radius=.25];
\draw [fill] (9,1) circle [radius=.25];
\draw [fill] (10,1) circle [radius=.25];
\draw [fill] (18,1) circle [radius=.25];
 \draw [->] (3.1, 2.2) .. controls (5.1, 3.7) .. (7.9,2.2)
        node[pos = 0.55, above]{\tiny{$O(n^{-2})$}};
  \draw [->] (7.9,1.8) .. controls (5.1, 0.3) .. (3.1, 1.8)
        node[pos = 0.45, below]{\tiny{$\frac{1}{n}(1+ o(1))$}};
 \draw [->] (11.1, 2.2) .. controls (13.1, 3.7) .. (15.9,2.2)
        node[pos = 0.55, above]{\tiny{$\frac{1}{n}(1+ o(1))$}};  \draw [->] (15.9,1.8) .. controls (13.1, 0.3) .. (11.1, 1.8)
        node[pos = 0.45, below]{\tiny{$O(n^{-2})$}};
 \end{tikzpicture}
\caption{Simple Random Walk Heuristic} 
\label{PicSRWHeuristic1}
\end{figure} 
\begin{figure}
\begin{tikzpicture}[scale=0.5]
\draw (0 , 0) grid (3, 4);
\draw  (8,0) grid (11,4);
\draw  (16,0) grid (19,4);
\draw  (24,0) grid (27,4);
\draw [fill] (1,1) circle [radius=.25];
\draw [fill] (9,1) circle [radius=.25];
\draw [fill] (10,1) circle [radius=.25];
\draw [fill] (17,1) circle [radius=.25];
\draw [fill] (18,1) circle [radius=.25];
\draw [fill] (18,2) circle [radius=.25];
\draw [fill] (25,1) circle [radius=.25];
\draw [fill] (26,2) circle [radius=.25];
 \draw [->] (3.1, 2.2) .. controls (5.1, 3.7) .. (7.9,2.2)
        node[pos = 0.55, above]{\tiny{$O(n^{-2})$}};
  \draw [->] (7.9,1.8) .. controls (5.1, 0.3) .. (3.1, 1.8)
        node[pos = 0.45, below]{\tiny{$\frac{1}{n}(1+ o(1))$}};
 \draw [->] (11.1, 2.2) .. controls (13.1, 3.7) .. (15.9,2.2)
        node[pos = 0.55, above]{\tiny{$O(n^{-2})$}};
  \draw [->] (15.9,1.8) .. controls (13.1, 0.3) .. (11.1, 1.8)
        node[pos = 0.45, below]{\tiny{$\frac{1}{n}(1+ o(1))$}};
 \draw [->] (19.1, 2.2) .. controls (21.1, 3.7) .. (23.9,2.2)
        node[pos = 0.55, above]{\tiny{$\frac{1}{n}(1+ o(1))$}};
  \draw [->] (23.9,1.8) .. controls (21.1, 0.3) .. (19.1, 1.8)
        node[pos = 0.45, below]{\tiny{$O(n^{-2})$}};
 \end{tikzpicture}
\caption{Particle Creation Heuristic}
\label{PicSRWHeuristic2}
\end{figure}
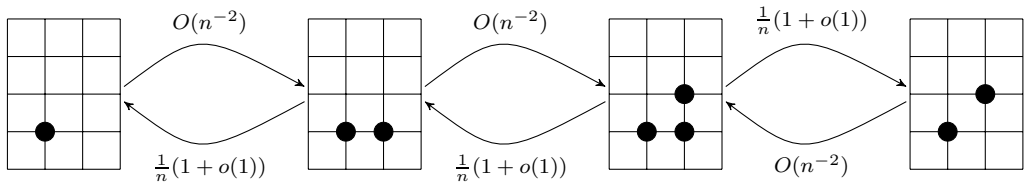 
\begin{enumerate}
\item If the initial state $X_{0}$ has a large number of particles, the number of particles remaining will generally be cut in half every $O(n^{3})$ steps. This observation is crucial to our proof. We show this bound using closely related bounds on the `coalescence time' of a collection of random walkers on a graph (see \cite{Cox89}). This bound implies that, after an initial transient period of at most $O(n^{3} \log(n))$ steps, the number of particles in the KCIP will be $O(1)$. This is the only place we obtain an extra $\log n$ factor in our proof of the upper bound of Theorem \ref{ThmMainResult}. We suspect that, in fact, this transient period is only of total length $O(n^{3})$. 
\item After the initial transient period, the KCIP generally has $O(1)$ well-separated particles. Every $O(n^{2})$ steps, a particle will spawn a neighbour; one of these neighboring particles will be removed in $O(n)$ steps. Ignoring the times at which any particle has a neighbour, the $O(1)$ well-separated particles will appear to be evolving according to an independent random walk on $\LL$, slowed down by a factor of roughly $n^{2}$; see Figure \ref{PicSRWHeuristic1}. Since all particles have no neighbours at most times under the stationary measure, ignoring the times at which any particle has a neighbour does not greatly influence one's view of the process.
\item After the initial transient period, and again ignoring the times at which any particle has a neighbour, it often takes $O(n^{3})$ steps to decrease the number of particles by 1. To see this, note that the number of particles can only decrease when two existing particles `collide.' Recall that the expected collision time for two random walkers on the torus is $O(n)$. By the above heuristic, the particles in the KCIP are undergoing simple random walk that is slowed down by a factor of $n^{2}$; thus, the expected time to a collision is $O(n^{3})$. It also turns out that collisions `often' result in the number of particles being decreased.
\item After the initial transient period, and again ignoring the times at which any particle has a neighbour, it takes $O(n^{3})$ steps to increase the number of particles by 1. Indeed, increasing the number of non-adjacent particles by 1 requires an intermediate time at which three particles are `touching' in $\LL$, where two of them share the same neighbor. It takes $O(n^{2})$ steps for a particle to spawn a neighbouring particle, and whenever two particles are adjacent, the probability of one of these two particles being removed before a third `touching' particle is added is $1 - O(n^{-1})$. Thus, it takes $O(n^{3})$ steps to obtain three `touching' particles; see Figure \ref{PicSRWHeuristic2} for a generic illustration of how this happens. It is easy to check that a configuration with three `touching' particles often degenerates into one with two non-adjacent particles after $O(n)$ steps.
\end{enumerate}
The basis of Conjecture \ref{ConjAldous} is heuristic (2) above: individual particles at distance greater than one in the KCIP on $\LL$ tend to behave like independent random walkers on $\LL$, slowed down by a factor of roughly $n^{2}$. Thus, for density $p = \frac{c}{n}$, we might expect the KCIP to have roughly $c$ particles during most times and to behave quite similarly to the simple exclusion process (SEP) with $c$ particles (see \cite{ClSu73, HoLi75} for an introduction). Heuristics (1) and (4) explain why Conjecture \ref{ConjAldous} is not telling the whole story: heuristic (1) points out that it takes a long time to go from $n$ particles to $O(1)$ particles, while heuristic (4) points out that it takes a long time to go from 1 particle to 2 well-separated particles. The lower bound in Theorem \ref{ThmMainResult} is obtained by making heuristic (4) rigorous.\par
An obvious modification to Conjecture \ref{ConjAldous} is that the mixing time of the KCIP is at most the maximum of these three time scales, and this is the approach we take in this paper. There are essentially three obstacles to making this approach rigorous. The first is to deal with the fact that we would like to compare a single KCIP to many different SEPs - the SEP with $c$ particles, but also the SEP with $c-1$ particles, $c+1$ particles, \textit{etc.} Breaking apart the KCIP in this way is the subject of ${\bf Step\, 1}$ and ${\bf Step\, 4}$ in the proof sketch below. The second is to complete the comparison of the EP to a suitably tamed version of the KCIP. This is the subject of ${\bf Step\, 2}$ below. The third is to ensure that a KCIP started with a large number ($\gg c$) particles quickly enters a state with roughly $c$ particles. This corresponds to ${\bf Step\, 3}$ below. 
\subsection{Proof sketch}\label{sec:proofsketch}
Our methods in some of the steps in the following outline are applicable for a generic graph $G$, and the others are specific to the torus $\LL$. For $1 \leq k \leq {n\over 2}$, let $\Omega_{k} \subset \Omega$ be configurations of $k$ particles for which no two particles are adjacent, \textit{i.e.},
\be \label{EqDefOmegaK}
\Omega_k = \big \{ X \in \{0,1\}^G \, : \, \sum_{v \in V} X[v] = k, \sum_{(u,v) \in E} X[u]X[v] = 0 \big \}.
\ee 
Also set $\Omega' = \Omega \backslash \cup_{k=1}^{{n\over 2}} \Omega_{k}$. For each $k \leq n$, we will denote by $\tmix^{(k)}$ the mixing time of $X_t$ `restricted to' $\Omega_k$ (our notion of `restriction' is defined more carefully in Section \ref{SecCompToExc}). Define the quantity
\be
\O_k(\epsilon, N) = \sup_{x \in \Omega} \, \inf\Big \{T \geq 1 \, : \, X_{0} = x, \, \P\big(\sum_{s=0}^T 1_{X_s \in \Omega_k} > N\big) > 1-\epsilon \Big\}.
\ee
For a fixed $N$ and small $\epsilon$, $\O_k (\epsilon, N)$ denotes the first (random) time at which the occupation measure of $X_t$ in $\Omega_k$ exceeds $N$ with high $(1-\epsilon)$ probability. For $x \in \Omega_k$, define the exit time:
\be \label{EqExitTimeDef}
L_k(x) = \inf \big\{t : X_0 = x, X_t \in \cup_{j \neq k} \Omega_j \big\}.
\ee
\noindent Our proof strategy for the upper bound in Theorem \ref{ThmMainResult} entails the following steps. 
\begin{itemize}
\item[{\bf Step 1.}] We show that for a universal constant $k_\mathrm{max} =k_\mathrm{max}(G,c)$ depending only on the local structure of the graph $G$ and the constant $c$ from \eqref{eqn:pcn},
\be
\tmix = O\Big(\sup_{1 \leq k \leq k_\mathrm{max}} \tmix^{(k)} + \sup_{1 \leq k \leq k_\mathrm{max}} \O_k \big(\frac{1}{8 k_\mathrm{max}},C \tmix^{(k)} \big) \Big).
\ee
This is an immediate consequence of Lemma \ref{LemmaBasicMixing}.

\item[{\bf Step 2.}] By a comparison argument using the simple exclusion process, we show that $ \tmix^{(k)} = O((\alpha^{\mathrm{SE}}_{n,k})^{-1} n\log n )$ uniformly in $1 \leq k \leq k_\mathrm{max}$, where $\alpha^{\mathrm{SE}}_{n,k}$ is the log-Sobolev constant of the simple exclusion process on $G$  with $k$ particles. For $G= \LL$, it is known that $(\alpha^{\mathrm{SE}}_{n,k})^{-1} = O\big( n^{1 + \frac{2}{d}} \big)$ (see \cite{Yau97}) and thus  $ \tmix^{(k)} = O(n^{2 + {2 \over d}} \log n )$ uniformly in $1 \leq k \leq k_\mathrm{max}$ for the KCIP on $\LL$. See Lemma \ref{CorMixingTimeRestWalk}. 

\item[{\bf Step 3.}] For $G = \LL$, by coupling the KCIP to a `colored' version of the coalesence process over short time periods, we show that the process
\be \label{EqDefNumVerts}
V_t = \sum_{v \in V} X_t[v]
\ee 
satisfies the `drift condition'
\be
\E[V_{t + \epsilon n^3} - V_t | X_{t}] \leq -\delta V_t + C 
\ee
for some $C< \infty$ and $\epsilon, \delta > 0$, all independent of $n$. See 
Theorem \ref{LemmaContractionEstimate}.

\item[{\bf Step 4.}] For $G = \LL$, by direct calculation and further comparison to a coalesence process, we show that  
\be 
\inf_{x \in \Omega_k} \P[L_k(x) > C_1 n^3] > \epsilon > 0
\ee 
and  
\be
\sup_{x \in \Omega_k} \P[L_k(x) < C_2 n^3] > \epsilon > 0
\ee
 for some constants $C_1, \, C_2, \, \epsilon > 0$ independent of $n$. 
See Lemmas \ref{LemSoftCoalBound1} and \ref{LemmaIneqTransProbKNearK}.
\item[{\bf Step 5.}] Conclude from ${\bf Step\, 3}$ and ${\bf Step\, 4}$ that $\sup_{1 \leq k \leq k_\mathrm{max}} \O_k (\frac{1}{8 k_\mathrm{max}}, C \tmix^{(k)}) = O(n^3 \log(n))$. See Lemma \ref{LemmaOccMeasureBound}. 
\end{itemize}

The lower bound is obtained in Theorem \ref{ThmMainResult} by a direct computation of the expected time for the KCIP $\{X_t\}$ on $\LL$ to first have two non-adjacent particles when started from a single particle. Some of the steps above that are specific to the torus $\LL$ can be extended for other graphs as well, as discussed in  
Section \ref{SecConclusion}. 
\section{General Mixing Bounds for Decomposable Markov Chains} \label{SecGenMix}
In this section, we give a general bound on the mixing time of decomposable Markov chains. This result will be later applied to the KCIP Markov chain to achieve {\bf Step 1} in the proof of Theorem \ref{ThmMainResult}. The bounds in this section apply to Markov chains other than the KCIP and thus are of independent interest; they are developed further and applied to other examples in the forthcoming note \cite{PiSm15}. \par

Consider an aperiodic, irreducible, reversible and $\frac{1}{2}$-lazy Markov chain $\{ Z_{t} \}_{t \in \mathbb{N}}$ with transition matrix $K$ and stationary distribution $\tpi$ on a finite state space $\Theta$. Our goal is to bound the mixing time of $\{ Z_{t} \}_{t \in \mathbb{N}}$ in terms of the mixing times of various \textit{restricted} and \textit{projected} chains; these $L^{1}$ mixing bounds are loosely analogous to the $L^{2}$ bounds in \cite{JSTV04}. We begin by fixing $n$ and writing the partition $\Theta = \sqcup_{k=1}^{n} \Theta_{k}$. For $1 \leq k \leq n$, set $\eta_{k}(0) = 0$ and for $s \in \mathbb{N}$, recursively define the sequences of times 
\be \label{EqDefOccupationMeasureCounters}
\eta_{k}(s) &= \inf \{ t > \eta_{k}(s-1) \, : \, Z_{t} \in \Theta_{k} \}, \\
\kappa_{k}(s) &= \sup \{u \, : \, \eta_{k}(u) \leq s \}.
\ee 
The quantity $\kappa_k$ can also be written as
\be \label{eqn:kappak}
\kappa_{k}(T) = \sum_{t=1}^{T} \textbf{1}_{Z_{t} \in \Theta_{k}}.
\ee 
Both $\eta, \kappa$ depend on the initial condition $Z_1$.
We also define the associated \textit{restricted} processes $\{ Z_{t}^{(k)} \}_{t \in \mathbb{N}}$ by 
\be \label{EqDefRestChain}
Z_{t}^{(k)} = Z_{\eta_{k}(t)}. 
\ee 
This process is called the \textit{trace} of $\{ Z_{t} \}_{t \in \mathbb{N}}$ on $\Theta_{k}$. Since $\{ Z_{t} \}_{t \in \mathbb{N}}$ is recurrent, we have for all $t \in \mathbb{N}$ that $\eta_{k}(t) < \infty$ almost surely, and so $Z_{t}^{(k)}$ is almost surely well-defined for all $t \in \mathbb{N}$. The process $\{ Z_{t}^{(k)} \}_{t \in \mathbb{N}}$ is a Markov chain on $\Theta_{k}$, and denote by $K_{k}$ the associated transition kernel on $\Theta_{k}$. The kernel $K_{k}$ inherits aperiodicity, irreducibility, reversibility and $\frac{1}{2}$-laziness from $K$ and its stationary distribution is given by $\tpi_{k}(A) = \frac{ \tpi(A)}{\tpi(\Theta_{k})}$ for all $A \subset \Theta_{k}$. We denote by $\varphi_k$ the mixing time of the kernel $K_{k}$.  \par 
For $0 < a < \frac{1}{2}$, define the universal constants $c_a$ and $c_a'$ as in Theorem 1.1 of \cite{PeSo13}. We then have the following simple bound on the mixing time of $K$: 

\begin{lemma} [Basic Mixing Bound] \label{LemmaBasicMixing}
Fix $0 < a < \frac{1}{2}$ and $1 - a < \beta < 1$. Define $\gamma = \min \big({1 \over 2}, \frac{a + \beta - 1}{\beta} \big) > 0$ and fix a collection of indices $I \subset \{1, 2, \ldots, n \}$ satisfying
\be 
\sum_{k \in I} \tpi(\Theta_{k}) > \beta.
\ee 
Then the mixing time $\tmix^Z$ of $\{ Z_{t} \}_{t \in \mathbb{N}}$ satisfies
\be 
\frac{\tmix^Z}{4 c_a} \leq \mathcal{T} \equiv \inf \Big\{ T \, : \, \inf_{0 < t < T}\sup_{k \in I }  \Big( \frac{ c_{\gamma}'\varphi_k}{t} + \sup_{z \in \Theta} \P_{z}[\kappa_{k}(T) < t] \Big) < \frac{1}{4} \Big\}.
\ee 
\end{lemma}
Thus Lemma \ref{LemmaBasicMixing} relates the mixing time of a Markov chain to the mixing times of its traces on subsets of the state space with reasonably large stationary measure and the corresponding occupation times on those subsets. This is possible due to the remarkable results obtained in \cite{PeSo13,Oliv12d}, where the authors obtain a bound on the mixing times of \emph{reversible} Markov chains in terms of hitting times.
\begin{remark} Since our mixing bounds in Section \ref{SecCompToExc} below are obtained by a bound on the log-Sobolev constant of various restricted chains, and our bound on occupation measure can easily be translated into a bound on the spectral gap of an associated projected chain, the reader may ask why we do not use the well developed theory in \cite{JSTV04} for decomposable Markov chains to bound the mixing of a Markov chain in terms of restricted and projected chains. One reason is convenience: unlike the bounds in \cite{JSTV04}, Lemma \ref{LemmaBasicMixing} requires only a bound on $\varphi_{k}$ for \textit{some} $k$, not \textit{all} $k$. Since the bounds in Section \ref{SecCompToExc} apply only for $k$ small, Lemma \ref{LemmaBasicMixing} allows us to avoid doing the substantial extra work of finding bounds on $\varphi_{k}$ for $k$ large. The second, and more important, reason is that it is impossible to get an upper bound on the mixing time that is smaller than $O(n^{4} \log(n))$ using bounds from \cite{JSTV04} and any partition of $\Omega$ similar to the partition that we use. The extra factor of $n$ comes primarily from the fact that the probability of moving from $\Omega_{k+1}$ to $\Omega_{k}$ within $O(n^{2})$ steps is very far from uniform over starting points $x \in \Omega_{k+1}$.
\end{remark}
\section{Lower bound for the mixing time of KCIP on $\LL$}\label{sec:lowerbd}
%
In this section, we give a direct computation that leads to a lower bound on the mixing time of the KCIP on any $m$-regular graph $G = (V,E)$ that contains no triangles. The torus $\LL$ is $2d$-regular and triangle-free, and this bound will immediately give the lower bound in Theorem \ref{ThmMainResult}. Before proceeding, we verify that the stationary distribution $\pi$ is given by formula \eqref{eqn:pistat}. If $\P[X_{t+1} = y \vert X_{t} = x] > 0$, then either
\be \label{EqStatDistCheck1}
\P[X_{t+1} = y \vert X_{t} = x] &= \frac{p}{1-p} \P[X_{t+1} = x \vert X_{t} = y]\\ 
&= \frac{\pi(y)}{\pi(x)} \P[X_{t+1} = x \vert X_{t} = y]
\ee 
{or}
\be \label{EqStatDistCheck2}
\P[X_{t+1} = y \vert X_{t} = x] &= \frac{1-p}{p} \P[X_{t+1} = x \vert X_{t} = x]\\
&= \frac{\pi(y)}{\pi(x)} \P[X_{t+1} = x \vert X_{t} = y],
\ee 
depending on whether $\sum_{v \in G} x[v] < \sum_{v \in G} y[v]$ or not. Thus the Markov chain $\{X_t\}_{{t \in \mathbb{N}}}$ satisfies the detailed balance equation with respect to $\pi$ and thus has $\pi$ as its stationary distribution. \par
We continue by setting notation that will be used throughout the remainder of the paper. 
Let $G_{t}$ be the subgraph of $G$ induced by the vertices $\{v \in G \, : \, X_{t}[v] = 1 \}$, with vertices $V(G_{t})$ and edges $E(G_{t})$. For a vertex $u \in \LL$, define $\mathrm{Comp}_{t}(u)$ to be the collection of vertices contained in the same connected component as $u$ in $G_{t}$, define $\mathrm{Comp}_{t}$ to be the collection of distinct connected components in $G_{t}$, and let 
\be
Y_{t} = \vert \mathrm{Comp}_{t} \vert
\ee
 be the number of connected components in $G_{t}$. Recall that $V_{t} = \sum_{v \in G} X_{t}[v] = \vert G_{t} \vert$ is the number of vertices having state 1. For a KCIP started at time 0 with $V_{0} = 1$, define the associated \textit{triple} time by 
\be \label{EqTripleTime} 
\ztr &= \inf \{ t \, : \, V_{t} \geq 3 \}.
\ee  

\begin{lemma} [Component Growth] \label{LemmaNateshLowerBoundOrig}
Let $G$ be an $m$-regular graph ($m>1$) with no triangles. Fix $ \epsilon > 0$ and assume that $V_{0} = 1$. Then
\be \label{IneqCompGrowthLowerBound}
\P\big[\ztr  < \epsilon \frac{n^{3}}{3 c^{2} m (m-1)} \big] = O \big(\epsilon \big),
\ee 
where the implied constant is uniform over $0 < \epsilon < \epsilon_{0}$ sufficiently small and does not depend on $G$ or $c$.
\end{lemma}
\begin{proof}
Define the matrix 
\be \label{EqSmallMat}
K = 
\begin{pmatrix}
1 - \frac{cm}{n^{2}} & \frac{cm}{n^{2}} & 0\\
\frac{2}{n}(1 - \frac{c}{n}) & 1 - \frac{c(2m-1)}{n^{2}} - \frac{2}{n}(1 - \frac{c}{n}) & \frac{c(2m-1)}{n^{2}} \\
0 & 0 & 1
\end{pmatrix}.
\ee 
It is straightforward to check that, for $0 \leq s < \ztr$ and $a \in \{1,2,3\}$,
\be  \label{EqSmallMatRel}
\P[V_{s+1} = a | X_{s}] = K[V_{s},a].
\ee 
Most significantly, $X_{s}$ appears on the right-hand side only through $V_{s}$.
We justify this by considering the various cases. If $V_{s} = 1$, the transition probabilities $\P[V_{s+1} = i|X_s]$ depend only on the number of vertices labelled 1 at time $s$ (this is 1) and the number of vertices adjacent to this vertex (since our graph is $m$-regular, this is $m$). If $V_{s} = 2$, the transition probabilities $\P[V_{s+1} = i|X_s]$ depend only on the number of vertices labelled 1 at time $s$ (this is 2), the number of vertices labelled 1 adjacent to other vertices labelled 1 (this is also 2) and the number of vertices labelled 0 adjacent to vertices labelled 1 (since our graph is both $m$-regular and triangle-free, this is $2m-1$).\footnote{The transition probability $\P[V_{s+1} = i|V_s = 3]$ is independent of $X_{s}$ only if $G$ contains no 4-cycles; this is of course not the case for our graphs of interest.}\par 
By direct computation, 
\be 
\E_{1}[\ztr] &= \big( \frac{n^{2}}{cm} - 1 \big) + \E_{2}[\ztr] \\
\big( \frac{2}{n} \big( 1 - \frac{c}{n} \big) + \frac{c (2m-1)}{n^{2}} \big) \E_{2}[\ztr] &= 1 + \frac{2}{n} \big( 1 - \frac{c}{n} \big) \E_{1}[\ztr],
\ee 
and so we obtain
\be \label{AsympHittingTime}
\E_{1}[\ztr] = \frac{n^{3}}{c^{2} m (m-1)} + O \big( n^{2} \big).
\ee 
We recall from \cite{KaMc59} (and in particular the rewriting and generalization of that theory in \cite{Micl10}) that the distribution of $\ztr$ is given by a weighted sum of two geometric distributions. Thus, denoting $\mathrm{Geom}(q)$ for the geometric distribution with mean $q$, we have
\be \label{EqMiclRep1}
\ztr \stackrel{D}{=} a_{1} \mathrm{Geom}(\theta_{1}) + a_{2} \mathrm{Geom}(\theta_{2}),
\ee 
where by inequality \eqref{AsympHittingTime}, 
\be 
a_{1} \theta_{1} + a_{2} \theta_{2} = \frac{n^{3}}{c^{2} m (m-1)} + O \big( n^{2} \big).
\ee 
This implies that, for some $j \in \{1,2\}$ and all $n$ sufficiently large, $a_{j} \theta_{j} \geq \frac{n^{3}}{3 c^{2} m (m - 1)}$. Assume without loss of generality that $j = 1$. We then have
\be 
\P \big[a_{1} \theta_{1} < \epsilon \frac{n^{3}}{3 c^{2} m (m-1)} \big] = O \big( \epsilon \big).
\ee 
Combining this with formula \eqref{EqMiclRep1} completes the proof.
\end{proof}

We now conclude with the lower bound in Theorem \ref{ThmMainResult}:

\begin{thm} [Lower Bound on Mixing Time] \label{ThmLowerBoundOnMixingTime}
Fix an $m$-regular $(m>1)$ graph $G$ with no triangles. Then the KCIP on $G$ with success probability $p = \frac{c}{n}$ has mixing time satisfying
\be 
\tmix \geq C {1 \over \mathcal{Z}_{c}} \frac{n^3}{m(m-1)}
\ee 
where $C$ does not depend on $c$ or $G$ and
\be 
\mathcal{Z}_{c} = \frac{c^{2}  \max(1, -\log_{2}(\big( c + \frac{c^{2}}{2} \big) \frac{e^{-c}}{1 - e^{-c}}))}{ \big( c + \frac{c^{2}}{2} \big) \frac{e^{-c}}{1 - e^{-c}} } > 0.
\ee 
\end{thm}
\begin{proof}
Fix $\epsilon > 0$, let $T = \lfloor \epsilon \frac{n^{3}}{c^{2} m(m-1)} \rfloor$ and define the set $A = \{x \in \Omega \, : \, \sum_{v \in G} x[v] \leq 2 \}$. 
Let $X_0$ be such that $V_0 =1$. From Lemma \ref{LemmaNateshLowerBoundOrig} and Equation \ref{eqn:pistat} for the stationary
distribution $\pi$, we calculate:
\be 
\| \mathcal{L}(X_{T}) - \pi \|_{\TV} &\geq \pi(A^{c}) - \P[X_{T} \in A^{c}] \\
&\geq \pi(A^{c}) -  \P \big[\tau_{\mathrm{triple}} < T \big] \\
&=  1 - n \frac{c}{n} (1 - \frac{c}{n})^{n-1} \frac{1}{1 - (1-\frac{c}{n})^{n}}- \frac{n(n-1)}{2} \frac{c^{2}}{n^{2}} (1 - \frac{c}{n})^{n-2}\frac{1}{1 - (1-\frac{c}{n})^{n}} + O(\epsilon) \\
&= \big( c + \frac{c^{2}}{2} \big) \frac{e^{-c}}{1 - e^{-c}}(1 + o(1)) + O(\epsilon).
\ee 
Thus, for $\epsilon \ll  \frac{1}{4} \big( c + \frac{c^{2}}{2} \big) \frac{e^{-c}}{1 - e^{-c}} $ sufficiently small, 
\be
\| \mathcal{L}(X_{T}) - \pi \|_{\TV} \geq   \frac{1}{2} \big( c + \frac{c^{2}}{2} \big) \frac{e^{-c}}{1 - e^{-c}} > 0
\ee uniformly in $n > N(\epsilon)$ sufficiently large. Since the mixing profile satisfies 
\be 
\tau \Big(\frac{1}{2} \big( c + \frac{c^{2}}{2} \big) \frac{e^{-c}}{1 - e^{-c}}\Big) \leq \max(1, -\log_{2}(2\big( c + \frac{c^{2}}{2} \big) \frac{e^{-c}}{1 - e^{-c}})) \tau_{\mathrm{mix}},
\ee 
by Lemmas 4.11 and 4.12 of \cite{LPW09}, this implies
\be 
\tau_{\mathrm{mix}} \geq \frac{T}{ \max(1, -\log_{2}(2\big( c + \frac{c^{2}}{2} \big) \frac{e^{-c}}{1 - e^{-c}}))}
\ee 
for $\epsilon$ sufficiently small. This completes the proof.
\end{proof}

\section{Mixing times of the trace of KCIP on $\Omega_k$} \label{SecCompToExc}

In this section, we obtain bounds on the mixing times of the trace of the KCIP on $\Omega_{k}$. As these trace walks mix substantially more quickly than the KCIP Markov chain $X_t$ on $\Omega$, these mixing time bounds do not need to be tight for arguments to go through. \par
Fix $1 \leq k \leq \frac{n}{2}$, and let $Q_{n,k}$ be the kernel of the trace of $\{X_{t}\}_{t \in \mathbb{N}}$ on $\Omega_{k}$ (recall that the trace of a Markov chain on a subset of its state space is defined in formula \eqref{EqDefRestChain}). Denote by $\tau_{n,k}$ the mixing time of $Q_{n,k}$. The key result of this section is:

\begin{lemma} [Mixing Time of Restricted Walks] \label{CorMixingTimeRestWalk}
With the notation as above, 
\be 
\tau_{n,k} \leq  C n^{2 + \frac{2}{d} } \log(n)
\ee 
for some constant $C = C(c,k,d)$ that does not depend on $n$.
\end{lemma}

We proceed by using comparison theory developed for proving useful mixing bounds for a Markov chain by comparing its transition rates to a similar and better-understood chain (see, \textit{e.g.}, \cite{DiSa93b} or \cite{DGJM06} for an introduction to this method). We will use the simple exclusion process (SE) on  $\LL$ as the basis of our comparison argument (see \cite{ClSu73, HoLi75} for an introduction to the simple exclusion process).
\begin{defn} [Simple Exclusion Process on $\LL$] \label{DefSimpleExclusion}
The simple exclusion process $\{ Z_{t} \}_{t \in \mathbb{N}}$ is a Markov chain on the finite state space
\be \label{EqDefOmegaTildeK}
\Omega^{\mathrm{SE}}_{n,k} \equiv \{ Z \in \{ 0, 1 \}^{n} \, : \, \sum_{i} Z[i] = k \}.
\ee
To update $Z_{t}$, choose two adjacent vertices $u_{t}, v_{t} \in \Lambda(L,d)$ uniformly at random and set
\be 
Z_{t+1}[u_{t}] &= Z_{t}[v_{t}], \\
Z_{t+1}[v_{t}] &= Z_{t}[u_{t}] 
\ee 
and $Z_{t+1}[w] = Z_{t}[w]$ for $w \notin \{ u_{t}, v_{t} \}$. 
\end{defn}

The approach in this section is to first note that the simple exclusion process with $k$ particles has good mixing properties (we use the results in \cite{Yau97}, though others would suffice for our purposes) and then use a comparison argument to show that the mixing properties of the trace of the KCIP on $\Omega_{k}$ cannot be much worse. We use the simple exclusion process because it makes both parts of this argument easy: it has already been carefully analyzed, and it is similar enough to the trace of the KCIP that the comparison argument is short and involves only soft arguments. It is likely that the conclusions we need can be achieved by comparison to other processes on the torus.
\subsection{Comparison of Markov chains using Dirichlet forms}
Before stating the main result of this section carefully, for the reader's convenience, we recall some relevant results for comparing Dirichlet forms.  We use the bounds in \cite{Smit14a}, rather than the similar and simpler results from \cite{DiSa93b,DGJM06}, because we will compare chains (KCIP and SE) with different state spaces; the bounds in \cite{DiSa93b,DGJM06} cannot be used in this situation.

\begin{defn} [Norms, Forms and Related Functions]
For a general Markov chain on a finite state space $X$ with kernel $P$ and unique stationary distribution $\pi$, and any function $f \, : \, X \rightarrow \mathbb{R}$ that is not identically 0, we respectively define  the $L_2$ norm, variance, Dirichlet form and entropy form as:
\be \label{FunctionalDefs}
\| f \|_{2, \pi}^{2} &= \sum_{x \in X} \vert f(x) \vert^{2} \pi(x),\\
V_{\pi}(f) &= \frac{1}{2} \sum_{x,y \in X} \vert f(x) - f(y) \vert^{2} \pi(x) \pi(y), \\
\mathcal{E}_{P}(f,f) &= \frac{1}{2} \sum_{x,y \in X} \vert f(x) - f(y) \vert^{2} P(x,y) \pi(x), \\
L_{\pi}(f) &= \sum_{x \in X} \vert f(x) \vert^{2} \log \big( \frac{f(x)^{2}}{\| f \|_{2,\pi}^{2}} \big) \pi(x). 
\ee
Recall that the \textit{log-Sobolev constant} of a Markov transition matrix $P$ is given by
\be \label{EqVarCharAlpha}
\alpha(P) = \inf_{f \neq 0} \frac{\mathcal{E}_{P}(f,f)}{L_{\pi}(f)}.
\ee
\end{defn} 

\begin{defn}[Extensions] \label{DefFlowDistPath}
Let $K,Q$ be the kernels of two $\frac{1}{2}$-lazy, aperiodic, irreducible, reversible Markov chains. Assume that $K$ has stationary measure $\mu$ on a state space $\hTh$ while $Q$ has stationary measure $\nu$ on a state space $\Th \subset \hTh$. Denote by $f$ a function on $\Th$, and call a function $\widehat{f}$ on $\hTh$ an \emph{extension} of $f$ if  $\widehat{f}(x) = f(x)$ for all $x \in \Th$. 
\end{defn}
Next, fix a family of probability measures $\{ \P_{x}[y] \}_{x \in \hTh}$ on $\Th$ that satisfy $\P_{x}(\cdot) = \delta_{x}(\cdot)$ for $x \in \Th$. We will use only extensions of the form
\be \label{EqLinearFamilyExtensions}
\widehat{f}(x) = \sum_{y \in {\Th}} \P_{x}[y] f(y).
\ee
We call extensions of the form \eqref{EqLinearFamilyExtensions} \textit{linear extensions}. \par
Fix a linear extension. For each pair $(x,y) \in \hTh$ with $K(x,y) > 0$, fix a joint probability distribution $\P_{x,y}[a,b]$ on $\Th \times \Th$ satisfying $\sum_{a} \P_{x,y}[a,b] = \P_{y}[b]$ for all $b \in \Th$ and $\sum_{b} \P_{x,y}[a,b] = \P_{x}[a]$ for all $a \in \Th$. This is a coupling of the distributions $\P_{x}, \P_{y}$. \par 
\begin{defn}[Paths, Flows]
Finally, for each $a,b \in \Th$ with $\sum_{x,y \in \hTh} \P_{x,y}[a,b] > 0$, we define a \emph{flow} in $\Th$ from $a$ to $b$. To do so, call a sequence of vertices $\gamma = [ a = v_{0,a,b}, v_{1,a,b}, \ldots, v_{k[\gamma], a,b} = b ]$ a \emph{path} from $a$ to $b$ if $Q(v_{i,a,b}, v_{i+1,a,b}) > 0$ for all $0 \leq i < k[\gamma]$. Then let $\Gamma_{a,b}$ be the collection of all paths from $a$ to $b$. Call a function $F$ from paths to $[0,1]$ a \emph{flow} if $\sum_{\gamma \in \Gamma_{a,b}} F[\gamma] = 1$ for all $a,b$. For a path $\gamma \in \Gamma_{a,b}$, we will label its initial and final vertices by $i(\gamma) = a$, $o(\gamma) = b$. 
\end{defn}

The purpose of these definitions is to provide a way to compare the functionals described in formula \eqref{FunctionalDefs}. If there exists a family of measures $\{ \P_{x} \}_{x \in \hTh}$ so that the associated linear extensions given by formula \eqref{EqLinearFamilyExtensions} satisfy
\be
L_{\nu}(f) &\leq C_{L}\, L_{\mu}(\widehat{f}), \\
\mathcal{E}_{K}(\widehat{f}, \widehat{f}) &\leq C_{\mathcal{E}}\, \mathcal{E}_{Q}(f,f),
\ee
\noindent then the variational characterization of $\alpha$ given in formula \eqref{EqVarCharAlpha} implies
\be \label{IneqContentlessLogSobCompBound} 
\alpha(Q) &\geq \frac{1}{C_{L} C_{\mathcal{E}}} \alpha(K).
\ee

This is the motivation for Theorem 4 and Lemma 2 of \cite{Smit14a}. 
Theorem 4  of \cite{Smit14a} may be restated as:

\begin{thm}[Comparison of Dirichlet Forms for General Chains] \label{ThmDirGenChain}
Let $K,Q$ be the kernels of two  reversible Markov chains. Assume that $K$ has stationary measure $\mu$ on state space $\hTh$ while $Q$ has stationary measure $\nu$ on state space $\Th \subset \hTh$. Fix flow $F$, distributions $\P_{x}$ and couplings $\P_{x,y}$ as in the notation in Definition \ref{DefFlowDistPath} above. Then for any function $f$ on $\Th$ and the linear extension $\hat{f}$ of $f$ on $\hTh$ given by formula \eqref{EqLinearFamilyExtensions},
\begin{equation*}
\mathcal{E}_{K}(\widehat{f}, \widehat{f}) \leq \mathcal{A} \mathcal{E}_{Q}(f,f),
\end{equation*}
where 
\begin{align*}
\mathcal{A} = \sup_{Q(q,r) >0}  \frac{1}{Q(q,r) \nu(q)} & \Big( \sum_{\gamma \ni (q,r)} F[\gamma] k[\gamma] K(i(\gamma),o(\gamma)) \mu(i(\gamma)) \\
&\hspace{1cm}+ 2  \sum_{\gamma \ni (q,r)} k[\gamma] F[\gamma] \sum_{y \in \hTh \backslash \Th} \P_{y}[o(\gamma)] K(i(\gamma),y) \mu(i(\gamma)) \\
&\hspace{1cm}+ \sum_{\gamma \ni (q,r)}  k[\gamma] F[\gamma] \sum_{x,y \in \hTh \backslash \Th \, : \, K(x,y) > 0} \P_{x,y}[i(\gamma),o(\gamma)] K(x,y) \mu(x) \Big).
\end{align*}
\end{thm}

Lemma 2 of \cite{Smit14a} may be restated as:

\begin{lemma} [Comparison of Variance and Log-Sobolev Constants] \label{LemmaVarLogSobComp}
Let $\mu$ be a measure on $\hTh$ and $\nu$ be a measure on $\Th \subset \hTh$. Let $\tilde{C} = \sup_{y \in \Omega} \frac{\nu(y)}{\mu(y)}$. Then for any function $f$ on $\Th$ and linear extension $\hat{f}$ of $f$ on $\hTh$,
\be
V_{\nu}(f) &\leq \tilde{C} V_{\mu}(\widehat{f}), \\
L_{\nu}(f) &\leq \tilde{C} L_{\mu}(\widehat{f}). 
\ee
\end{lemma}

\subsection{Bounds on KCIP}

Next we prove our results. Fix $1 \leq k \leq {n \over 2}$ and denote by $\alpha_{n,k}$ the log-Sobolev constant of $Q_{n,k}$. Denote by $Q^{\mathrm{SE}}_{n,k}$ and $\alpha^{\mathrm{SE}}_{n,k}$ the kernel and log-Sobolev constant associated with the simple exclusion process with $k$ particles on $\Lambda(L,d)$. We then have:

\begin{lemma} [Comparison of Log-Sobolev Constants] \label{LemmaCompLogSob}
There exist $N = N(c,k,d) < \infty$ and $C = C(c,k,d) < \infty$ so that $n > N$ implies
\be \label{IneqCompLogSobConcMain}
\alpha_{n,k} \geq  C \frac{1}{n} \alpha^{\mathrm{SE}}_{n,k}.
\ee 
\end{lemma}

\begin{proof}
Our proof consists of comparing a sequence of very similar Markov chains, beginning with the trace of the KCIP and ending with the simple exclusion process. The bulk of our argument goes through Theorem \ref{ThmDirGenChain}, Lemma \ref{LemmaVarLogSobComp} and inequality \eqref{IneqContentlessLogSobCompBound}. \par 
Recall that the state space of $Q_{n,k}$ is $\Omega_{n,k} = \Omega_{k}$ as defined in formula \eqref{EqDefOmegaK}, while the state space of $Q^{\mathrm{SE}}_{n,k}$ is $\Omega^{\mathrm{SE}}_{n,k}$ as defined in formula \eqref{EqDefOmegaTildeK}. 
By the standard `birthday problem' computation, 
\be 
\prod_{i=1}^{k} (1 - (i-1) \frac{(2d+1)}{n}) \leq \frac{\vert \Omega_{n,k} \vert}{\vert \Omega^{\mathrm{SE}}_{n,k} \vert }  \leq 1.
\ee 
Thus, for any fixed $k \in \mathbb{N}$,
\be \label{IneqStatMeas}
1 - o(1) \leq \frac{\vert \Omega_{n,k} \vert}{\vert \Omega^{\mathrm{SE}}_{n,k} \vert } \leq 1
\ee 
as $n$ goes to infinity. Since the stationary distributions $\pi_{n,k}$ and $\pi^{\mathrm{SE}}_{n,k}$ of $Q_{n,k}$ and $Q^{\mathrm{SE}}_{n,k}$ respectively are uniform on $\Omega_{n,k}$ and $\Omega^{\mathrm{SE}}_{n,k}$ respectively, inequality \eqref{IneqStatMeas} implies that
\be \label{IneqBasicStationaryMeasureBound} 
1 < \frac{\pi_{n,k}(x)}{\pi^{\mathrm{SE}}_{n,k}(x)} \leq 1 + o(1)
\ee 
uniformly in $x \in \Omega_{n,k} \subset \Omega^{\mathrm{SE}}_{n,k}$. \par
Next, we define a less-lazy version of $Q_{n,k}$. Note that
\be
Q_{n,k}(x,x) \geq 1 - \frac{2cdk}{n^{2}}, 
\ee
and so for $n$ sufficiently large we can define a less-lazy version $Q^{\mathrm{NL}}_{n,k}$ of $Q_{n,k}$ by
\be 
Q_{n,k} = \big( 1 - \frac{c}{n} \big) \mathrm{Id} + \frac{c}{n} Q^{\mathrm{NL}}_{n,k},
\ee 
where $\mathrm{Id}$ is the identity kernel and $ Q^{\mathrm{NL}}_{n,k}$ is at least $\frac{1}{2}$-lazy itself.  Since $ Q^{\mathrm{NL}}_{n,k}$ is simply a less lazy version of $Q_{n,k}$, it is immediate that the associated Dirichlet forms $\mathcal{E}_{Q^{\mathrm{NL}}_{n,k}}$ and  $\mathcal{E}_{Q_{n,k}}$ satisfy
\be 
\mathcal{E}_{Q_{n,k}}(f,f) \geq \frac{c}{n} \mathcal{E}_{Q^{\mathrm{NL}}_{n,k}}(f,f) 
\ee 
for all $f \, : \, \Omega_{n,k} \rightarrow \mathbb{R}$  with $f \neq 0$. Applying Lemma \ref{LemmaVarLogSobComp} to inequality \eqref{IneqContentlessLogSobCompBound}, this bound implies the following inequality for the log-Sobolev constant $\alpha^{\mathrm{NL}}_{n,k}$ of $Q^{\mathrm{NL}}_{n,k}$:
\be \label{IneqLogSobSimp1}
\alpha_{n,k} \geq \frac{c}{n} \alpha^{\mathrm{NL}}_{n,k}.
\ee 
Thus, to prove inequality \eqref{IneqCompLogSobConcMain} it is sufficient to relate $Q^{\mathrm{NL}}_{n,k}$ and $Q^{\mathrm{SE}}_{n,k}$. \par
 Next, we define $Q^{\mathrm{MH-SE}}_{n,k}$ to be the Metropolis-Hasting chain with proposal chain equal to $Q^{\mathrm{SE}}_{n,k}$ and target measure $\pi_{n,k} = \mathrm{Unif}(\Omega_{n,k})$. If $x \neq y \in \Omega_{n,k}$ satisfy $Q^{\mathrm{MH-SE}}_{n,k}(x,y) > 0$, then $x$ and $y$ differ at exactly two vertices $u,v$, with $x[v] = y[u] = 0$, $x[u] = y[v] = 1$. Let $\phi_{0} = 0$ and inductively let $\phi_{i+1} = \inf\{t > \phi_{i} \, : \, X_{t} \neq X_{\phi_{i}} \}$ be the successive times at which the KCIP changes; we calculate 
\be \label{IneqLazyBasicComp}
Q^{\mathrm{NL}}_{n,k}(x,y) &\geq \frac{n}{c} Q_{n,k}(x,y) \\
&\geq \frac{n}{c} \P[X_{\phi_{1}}[v] = 1 | X_{0} = x] \P[X_{\phi_{2}}[u] = 0 | X_{0} = x, X_{\phi_{1}}[v] = 1] \\
&\geq \frac{n}{c} \frac{c}{n^{2}} \frac{\frac{1}{n}(1 - \frac{c}{n})}{\frac{1}{n}(1 - \frac{c}{n}) + \frac{2cd(k+1)}{n^{2}}} \\
&= \frac{1}{n} \big( \frac{n-c}{n - c + 2cd(k+1)} \big) \\
&\geq \frac{1}{2} \big( \frac{n-c}{n - c + 2cd(k+1)} \big) Q^{\mathrm{MH-SE}}_{n,k}(x,y).
\ee 
To see the second and third lines, consider starting at KCIP at $X_{0}$ and calculating the probability that $X_{\phi_{1}}$ is obtained from $X_{0}$ by changing the label of $v$ from 0 to 1, and that $X_{\phi_{2}}$ is obtained from $X_{\phi_{1}}$ by changing the label of $u$ from 1 to 0; these probabilities are at least $\frac{c}{n^{2}}$ and $\frac{\frac{1}{n}(1 - \frac{c}{n})}{\frac{1}{n}(1 - \frac{c}{n}) + \frac{2cd(k+1)}{n^{2}}}$ respectively. By the same short argument immediately preceding inequality \eqref{IneqLogSobSimp1}, inequality \eqref{IneqLazyBasicComp} implies that for all $n$ sufficiently large the log-Sobolev constant $\alpha^{\mathrm{MH-SE}}_{n,k}$ of $Q^{\mathrm{MH-SE}}_{n,k}$ satisfies
\be \label{IneqLogSobSimp2}
\alpha^{\mathrm{NL}}_{n,k} \geq \frac{1}{4} \alpha^{\mathrm{MH-SE}}_{n,k}.
\ee 
In light of this inequality and inequality \eqref{IneqLogSobSimp1}, to prove inequality  \eqref{IneqCompLogSobConcMain} it is enough to relate $Q^{\mathrm{MH-SE}}_{n,k}$ and $Q^{\mathrm{SE}}_{n,k}$. \par

We now give the main comparison argument, using Theorem \ref{ThmDirGenChain}. Using the notation of Theorem \ref{ThmDirGenChain}, we will compare kernels $Q = Q^{\mathrm{MH-SE}}_{n,k}$ and $K = Q^{\mathrm{SE}}_{n,k}$ on state spaces $\Theta = \Omega_{n,k}$ and $\widehat{\Theta} = \Omega^{\mathrm{SE}}_{n,k}$. Both of these kernels have stationary distributions that are uniform on their respective state spaces. To define the flows, distributions and couplings required by Theorem \ref{ThmDirGenChain}, we need slightly more notation. Define the graphs $G^{\mathrm{MH-SE}}_{n,k}$  and $G^{\mathrm{SE}}_{n,k}$ to have vertices 
\be 
V(G^{\mathrm{MH-SE}}_{n,k}) &= \Omega_{n,k}, \\
V(G^{\mathrm{SE}}_{n,k}) &= \Omega^{\mathrm{SE}}_{n,k} 
\ee 
and edges 
\be 
E(G^{\mathrm{MH-SE}}_{n,k}) &= \{ (x,y) \in G^{\mathrm{MH-SE}}_{n,k} \, : \, Q^{\mathrm{MH-SE}}_{n,k}(x,y) > 0 \}, \\
E(G^{\mathrm{SE}}_{n,k}) &= \{ (x,y) \in G^{\mathrm{SE}}_{n,k} \, : \, Q^{\mathrm{SE}}_{n,k}(x,y) > 0 \}. \\
\ee 
We denote by $d^{\mathrm{MH-SE}}$ and $d^{\mathrm{SE}}$ the usual graph distances on $G^{\mathrm{MH-SE}}_{n,k}$ and $G^{\mathrm{SE}}_{n,k}$ respectively. Next, define the distributions $\{ \P_{x} \}_{x \in \Omega^{\mathrm{SE}}_{n,k}}$ on $\Omega_{n,k}$ by $\P_{x}[\cdot] = \delta_{x}(\cdot)$ if $x \in \Omega_{n,k}$, and
\be 
\P_{x}[\cdot] = \mathrm{Unif} ( \{ y \in \Omega_{n,k} \, : \, d^{\mathrm{SE}}(x,y) = \min_{z \in \Omega_{n,k}} d^{\mathrm{SE}}(x,z) \} )
\ee 
otherwise. We define the couplings $\{\P_{x,y} \}_{x,y \in \Omega^{\mathrm{SE}}_{n,k}, \, Q_{n,k}^{\mathrm{SE}}(x,y) > 0}$ to be the independent couplings $\P_{x,y}[a,b] = \P_{x}[a] \P_{y}[b]$. Finally, for pairs $(a,b)$ satisfying $\sum_{x,y \in \Omega^{\mathrm{SE}}_{n,k}, \, Q_{n,k}^{\mathrm{SE}}(x,y) > 0} \P_{x,y}[a,b] > 0$, we define the flow $F$ on $\Gamma_{a,b}$ to be uniform on all minimum-length paths in $G^{\mathrm{MH-SE}}$ from $a$ to $b$. \par

We now show that the constant $\mathcal{A}$ that the above choices yield is uniformly bounded in $n$. Recall that $\frac{| \Theta |}{| \widehat{\Theta} |} = 1 + o(1)$ and $\frac{Q^{\mathrm{MH-SE}}(x_{1},y_{1})}{Q^{\mathrm{SE}}(x_{2},y_{2})}$ is either $0$ or $1$ when it is defined, and so the constant $\mathcal{A}$ can be bounded by 
\be
\mathcal{A} = (1 + o(1)) \sup_{Q(q,r) >0}   & \Big( \sum_{\gamma \ni (q,r)} F[\gamma] k[\gamma] + 2  \sum_{\gamma \ni (q,r)} k[\gamma] F[\gamma] \sum_{y \in \hTh \backslash \Th} \P_{y}[o(\gamma)]  \\
&\hspace{1cm}+ \sum_{\gamma \ni (q,r)}  k[\gamma] F[\gamma] \sum_{x,y \in \hTh \backslash \Th \, : \, K(x,y) > 0} \P_{x,y}[i(\gamma),o(\gamma)]  \Big).
\ee

To show that $\mathcal{A}$ is uniformly bounded in $n$, it is enough to check that all of the probabilities, couplings and paths that we have defined are \textit{local} in the sense that any particular probability, coupling or path involve only points that are a bounded distance from each other, uniformly in $n$. Since the details of the bounds are not important to us, we give very loose bounds. \par 
$\P_{x}$ is concentrated on points $y$ with $d^{\mathrm{SE}}(x,y) \leq 2k^{2}$. $\P_{x,y}$ is defined only on pairs $(x,y)$ satisfying $d^{\mathrm{SE}}(x,y) = 1$ and thus is supported on pairs $(a,b)$ satisfying  
\be \label{SimpleBoundOnSupportDist}
d^{\mathrm{SE}}(a,b) \leq 4 k^{2} + 1.
\ee 
For $x,y \in \Omega_{n,k}$, let 
\be
C_{x,y} = \max \{ | u - v | \, : u,v \in \LL, \, x[u] + y[u] \geq 1, \, x[v] + y[v] \geq 1\}
\ee
 be the maximum distance between any particles in $x$ or $y$. Set
 \be
 R_{C,n} = \max \{ d^{\mathrm{MH-SE}}(x,y) \, : \, C_{x,y} \leq C\}.
 \ee
 Since $G^{\mathrm{MH-SE}}$ is vertex-transitive and is connected for $n$ sufficiently large, we have for all $n > N(C,k,d)$ sufficiently large that
\be 
R_{C,n} \leq \mathcal{R}_C
\ee 
for some constant $\mathcal{R}_C$ that does not depend on $n$. Thus, for all $n > N(C,k,d)$ sufficiently large and all $x,y$ with $C_{x,y} \leq C$,
\be \label{InequalityPathLength1}
d^{\mathrm{MH-SE}}(x,y) \leq \mathcal{R}_C.
\ee 
If $d^{\mathrm{SE}}(x,y) \leq C$, then there exist $k$ vertices $v_{1},\ldots, v_{k} \in \LL$ that cover the particles of $x$ and $y$:
\be 
\{ v \in \LL \, : \, x[v] + y[v] \geq 1 \} \subset \cup_{i=1}^{k} \mathcal{B}_{2C}(v_{i}).
\ee
 By taking larger balls, this can be turned into a disjoint cover: there exist $1 \leq m \leq k$ vertices $u_{1},\ldots,u_{m}$ so that 
 \be
 \{ v \in \LL \, : \, x[v] + y[v] \geq 1 \} \subset \cup_{i=1}^{m} \mathcal{B}_{3k (C+ \mathcal{R}_C)}(u_{i})
 \ee
 with $\mathcal{B}_{3k (C+ \mathcal{R}_C)}(u_{i}) \cap  \mathcal{B}_{3k (C+ \mathcal{R}_C)}(u_{j}) = \emptyset$ for all $i \neq j$.  By the definition of $\mathcal{R}_C$, no minimal path from $x$ to $y$ can have any particles outside of the cover $\cup_{i=1}^{m} \mathcal{B}_{3k (C+ \mathcal{R}_C)}(u_{i})$; thus, by inequality \eqref{InequalityPathLength1}, for all $x,y \in \Omega_{n,k}$ with $d^{\mathrm{SE}}(x,y) \leq C$ we have that
\be 
d^{\mathrm{MH-SE}}(x,y) \leq R_{3k (C+ \mathcal{R}_C)}.
\ee 
Combining this bound with inequality \eqref{SimpleBoundOnSupportDist}, we conclude that all paths $\gamma$ with $F(\gamma) > 0$ have at most $\mathcal{R}_{3k (4k^{2} + 1 + \mathcal{R}_{4 k^{2} + 1})}$ points. Since balls of radius $\ell$ in the graph $\Lambda(L,d)$ have at most $(2d)^{\ell} +1$ vertices in them, balls of radius $\ell$ in $G^{\mathrm{MH-SE}}_{n,k}$ have at most $(2dk)^{\ell} + k$ vertices in them. Thus, at most $(2dk)^{\mathcal{R}_{3k (4k^{2} + 1 + \mathcal{R}_{4 k^{2} + 1})}} + k$ paths in $G^{\mathrm{MH-SE}}_{n,k}$ with positive support can pass through any given edge. Most importantly, all of these bounds depend only on $k$ and $d$; they are uniform in $n$ sufficiently large. Combining this observation with inequality \eqref{IneqBasicStationaryMeasureBound}  and applying Theorem \ref{ThmDirGenChain}, this implies  
\be \label{IneqCompMainConcProxy}
\mathcal{E}_{Q^{\mathrm{SE}}_{n,k}}(\widehat{f}, \widehat{f}) \leq \mathcal{A}_{k,d} \,\mathcal{E}_{Q^{\mathrm{MH-SE}}_{n,k}}(f,f)
\ee
for some constant $\mathcal{A}_{k,d}$ that depends only on $k$ and $d$. Thus, applying inequality \eqref{IneqStatMeas} and Lemma \ref{LemmaVarLogSobComp} to inequality \eqref{IneqContentlessLogSobCompBound}, we conclude that there exists some constant $C_{k,d}'$ so that
\be 
\alpha^{\mathrm{MH-SE}}_{n,k} \geq C_{k,d}' \alpha^{\mathrm{SE}}_{n,k}.
\ee  
Combining this bound with inequalities \eqref{IneqLogSobSimp1} and \eqref{IneqLogSobSimp2} completes the proof of inequality \eqref{IneqCompLogSobConcMain}.
\end{proof}

Finally, we prove Lemma \ref{CorMixingTimeRestWalk}:

\begin{proof}[Proof of Lemma \ref{CorMixingTimeRestWalk}] 
Translating the main result of \cite{Yau97} into our discrete-time setting, we have
\be 
\alpha^{\mathrm{SE}}_{n,k} \geq C_{d} n^{-1 - \frac{2}{d}},
\ee 
for some constant $C_{d}$ that depends only on $d$. By Lemma \ref{LemmaCompLogSob}, this implies
\be 
\alpha_{n,k} \geq C n^{-2 - \frac{2}{d}}
\ee 
for some constant $C = C(c,d,k)$ that doesn't depend on $n$. Applying inequality 3.3 of \cite{DiSa96c} yields the conclusion. 
\end{proof}

\section{Drift Condition for $V_t$} \label{SecDriftCond}

Recall the process $V_{t} = \sum_{v \in \LL} X_{t}[v]$ from formula \eqref{EqDefNumVerts}. The graph $G_{t}$  is the subgraph of $G=\LL$ induced by the vertices $\{v \in \LL \, : \, X_{t}[v] = 1 \}$, with vertices $V(G_{t})$ and edges $E(G_{t})$. Let $\mathcal{F}_{t}$ denote the $\sigma$-algebra generated by the random variables $\{X_{s} \}_{s \leq t}$.
The key result in this section is the following drift condition on $\{ V_{s} \}_{s \geq t}$: 
\begin{theorem} \label{LemmaContractionEstimate}
There exists some constant $0 < \epsilon_{0} = \epsilon_{0}(c,d)$ independent of $n$ so that for all $0 < \epsilon < \epsilon_{0}$, there exist constants $C_{G} = C(\epsilon,c,d) < \infty$, $\alpha = \alpha(\epsilon,c,d)$ and $N = N(\epsilon,c,d)$ so that, for all $k \in \mathbb{N}$ and all $n > N$,
\be \label{eqn:conddrift}
\E[V_{t + k \epsilon n^{3}} | \mathcal{F}_{t}] \leq (1-\alpha)^{k} V_{t} + C_{G}. 
\ee 
\end{theorem}

Besides a small number of definitions that are explicitly recalled, Theorem \ref{LemmaContractionEstimate} is the only part of Section \ref{SecDriftCond} used in the remaining sections. As the proof of Theorem \ref{LemmaContractionEstimate} is somewhat long, we give an outline:
\begin{enumerate}
\item We show that the number of particles $V_{t}$ is generally close to the number of connected components in $G_{t}$, and so it is enough to bound the latter quantity (see Lemma \ref{LemmCompPartNumComp}). 
\item By embedding a coalescence process into the KCIP, we show that the number of `collisions' between components of the KCIP over the time interval $\{t, t+1,\ldots, t + \epsilon n^{3} \}$ is on the order of the number of connected components in $G_{t}$ (see Lemma  \ref{LemmLbNumCol}). 
\item By direct computation, we show that collisions involving components of size 1 will decrease a certain biased count of the number of connected components $G_{t}$ (see Definition \ref{DefCorrCompCount} and Lemma \ref{LemCalc1}) while other collisions will not increase this observable by much (see Lemma \ref{LemmaCalc3}). 
\item By an argument based on bounding the influence of faraway points, we show that a substantial fraction of collisions occur between components of size 1 (see Lemma \ref{LemmaEL1}).
\item Steps (2)-(4) above will yield that a positive fraction of connected components of $G_{t}$ will be involved in a `good' collision over a reasonable time scale, and this leads to a contraction estimate on a biased count of the number of components (see Lemma \ref{LemmaCompNumCompsNumCols}). By the observation made in (1) above, this leads to our final contraction estimate on $V_{t}$.
\end{enumerate}

Recall from Section \ref{sec:lowerbd} that for a vertex $u$, $\mathrm{Comp}_{t}(u)$ is the collection of vertices contained in the same connected component as $u$ in $G_{t}$. Recall also that $Y_{t}$ denotes the number of connected components in $G_{t}$. Define the \textit{number of excess particles}
\be \label{eqn:deltavy}
\delta_{s} = V_{s} - Y_{s}.
\ee 
The next lemma compares the number of particles to the number of components.
\begin{lemma} \label{LemmCompPartNumComp}
For $n > 2c$ and for all $s,t \geq 0$ 
\be \label{IneqSimpCompNPartNComp}
\E[\delta_{t+s} | \mathcal{F}_{t} ] \leq \big(1 - \frac{1}{2n} \big)^{s} \delta_{t} + 4cd.
\ee 

\end{lemma}
\begin{proof}
We first show that
\be \label{IneqSimpCompNPartNComp0}
\E[\delta_{t+s} \mathbf{1}_{\sup_{t \leq i \leq t+s} V_{i} \leq k} | \mathcal{F}_{t}] \leq \Big(1 - \frac{1}{n} \big(1 - \frac{c}{n}\big) \Big)^{s} \delta_{t} + \frac{2cdk}{n} \frac{1}{1 - \frac{c}{n}}.
\ee 
Indeed, for $v \in \LL$, let $N^\mathrm{adj}_{s}(v)$ be the number of components of $G_{s}$ that are adjacent to $v$. Define
\be
A_{s} &= \{v \in \LL \, : \, X_{s}[v] = 1, \, |\mathrm{Comp}_{s}(v) | > 1 \}, \\
B_{s} &= \{v \in \LL \, : \, X_{s}[v] = 0, \, N^\mathrm{adj}_{s}(v)  > 0 \}, \\
D_{s} &= \LL \backslash (A_{s} \cup B_{s}).
\ee
Then we have
\be \label{IneqDiffDec}
\E[\delta_{s+1} | \mathcal{F}_{s}] &= \E[\delta_{s+1} | \mathcal{F}_{s}, v_{s} \in A_{s}] \P[v_{s} \in A_{s} | \mathcal{F}_{s}] \\
&\hspace{2cm}+ \E[\delta_{s+1} | \mathcal{F}_{s}, v_{s} \in B_{s}] \P[v_{s} \in B_{s} | \mathcal{F}_{s}] + \E[\delta_{s+1} | \mathcal{F}_{s}, v_{s} \in D_{s}] \P[v_{s} \in D_{s} | \mathcal{F}_{s}] \\
&\leq (\delta_{s} -(1- \frac{c}{n}))\P[v_{s} \in A_{s} | \mathcal{F}_{s}] + (\delta_{s} +\frac{c}{n})\P[v_{s} \in B_{s} | \mathcal{F}_{s}] + \delta_{s} \P[v_{s} \in D_{s} | \mathcal{F}_{s}] \\
&= \delta_{s} \P[v_{s} \in A_{s} \cup B_{s} \cup D_{s} | \mathcal{F}_{s}] - (1- \frac{c}{n})\P[v_{s} \in A_{s} | \mathcal{F}_{s}] + \frac{c}{n} \P[v_{s} \in B_{s} | \mathcal{F}_{s}] \\
&\leq \delta_{s} - \frac{\delta_{s}}{n}\big(1 - \frac{c}{n}\big) + \frac{2 c d V_{s}}{n^{2}}. 
\ee
This can be iterated to give:
\be 
\E[(\delta_{t+s}) \textbf{1}_{\sup_{t \leq i \leq t+s} V_{i} \leq k} | \mathcal{F}_{t}] &= \E[ \E[(\delta_{t+s}) \textbf{1}_{\sup_{t \leq i \leq t+s} V_{i} \leq k} | \mathcal{F}_{t+s-1}] | \mathcal{F}_{t}] \\
&\leq \E \Big[ \Big( \big(1 - \frac{1}{n} \big(1 - \frac{c}{n}\big)\big) \delta_{t+s-1} + \frac{ 2cd V_{t+s-1}}{n^{2}} \Big) \textbf{1}_{\sup_{t \leq i \leq t+s} V_{i} \leq k} \big| \mathcal{F}_{t} \Big]  \\
&\leq \E\Big[ \Big( \big(1 - \frac{1}{n}\big(1 - \frac{c}{n}\big)\big) \delta_{t+s-1} + \frac{ 2cd k}{n^{2}} \Big) \textbf{1}_{\sup_{t \leq i \leq t+s} V_{i} \leq k}  \big| \mathcal{F}_{t} \Big] \\
&\leq \ldots \\
&\leq \big(1 - \frac{1}{n} \big(1 - \frac{c}{n}\big)\big)^{s} \delta_{t} + \frac{2cdk}{n} \frac{1}{1 - \frac{c}{n}},   
\ee 
which is inequality \eqref{IneqSimpCompNPartNComp0}. Since $V_{s} \leq n$, this implies inequality \eqref{IneqSimpCompNPartNComp} for $n > 2c$ and the proof is  finished.
\end{proof}

\subsection{Corrected {number of} components}
Our next goal is to obtain a bound on $\E[Y_{t+s}|\mathcal{F}_t]$ in terms of $Y_t$ over certain time intervals. To this end, we digress briefly and introduce a new object called the ``corrected {number of} components" of a {graph}.
\begin{defn} \label{DefCorrCompCount}
For a graph $H$, we define a Markov chain $\{ H_{i} \}_{i \geq 0}$ with absorbing states as follows. Set $H_{0} = H$. For $i \geq 0$, if all components of $H_{i}$ are size 1, set $H_{i+1} = H_{i}$. Otherwise, select uniformly at random a vertex $v_{i}$ in $H_{i}$ that also has at least one neighbour in $H_{i}$ and set $H_{i+1} = H_{i} \backslash \{ v_{i} \}$.  We then define $\cN_{H} = \E[\lim_{i \rightarrow \infty} \vert H_{i} \vert]$. Since  $H_{i+1} = H_{i}$ for all $i > | H |$, this (random) limit always exists. See Figure \ref{PicEvolutionOfH} for an illustration of the evolution of $H_{t}$ for an initial graph $H_{0}$ with five vertices.
\end{defn}
\begin{defn}
 For $k \in \mathbb{N}$, set
\be
\cN_{k} = \sup \{\cN_{H} \, : \, H \subset \LL, \vert H \vert \leq k \}.
\ee
\end{defn}

\begin{figure} 
\begin{tikzpicture}[scale=0.35]
\draw (1,0) -- (3,0);
\draw (3.5, 0) -- (8,0);
\draw (9, 0) -- (14,0);
\draw (14.5,0) -- (22,0);
\draw  (23,0) -- (30,0);
\draw (0,4) -- (8.5,4);
\draw (10,4) -- (21,4);
\draw (10,8) -- (21,8);
\draw  (22,4) -- (32,4);
\draw (13,8) -- (19,8);
\draw [fill] (2,0) circle [radius=.25];
\draw [fill] (4.0,0) circle [radius=.25];
\draw [fill] (7.0,0) circle [radius=.25];
\draw [fill] (1,4) circle [radius=.25];
\draw [fill] (3,4) circle [radius=.25];
\draw [fill] (5,4) circle [radius=.25];
\draw [fill] (7,4) circle [radius=.25];
\draw [fill] (9.5,0) circle [radius=.25];
\draw [fill] (13,0) circle [radius=.25];
\draw [fill] (15,0) circle [radius=.25];
\draw [fill] (18,0) circle [radius=.25];
\draw [fill] (21.0,0) circle [radius=.25];
\draw [fill] (11,4) circle [radius=.25];
\draw [fill] (13,4) circle [radius=.25];
\draw [fill] (15,4) circle [radius=.25];
\draw [fill] (19.0,4) circle [radius=.25];
\draw [fill] (11,8) circle [radius=.25];
\draw [fill] (13,8) circle [radius=.25];
\draw [fill] (15,8) circle [radius=.25];
\draw [fill] (17,8) circle [radius=.25];
\draw [fill] (19.,8) circle [radius=.25];
\draw [fill] (23,4) circle [radius=.25];
\draw [fill] (25,4) circle [radius=.25];
\draw [fill] (29,4) circle [radius=.25];
\draw [fill] (31,4) circle [radius=.25];
\draw [fill] (25,0) circle [radius=.25];
\draw [fill] (29,0) circle [radius=.25];
\draw [->] (15,7.5) -- (15,4.5);
\draw (15.5,6) node{$\frac{2}{5}$};
\draw [->] (14.2,7.5) -- (4.0,4.5);
\draw (8,7) node{$\frac{2}{5}$};
\draw [->] (15.8,7.5) -- (26.0,4.5);
\draw (21,7) node{$\frac{1}{5}$};
\draw [->] (3,3.5) -- (2.0,0.5);
\draw (3,2) node{$\frac{1}{3}$};
\draw [->] (5,3.5) -- (6.0,0.5);
\draw (6.2,2) node{$\frac{2}{3}$};
\draw [->] (14,3.5) -- (12.0,0.5);
\draw (14.0,2) node{$\frac{1}{3}$};
\draw [->] (15.5,3.5) -- (17.0,0.5);
\draw (17.2,2) node{$\frac{2}{3}$};
\draw [->] (27,3.5) -- (27.0,0.5);
\draw (27.5,2) node{$1$};
 \end{tikzpicture}
\caption{Evolution of the component sizes of $H_{t}$.}
\label{PicEvolutionOfH}
\end{figure}
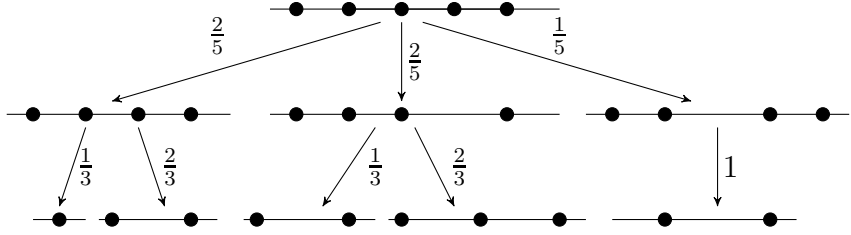

We make some initial observations. By simple case checking, we can show that
$\cN_{1} = {\cN}_{2} = 1, \, {\cN}_{3} = \frac{4}{3},
{\cN}_{4} = \frac{7}{4}.$ We clearly also have
\be  \label{EqTrivialStarCal}
{\cN}_{k}  < k, \quad \, k \geq 2.
\ee
Next, we show that $\cN_{H}$ can never be too small:
\begin{lemma} \label{LemCalc0Eff}
For any subgraph $H \subset \LL$, we have
\be 
\cN_{H} \geq \frac{|H|}{2d+1}.
\ee 
\end{lemma}
\begin{proof}
Let $\{ H_{i} \}_{i \geq 0}$ be the Markov process given in Definition \ref{DefCorrCompCount}, and let $H_{\infty}$ be its limit. We have
\be \label{EqCalc0EffTriv}
\cN_{H} &= \E \big[ \sum_{v \in H} \textbf{1}_{v \in H_{\infty}} \big] = \sum_{v \in H} \P[v \in H_{\infty}].
\ee 
However, for any particular vertex $v \in H$, $v \in H_{\infty}$ as long as it is the last vertex among its neighbours to be selected as an update vertex $v_{i}$. Since $v$ has at most $2d$ neighbours in $H$,
\be 
\P[v \in H_{\infty}] \geq \frac{1}{2d+1}.
\ee 
Combining this with formula \eqref{EqCalc0EffTriv},
\be 
\tilde{C}_{H} \geq \frac{|H|}{2d+1},
\ee 
completing the proof.
\end{proof}

For $1 \leq k \leq 2d$, let $H(k)$ be the \textit{star graph} with $k$ leaves (see Figure \ref{PicStarGraph} for a star graph with $d=2$, $k=4$). 

\begin{figure}
\begin{tikzpicture}[scale=0.35]
   \draw (1,0) -- (1,1) -- (1,2);
   \draw (0,1) -- (1,1) -- (2,1);
   \draw [fill] (1,1) circle [radius=.25];
   \draw [fill] (0,1) circle [radius=.25];
   \draw [fill] (2,1) circle [radius=.25];
   \draw [fill] (1,0) circle [radius=.25];
 \draw [fill] (1,2) circle [radius=.25];
\end{tikzpicture}
\caption{Star graph} 
\label{PicStarGraph}
\end{figure}
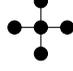

For the purposes of the following lemma, we denote the vertex set of this graph by $V(H(k)) = \{1,2,\ldots, k+1 \}$ and the edge set by $E(V(k)) = \{ (i,k+1) \, : \, 1 \leq i \leq k \}$.
We give basic bounds on how $\cN_{H(k)}$ depends on small changes to the subgraph $H$:

\begin{lemma} [Shrinking Star Graphs] \label{LemCalc1}
 For $\cN_{H}$ as in Definition \ref{DefCorrCompCount},
\be \label{EqCorrCountStar}
\cN_{H(k)} = \frac{k}{2} + \frac{1}{k+1}.
\ee

\end{lemma}

\begin{proof}
It is easy to check that $\cN_{H(1)} = 1$ and $\cN_{H(2)} = \cN_{3} = \frac{4}{3}$. Let $\{ H_{i} \}_{i \in \mathbb{N}}$ be the Markov chain described in Definition \eqref{DefCorrCompCount}, and let $\{v_{i} \}_{i \in \mathbb{N}}$ be the associated sequence of vertices. The quantity $\cN_{H(k)}$ satisfies the recurrence:
\be \label{EqShrinkStarRec}
\cN_{H(k)} &= k \P[v_{0} = k+1]   + \P[v_{0} \neq k+1] \cN_{H(k-1)} \\
&= \frac{k}{k+1} + \frac{k}{k+1} \cN_{H(k-1)}.
\ee
Iterating, for all $0 \leq q \leq k-3$ we have:
\be 
\cN_{H(k)} = \frac{1}{k+1}(k + (k-1) + \ldots + (k-q)) + \frac{k-q}{k+1} \cN_{H(k-q-1)}.
\ee 
Setting $q = k-3$ gives formula \eqref{EqCorrCountStar}. 
\end{proof}

More generally,

\begin{lemma} \label{LemCalc2}
For any graph $G$ and all subgraphs $H \subset G$ and all vertices $v \in G$ (adjacent to $H$ or not),
\be 
\cN_{H} \leq \cN_{H \cup \{ v \} } \leq \cN_{H} + 1.
\ee 
\end{lemma}

\begin{proof}
We begin by showing $\cN_{H \cup \{ v \} } \leq \cN_{H} + 1$ by induction on $m = | H |$. For $\vert H \vert \in \{ 1, 2\}$ this is clear by direct computation. Assume that  $\cN_{H \cup \{ v \} } \leq \cN_{H} + 1$ holds for all $|H| \leq m$. Fix $H$ with $|H| = m+1$, define $N(H)$ to be the vertices of $H$ with at least one neighbour in $H$ and, by the same argument as in recurrence \eqref{EqShrinkStarRec}, 
\be 
\cN_{H \cup \{ v \} } &=  \frac{1}{| N(H \cup \{v \})|} \sum_{u \in N(H \cup \{ v \})} \cN_{H \cup \{v \} \backslash \{ u \} }  \\
&\leq \frac{1}{| N(H \cup \{v \}) |} \cN_{H} \textbf{1}_{v \in N(H \cup \{v \})} + \frac{1}{| N(H \cup \{v \}) |} \sum_{u \in N(H \cup \{v \} ) \backslash \{ v \}} \big( \cN_{H \backslash \{u \}} + 1 \big) \\
&\leq \cN_{H} + 1,
\ee 
where the induction hypothesis is used in the second line. \\
We also prove $\cN_{H} \leq \cN_{H \cup \{ v \} }$ by induction on $m = | H|$. Again, the case that $\vert H \vert \in \{ 1, 2\}$ is clear. Assume that  $\cN_{H} \leq \cN_{H \cup \{ v \} }$ holds for all $|H| \leq m$. Fix $H$ with $|H| = m+1$. We consider two cases: $v$ is adjacent to $H$ or $v$ is not adjacent to $H$. If $v$ is not adjacent to $H$, we clearly have $\cN_{H \cup \{v\}} = \cN_{H} + 1$. If $v$ is adjacent to $H$, then $N(H \cup \{ v \} ) = N(H) \cup \{ v \}$. In this case, 
\be 
\cN_{H \cup \{ v \} } &= \frac{1}{\vert N(H) \vert + 1} \cN_{H} + \frac{1}{\vert N(H) \vert +1} \sum_{u \in N(H)} \cN_{H \cup \{v \} \backslash \{ u \} }  \\ \\
&\geq \frac{1}{\vert N(H) \vert + 1} \cN_{H} + \frac{1}{\vert N(H) \vert + 1} \sum_{u \in N(H)} \cN_{H  \backslash \{ u \} } \\
&= \cN_{H},
\ee 
where the induction hypothesis is used in the second line. This completes the proof.
\end{proof}
\subsection{{A corrected version of} of $Y_t$}
 In order to bound $Y_{s} - Y_{t}$, we introduce the process 
  \be \label{eqn:ys}
  \tilde{Y}_{s} = \sum_{H \in \mathrm{Comp}_{s}} \cN_{H}.
 \ee 
We think of $\{ \tilde{Y}_{s} \}_{s \geq t}$ as a `corrected' version of $\{ Y_{s} \}_{s \geq t}$. In general, $\{\tilde{Y}_{s} \}_{s \geq t}$ is easier to work with than $\{ Y_{s} \}_{s \geq t}$. One reason is the martingale-like property
 \be
 \E[(\tilde{Y}_{t+1} - \tilde{Y}_{t})\textbf{1}_{p_{t} > \frac{c}{n}}|\mathcal{F}_t] = 0,
 \ee
which does not hold for $\{Y_{s}\}_{s \geq t}$ (here $p_t$ is as defined in \eqref{EqCiRep}). In addition, $\{\tilde{Y}_{s} \}_{s \geq t}$ is much better-behaved than $\{Y_{s} \}_{s \geq t}$ over short time intervals, especially when the KCIP is far from equilibrium. Using \eqref{EqTrivialStarCal}, it can be verified that
\be \label{IneqRev1Star1DeltaBd}
Y_{s} \leq \tilde{Y}_{s} \leq Y_{s} + \delta_{s}.
\ee
Since $\delta_{s}$ is often small (see Lemma \ref{LemmCompPartNumComp}), $\tilde{Y}_{s}$ is a good proxy for the quantity $Y_{s}$.

Next, we show that the corrected number of components $\{ \tilde{Y}_{s} \}_{s \geq t}$ does not grow too quickly over a moderate time period:
\begin{lemma}  \label{LemmaCalc3}
Fix $0 < \epsilon \leq \frac{1}{96 c^{2} (d+1)^{3}}$. Then for all $n > N(c,d)$  sufficiently large,
and $n \log(n) \leq s \leq \frac{n^{3}}{96 c^{2} d^{3}}$,
\be
\E[\tilde{Y}_{t + s} \vert \mathcal{F}_{t}] \leq \tilde{Y}_{t} (1 + O(\epsilon)) + O(1).
\ee
\end{lemma}

\begin{proof}
We will show that for $0 \leq s \leq \epsilon n^3$,
\be \label{eqn:sunifbdyt}
\E[\tilde{Y}_{t + s} \vert \mathcal{F}_{t}] \leq  \tilde{Y}_{t} \big( 1 + \frac{96 c^{2} (d+1)^{3}}{n^{3}} \big( 8n \big(1 - \frac{1}{8n} \big)^{s} + s \big) \big) +  \frac{24 cd}{n} \delta_{t} + \big(1 - \frac{1}{8n} \big)^{s} \delta_{t}.
\ee 
Since $\delta_t \leq n$, Lemma \ref{LemmaCalc3} immediately follows from 
\eqref{eqn:sunifbdyt}. \par
Assume without loss of generality that $t = 0$ and define
\be
f(x) = \frac{3cd}{n^{2}} \sum_{u=0}^{x} \E[\delta_{u} \vert \mathcal{F}_{0}].
\ee
By Lemma \ref{LemCalc2}, $\E[\tilde{Y}_{s+1} | \mathcal{F}_{s}] \leq \tilde{Y}_{s} + 1$ for any time $s$. Inequality \eqref{EqTrivialStarCal} provides further necessary (but not sufficient) conditions for $\tilde{Y}_{s+1} > \tilde{Y}_{s}$ to hold when the update variable $p_{s}$ in representation \eqref{EqCiRep} satisfies $p_{s} < \frac{c}{n}$: the vertex $v_{s}$ in representation \eqref{EqCiRep} must be adjacent to a component of size at least two in $G_{s}$, and furthermore we must have $X_{s}[v_{s}] = 0$. At time $s$, there are at most  $3 d \delta_{s}$ vertices satisfying this necessary condition, and so for $0 \leq s \leq \epsilon n^{3}$,

\be \label{IneqBasicYTilde}
\E[\tilde{Y}_{s} \vert \mathcal{F}_{0}] &= \E[ \E[\tilde{Y}_{s} \vert \mathcal{F}_{s-1}] | \mathcal{F}_{0}] \\
&= \E \Big[ \E \big[\tilde{Y}_{s} \vert \mathcal{F}_{s-1}, \big\{p_{s-1} \leq \frac{c}{n} \big\}\big] \P \big[p_{s-1} \leq \frac{c}{n} | \mathcal{F}_{s-1}\big] \\
&+ \E\big[\tilde{Y}_{s} \vert \mathcal{F}_{s-1}, \big\{p_{s-1} > \frac{c}{n} \big\}\big] \P\big[p_{s-1} > \frac{c}{n} | \mathcal{F}_{s-1}\big] \big| \mathcal{F}_{0}\Big] \\
&\leq \E \big[ \big( \tilde{Y}_{s-1} + \frac{3 d \delta_{s-1}}{n} \big) \P \big[p_{s-1} \leq \frac{c}{n} | \mathcal{F}_{s-1}\big] + \tilde{Y}_{s-1} \P \big[p_{s-1} > \frac{c}{n} | \mathcal{F}_{s-1}\big]  \big| \mathcal{F}_{0} \big]  \\
&\leq \E[\tilde{Y}_{s-1} \vert \mathcal{F}_{0}] + \frac{3 cd}{n^{2}} \E[\delta_{s-1} \vert \mathcal{F}_{0} ] \\
&\leq \ldots \\
&\leq \tilde{Y}_{0} +  \frac{3cd}{n^{2}} \sum_{u=0}^{s-1} \E[\delta_{u} \vert \mathcal{F}_{0}]. \\
\ee 
Lemma \ref{LemCalc0Eff} implies 
\be \label{IneqSimpCorrCompCountVsPart} 
\tilde{Y}_{s} \geq \frac{V_{s}}{2d+1}.
\ee 
Using this fact, and then inequality \eqref{IneqDiffDec} followed by inequality \eqref{IneqBasicYTilde}, we find that for all $n$ sufficiently large,
\be \label{LongCalcOfCorrComp}
f(s) &\equiv \frac{3cd}{n^{2}} \sum_{u=0}^{s} \E[\delta_{u} \vert \mathcal{F}_{0}] \\
&= \frac{3cd}{n^{2}} \big( \delta_{0} +  \sum_{u=1}^{s} \E[\E [\delta_{u} | \mathcal{F}_{u-1}] | \mathcal{F}_{0}] \big) \\
&\leq \frac{3cd}{n^{2}} \big( \delta_{0} + \sum_{u=1}^{s} \E[ \big( 1- \frac{1}{4n} \big) \delta_{u-1} + \frac{2cd V_{u-1}}{n^{2}}  | \mathcal{F}_{0}] \big) \\
&\leq  \frac{3cd}{n^{2}} \big( \delta_{0} +  \big( 1- \frac{1}{4n} \big) \sum_{u=0}^{s-1} \E[ \delta_{u} | \mathcal{F}_{0}] + \frac{4c(d+1)^{2}}{n^{2}} \sum_{u=0}^{s-1} \E[ \tilde{Y}_{u} | \mathcal{F}_{0}] \big) \\
&\leq \frac{3cd}{n^{2}} \big( \delta_{0} +  \big( 1- \frac{1}{4n} \big) \sum_{u=0}^{s-1} \E[ \delta_{u} | \mathcal{F}_{0}] + \frac{4c(d+1)^{2}}{n^{2}} (s-1) \tilde{Y}_{0} + \frac{12 c^{2} (d+1)^{3}}{n^{4}} \sum_{u=0}^{s-2} u \E[\delta_{u} | \mathcal{F}_{0}] \big) \\
&\leq \frac{3cd}{n^{2}} \delta_{0} + \frac{12c^{2}(d+1)^{3}}{n^{4}} (s-1) \tilde{Y}_{0} + \frac{3cd}{n^{2}} \big( 1 - \frac{1}{8n} \big) \sum_{u=0}^{s-1} \E[\delta_{u} | \mathcal{F}_{0}] \\
&= \frac{3cd}{n^{2}} \delta_{0} + \frac{12c^{2}(d+1)^{3}}{n^{4}} (s-1) \tilde{Y}_{0} + \big( 1 - \frac{1}{8n} \big) f(s-1).
\ee
Inequality \eqref{IneqBasicYTilde} is used in the fifth step of this bound and we use the hypothesis that $\epsilon \leq \frac{1}{96 c^{2} (d+1)^{3}}$ in the sixth step. Using this calculation with the last line of inequality \eqref{IneqBasicYTilde}, we have

\be \label{EqLastLineLemmaCalc3}
\E[\tilde{Y}_{s} \vert \mathcal{F}_{0}] &\leq \tilde{Y}_{0} + f(s-1) \\
&\leq \tilde{Y}_{0} \big( 1 + \frac{12 c^{2} (d+1)^{3}}{n^{4}} (s-1) \big) + \frac{3cd}{n^{2}} \delta_{0} + \big(1 - \frac{1}{8n} \big) f(s-2) \\
&\leq \ldots \\
&\leq \tilde{Y}_{0} \big( 1 + \sum_{x=0}^{s-q} \frac{12 c^{2} (d+1)^{3}}{n^{4}} x \big(1 - \frac{1}{8n} \big)^{s-x} \big) + \frac{3cd}{n^{2}} \delta_{0} \sum_{x=0}^{s-q} \big( 1 - \frac{1}{8n} \big)^{x} + \big(1 - \frac{1}{8n} \big)^{s-q} f(q) \\
&\leq \ldots \\
&\leq \tilde{Y}_{0} \big( 1 + \sum_{x=0}^{s}\frac{12 c^{2} (d+1)^{3}}{n^{4}} x \big(1 - \frac{1}{8n} \big)^{s-x} \big) + \frac{3cd}{n^{2}} \delta_{0} \sum_{x=0}^{s} \big( 1 - \frac{1}{8n} \big)^{x} + \big(1 - \frac{1}{8n} \big)^{s} f(0) \\
&\leq \tilde{Y}_{0} \big( 1 + \frac{96 c^{2} (d+1)^{3}}{n^{3}} \big( 8n \big(1 - \frac{1}{8n} \big)^{s} + s \big) \big) +  \frac{24 cd}{n} \delta_{0} + \big(1 - \frac{1}{8n} \big)^{s} \delta_{0}.
\ee 
This completes the proof of inequality \eqref{eqn:sunifbdyt}; by the comment immediately following inequality \eqref{eqn:sunifbdyt}, this completes the proof of the Lemma.
\end{proof}

\subsection{Typical component size involved in collisions}
In this subsection, we study the behaviour of a `typical' collision in the KCIP by looking at the behaviour of the KCIP only for vertices and times close to the collision. This approach is illustrated in Figure \ref{PicBlowUp}.

\begin{figure}
\begin{tikzpicture}[scale=0.5]
\begin{scope}
    \draw (0, 0) grid (10, 10);
    \draw  [fill] (1,3) circle [radius=.25];
    \draw  [fill] (1,7) circle [radius=.25];
    \draw  [fill] (2,5) circle [radius=.25];
    \draw  [fill] (2,9) circle [radius=.25];
    \draw  [fill] (3,1) circle [radius=.25];
    \draw  [fill] (4,6) circle [radius=.25];
     \draw [fill] (6,1) circle [radius=.25];
     \draw [fill] (7,4) circle [radius=.25];
     \draw [fill] (8,2) circle [radius=.25];
     \draw [fill] (8,9) circle [radius=.25];
      \draw [fill] (9,3) circle [radius=.25]; 
      \draw [fill] (9,5) circle [radius=.25];

      \draw [fill] (3,2) circle [radius=.25];
      \draw [fill] (5,6) circle [radius=.25];
      \draw [fill] (5,7) circle [radius=.25];
      \draw [fill] (7,5) circle [radius=.25];
      \draw [fill] (8,8) circle [radius=.25];
      \draw [red] (4.5, 6.5) circle [radius = 1.5];
      \end{scope}
     \begin{scope}
      \draw (17,10) -- (17, 14) -- (20,14) -- (20,10) -- cycle;
      \draw (17,11) -- (18, 11) -- (19,11) -- (20,11);
      \draw (17,12) -- (18, 12) -- (19,12) -- (20,12);
      \draw (17,13) -- (18, 13) -- (19,13) -- (20,13);
      \draw (18,10) -- (18,11) -- (18,12) -- (18,13) -- (18,14);
      \draw (19,10) -- (19,11) -- (19,12) -- (19,13) -- (19,14) ;
       \draw [fill] (18,11) circle [radius=.25];
       \draw [fill] (19,11) circle [radius=.25];
        \draw [fill] (19,12) circle [radius=.25];
      \end{scope}
       
       \begin{scope}
      \draw (12,7) -- (12, 11) -- (15,11) -- (15,7) -- cycle;
      \draw (12,8) -- (13, 8) -- (14,8) -- (15,8);
      \draw (12,9) -- (13, 9) -- (14,9) -- (15,9);
      \draw (12,10) -- (13, 10) -- (14,10) -- (15,10);
      \draw (13,7) -- (13,8) -- (13,9) -- (13,9) -- (13,10)--(13,11);
      \draw (14,7) -- (14,8) -- (14,9) -- (14,10) -- (14,11) ;
      \draw [fill] (13,8) circle [radius=.25];
       \draw [fill] (14,8) circle [radius=.25];
       \end{scope}
      
      \begin{scope}
       \draw (12,0) -- (12, 4) -- (15,4) -- (15,0) -- cycle;
        \draw (12,1) -- (13, 1) -- (14,1) -- (15,1);
      \draw (12,2) -- (13, 2) -- (14,2) -- (15,2);
      \draw (12,3) -- (13, 3) -- (14,3) -- (15,3);
      \draw (13,0) -- (13,1) -- (13,2) -- (13,2) -- (13,3)--(13,4);
      \draw (14,0) -- (14,1) -- (14,2) -- (14,3) -- (14,4) ;
      \draw [fill] (13,1) circle [radius=.25];
      \end{scope}
      \begin{scope}
      \draw (22,7) -- (22, 11) -- (25,11) -- (25,7) -- cycle;
      \draw (22,8) -- (23, 8) -- (24,8) -- (25,8);
      \draw (22,9) -- (23, 9) -- (24,9) -- (25,9);
      \draw (22,10) -- (23, 10) -- (24,10) -- (25,10);
      \draw (23,7) -- (23,8) -- (23,9) -- (23,9) -- (23,10)--(23,11);
      \draw (24,7) -- (24,8) -- (24,9) -- (24,10) -- (24,11) ;
      \draw [fill] (24,8) circle [radius=.25];
       \draw [fill] (24,9) circle [radius=.25];
      \end{scope}
      \begin{scope}
       \draw (22,0) -- (22, 4) -- (25,4) -- (25,0) -- cycle;
       \draw (22,1) -- (23, 1) -- (24,1) -- (25,1);
      \draw (22,2) -- (23, 2) -- (24,2) -- (25,2);
      \draw (22,3) -- (23, 3) -- (24,3) -- (25,3);
      \draw (23,0) -- (23,1) -- (23,2) -- (23,2) -- (23,3)--(23,4);
      \draw (24,0) -- (24,1) -- (24,2) -- (24,3) -- (24,4) ;
      \draw [fill] (24,2) circle [radius=.25];
      \end{scope}
      \begin{scope}
      \draw [red] (4.5,5) -- (20,10);
      \draw [red] (4.5,8) -- (17,14);
      \end{scope}
      
      \begin{scope}
      \draw (17,0) -- (17, 4) -- (20,4) -- (20,0) -- cycle;
      \draw (17,1) -- (18, 1) -- (19,1) -- (20,1);
      \draw (17,2) -- (18, 2) -- (19,2) -- (20,2);
      \draw (17,3) -- (18, 3) -- (19,3) -- (20,3);
      \draw (18,0) -- (18,1) -- (18,2) -- (18,3) -- (18,4);
      \draw (19,0) -- (19,1) -- (19,2) -- (19,3) -- (19,4) ;
       \draw [fill] (18,1) circle [radius=.25];
       \draw [fill] (19,2) circle [radius=.25];
      \end{scope}
      
      \begin{scope}
      \draw [->] (16.8, 12) .. controls (15, 11.9) .. (13.5,11.2)
        node[pos = 0.3, above]{$\frac{1}{3}$};
      \draw [->] (20.2, 12) .. controls (22, 11.9) .. (23.5,11.2)
       node[pos = 0.3, above]{$\frac{1}{3}$};
      \draw [->] (18.5,9.3) -- (18.5,4.5);
      \draw [->]  (13.5,6.5) -- (13.5,4.5);
      \draw [->]  (23.5,6.5) -- (23.5,4.5);
      \draw (14.0, 5.5) node {1};
      \draw (24.0, 5.5) node {1};
      \draw (19.0, 7.0) node{$\frac{1}{3}$};
      \end{scope}
      
\end{tikzpicture}

\caption{Neighborhood of a Collision}
\label{PicBlowUp}
\end{figure}
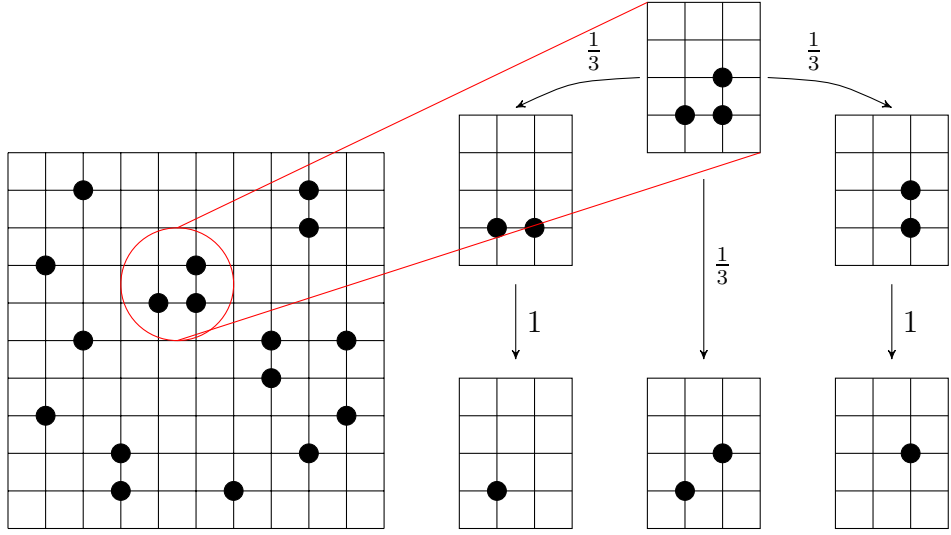 

Fix $t > 0$. Define  the \textit{number of collisions} of the KCIP to be
\be \label{EqDefCCs}
\cC_{s} = \vert \{ t \leq u < s \, : \, Y_{u+1} < Y_{u} \} |,
\ee 
and define the \textit{set of collision times} 
\be 
\colT{s} = \{ t \leq u \leq t + s \, : \, Y_{u+1} < Y_{u} \}.
\ee 
Recall the definition of the update variable $v_{s}$ used in formula \eqref{EqCiRep}. For $u \in \colT{s}$, we define
\be
\colM{u} = \sup \{ \vert \mathrm{Comp}_{u}(w) \vert \, : \, (v_{u}, w) \in E \}
\ee
to be the largest component involved in a collision; for $u \notin \colT{s}$ we set $\colM{u} = 0$. We first show that, given $u \in \colT{s}$, $\colM{u}$ is often equal to 1: 
\begin{lemma} [Typical Component Size] \label{LemmaEL1}
There exists $\delta = \delta(c,d) > 0$ independent of $n$ so that for any $\epsilon >0$
\be \label{IneqTypCompSizeMainIneq}
\E \big[ \vert \big\{ t + n \log(n)^{4} &\leq u \leq t + \epsilon n^{3} \, : \, \colM{u} = 1 \big\} \vert \, | \, \mathcal{F}_{t} \big] \\
& \geq (\delta - o(n^{-1})) \, \E \big[ \vert \big\{ t + n \log(n)^{4} \leq u \leq t + \epsilon n^{3} \, : \, \colM{u} \neq 0 \big\} \vert \, | \, \mathcal{F}_{t} \big] 
\ee
uniformly in the initial configuration $X_{t}$.
\end{lemma} 
\begin{remark}
We briefly discuss why the conclusion obtained in Lemma \ref{LemmaEL1} is plausible. Fix a time $t$ and consider any vertex $v$ and any neighborhood $\mathcal{B}(v)$ whose size does not grow with $n$. It is easy to check that over any time interval of length $T$, with  $n \ll T \ll n^{2}$, it is very likely that all vertices in $\mathcal{B}(v)$ have been updated and also that no particles have been added to $\mathcal{B}(v)$. On that very likely event, all components of $G_{t+T} \cap \mathcal{B}(v)$ are of size 1. In other words, after a burn-in time of length $O(n)$, most vertices are not close to any components of size greater than 1. The proof largely consists of checking that conditioning on a collision occurring at vertex $v$ at times in the interval $(t+T, t+ \epsilon n^3)$ does not affect this conclusion too much. 
\end{remark}
\begin{proof}[Proof of  Lemma \ref{LemmaEL1}]
Set $T = n \log(n)^{4}$. Let $\{ p_{s} \}_{s \in \mathbb{N}}$, $\{ v_{s} \}_{s \in \mathbb{N}}$ {be as given in formula \eqref{EqCiRep} to define the dynamics of the KCIP}. For fixed $v \in \LL$ and $t + T \leq r \leq t + \epsilon n^{3}$, define the event 
\be
\EvCo{v,r} = \{r \in \colT{\epsilon n^{3}}, v_{r} = v \}.
\ee
Thus $\EvCo{v,r}$ denotes the event that a collision occurs at time $r$ and vertex $v$.
We will show that, for any $v \in \LL$, uniformly in $t + T \leq r \leq t + \epsilon n^{3}$,
\be \label{eqn:colMrcond}
\P[\colM{r} = 1 | \mathcal{F}_{t}, \EvCo{v,r}] \geq (\delta - o(n^{-1}))
\ee
for some $\delta > 0$.
Inequality \eqref{eqn:colMrcond} will immediately yield \eqref{IneqTypCompSizeMainIneq}. To see this, \eqref{eqn:colMrcond} implies that
\be 
\E \big[ \vert \big\{ t + T &\leq u \leq t + \epsilon n^{3} \, : \, \colM{u} = 1 \big\} \vert \, | \, \mathcal{F}_{t} \big] \\
&= \sum_{v \in \LL} \sum_{r = t + T}^{t + \epsilon n^{3}} \P[\colM{r} = 1 | \mathcal{F}_{t}, \EvCo{v,r}] \P[\EvCo{v,r} | \mathcal{F}_{t}] \\
&\geq (\delta - o(n^{-1})) \sum_{v \in \LL} \sum_{r = t + T}^{t + \epsilon n^{3}}  \P[\EvCo{v,r} | \mathcal{F}_{t}] \\
&= (\delta - o(n^{-1})) \E \big[ \vert \big\{ t + T \leq u \leq t + \epsilon n^{3} \, : \, \colM{u} \neq 0 \big\} \vert \, | \, \mathcal{F}_{t} \big],
\ee 
yielding inequality \eqref{IneqTypCompSizeMainIneq}. Thus it suffices to show \eqref{eqn:colMrcond}.\par
  We begin with a simple bound on the probability that the number of particles in a small region ever goes up by more than a small number over any small time interval.  For $n \in \mathbb{N}$ and $0 \leq q \leq 1$, denote by $\mathrm{Binomial(n,q)}$ a binomial random variable with $n$ trials and success probability $q$. For any fixed $t \leq r \leq t + \epsilon n^{3}$, $v \in \LL$ and $\ell, k \in \mathbb{N}$, we have for $n > \sqrt{c (2d)^{\ell + 1}}$,
\be 
\P[ \sum_{s=r+1}^{r+T} \textbf{1}_{v_{s} \in \mathcal{B}_{\ell}(v)} \textbf{1}_{p_{s} < \frac{c}{n}} \geq k] &\leq \P[\mathrm{Binomial}(T, \frac{c (2d)^{\ell + 1}}{n^{2}}) \geq k] \\
&\leq \P[\mathrm{Binomial}(T, \frac{c (2d)^{\ell + 1}}{n^{2}}) \geq 1]^{k} \\
&\leq \big( \frac{c (2d)^{\ell +1} \log(n)^{4}}{n} \big)^{k},\label{IneqSparseSimple}
\ee 
where the first inequality uses the simple bound $| \mathcal{B}_{\ell}(v) | \leq (2d)^{\ell + 1}$. \par
For $v \in \LL$, $\ell \in \mathbb{N}$, and $t + T \leq r \leq t + \epsilon n^{3}$, define the event
\be
\EvSp{v,r, \ell} = \{ \sum_{s=r-T}^{r-1} \textbf{1}_{v_{s} \in \mathcal{B}_{\ell}(v)} \textbf{1}_{p_{s} < \frac{c}{n}} < 8 \}
\ee
 and let  $\mathcal{A}^{(\mathrm{sparse})} = \cup_{t + T \leq r \leq t + \epsilon n^{3}, v \in \LL} \EvSp{v,r,100}$. By inequality \eqref{IneqSparseSimple} and a union bound, 
 \be \label{IneqEasySparsityBound}
\P [ \mathcal{A}^{(\mathrm{sparse})}|\mathcal{F}_t] &\geq 1 - O(n^{4} \frac{\log(n)^{32}}{n^{8}}) = 1- o(n^{-3}).
\ee 
From \eqref{IneqEasySparsityBound}, we thus have 
\be
\P [\colM{r} = 1 | \mathcal{F}_{t}, \EvCo{v,r}] &\geq \P [\colM{r} = 1 | \mathcal{F}_{t}, \EvCo{v,r}, \EvSp{v,r, 100}] \P[\EvSp{v,r, 100}|\mathcal{F}_t] \\
&\geq \P [\colM{r} = 1 | \mathcal{F}_{t}, \EvCo{v,r}, \EvSp{v,r, 100}] (1-o(1)). \label{eqn:Mcolsparsebd}
\ee
In light of \eqref{eqn:Mcolsparsebd}, to show \eqref{eqn:colMrcond}, it is enough to show
\be \label{eqn:colllemrealbd}
\P [\colM{r} = 1 | \mathcal{F}_{t}, \EvCo{v,r}, \EvSp{v,r, 100}] \geq  (\delta - o(n^{-1})).
\ee
To this end, we proceed by defining two events of interest. For $v \in \LL$, $t +T \leq r \leq t + \epsilon n^{3}$, and some constant $B_{1}$ to be determined later, define the event 
\be
\EvGap{v,r} = \{ \sum_{u = r - B_{1}n}^{r} \textbf{1}_{p_{u} < \frac{c}{n}} \textbf{1}_{v_{u} \in \mathcal{B}_{100}(v)} = 0 \}
\ee that no vertices are added near $v$ in the (short) time interval of length $B_{1} n$ immediately before the collision.  Finally, for $k \in \mathbb{N}$ and $t + T \leq r \leq t + \epsilon n^{3}$, define the event 
\be
\EvReg{v,r,k} = \{ \sum_{u = r-T}^{{r}} \textbf{1}_{v_{u} \in \mathcal{B}_{100}(v)} \leq k \}
\ee
that at most $k$ vertices are updated in a ball around $v$ in the interval of length $T$ before the collision. \par

Fix the constant $C_{1} = 25 (2d)^{101}$. For any $t + T \leq r \leq t + \epsilon n^{3}$, we have
\be 
\P &[\colM{r} = 1 | \mathcal{F}_{t}, \EvCo{v,r}, \EvSp{v,r, 100}] \\
&\geq \P[\colM{r} = 1 | \mathcal{F}_{t}, \EvCo{v,r}, \EvSp{v,r, 100}, \EvGap{v,r} ] \times\P[\EvGap{v,r} | \mathcal{F}_{t}, \EvCo{v,r}, \EvSp{v,r, 100}] \\
&\geq \P[\colM{r} = 1 | \mathcal{F}_{t}, \EvCo{v,r}, \EvSp{v,r, 100}, \EvGap{v,r} ] \times\P[\EvGap{v,r} | \mathcal{F}_{t}, \EvCo{v,r}, \EvSp{v,r, 100}, \EvReg{v,r, C_{1} \log(n)^{4}} ] 
\\
&\hspace{1cm} \times \P[\EvReg{v,r, C_{1} \log(n)^{4}} |  \mathcal{F}_{t}, \EvCo{v,r}, \EvSp{v,r, 100}]\\
&\equiv \T_1 \times \T_2 \times \T_3. \label{IneqLemmaEL1MainDecomp}
\ee 
{The remainder of the proof consists of obtaining a lower bound for}  each of the three factors $\T_1, \T_2$ and $\T_3$ in inequality \eqref{IneqLemmaEL1MainDecomp}. We begin by bounding $\T_3$. For a subgraph $H \subset \LL$ and a configuration $X \in \{ 0, 1 \}^{\LL}$, denote by $X|_{H}$ the restriction of $X$ to $H$; that is, $X|_{H} \in \{0,1\}^{H}$ and satisfies $X|_{H}[w] = X[w]$ for all $w \in H$. Let 
\be
\psi_{H} = \{ \psi_{H}(0), \psi_{H}(1), \ldots \} \equiv \{ s \geq t \, : \, X_{s}|_{H} \neq X_{s+1}|_{H} \}\ee
be the ordered sequence of times at which the restriction of $\{ X_{s} \}_{s \in \mathbb{N}}$ to $H$ changes. Finally, let  $\mathcal{G}_{v,r}$ be the $\sigma$-algebra generated by the random variables $\{ X_{s}|_{\mathcal{B}_{100}(v)}\}_{ s \in {\psi_{\mathcal{B}_{100}(v)}} \cap \{r - T, \ldots, r-1 \}}$ and $\psi_{\mathcal{B}_{100}(v)}$. For notational convenience, let $\S_t$ be shorthand for $\mathcal{F}_{t}, \EvCo{v,r} ,\EvSp{v,r, 100}$ and denote by $\psi^c_{\mathcal{B}_{100}(v)}$ the complement of the set $\psi_{\mathcal{B}_{100}(v)}$. We have
\be 
1 &- \T_3 = 1 - \E[ \P[\EvReg{v,r, C_{1} \log(n)^{4}} | \S_t, \mathcal{G}_{v,r}] |  \S_t] \\
&\leq \frac{1}{C_{1} \log(n)^{4}} \E \bigg[{\E[ \sum_{s = r - T}^{r-1} \textbf{1}_{v_{s} \in \mathcal{B}_{100}(v)} \big| \S_t, \mathcal{G}_{v,r}]} |  \S_t \bigg] \label{IneqCollSizeBdType3Set0} \\
&\leq  \frac{1}{C_{1} \log(n)^{4}}\E \bigg[   {(2d)^{101} + 16 + \E[ \sum_{s \in  \{r - T \leq u \leq r-1 \, : \, v_{u} \in \mathcal{B}_{100}(v) \} \cap \psi^c_{\mathcal{B}_{100}(v)}} \textbf{1}_{v_{s} \in \mathcal{B}_{100}(v)} |  \S_t, \mathcal{G}_{v,r}]} \big|  \S_t \bigg].\\\label{IneqCollSizeBdType3Set1}
\ee 
Inequality \eqref{IneqCollSizeBdType3Set0} is simply an application of Markov's inequality. To obtain \eqref{IneqCollSizeBdType3Set1}, we split $\{ s \, : \, v_{s} \in \mathcal{B}_{100}(v) \} \cap \{r-T, r-T + 1,\ldots, r-1\}$ into two sets: $\{ r-T, \ldots, r-1 \} \cap \psi_{\mathcal{B}_{100}(v)}$ and everything else. Inequality \eqref{IneqCollSizeBdType3Set1} then follows from noting that the number $| \{ r-T, \ldots, r-1 \} \cap \psi_{\mathcal{B}_{100}(v)} |$ of times that a particle is added to or removed from $\mathcal{B}_{100}(v)$ between times $r-T$ and $r-1$ is, at most, the number of particles in that set at time $r-T$ plus twice the number that have been added between times $r-T$ and $r-1$. Thus, conditional on $\mathcal{A}^{\mathrm{sparse}}$, we have $| \{ r-T, \ldots, r-1 \} \cap \psi_{\mathcal{B}_{100}(v)} | \leq (2d)^{101} + 16$.  \\

Fix $s \in  \{r - T, \ldots, r-1 \} \cap \psi_{\mathcal{B}_{100}(v)}$. We claim that if $w_{1}, w_{2} \in \LL \backslash \mathcal{B}_{108}(v)$, then 
\be \label{IneqRationIrrelevantUpdates}
\frac{\P[v_{s} = w_{1} | \S_t, \mathcal{G}_{v,r}]}{\P[v_{s} = w_{2} |  \S_t, \mathcal{G}_{v,r}]} =1,
\ee 
since updates to these two vertices cannot influence anything in $\mathcal{B}_{100}(v)$ before time $r$. More formally, for any update sequence $\{(v_{u}, p_{u}) \}_{u = r -T}^{r-1}$ with $v_{u} = w_{1}$, define the update sequence $\{(v_{u}', p_{u}) \}_{u = r -T}^{r-1}$ by $v_{u}' = v_{u} $ for $u \neq s$ and $v_{s}' = w_{2}$. This map is a bijection between the update sequences allowed by the conditions in the numerator of 
Equation \eqref{IneqRationIrrelevantUpdates} and the update sequences allowed by the conditions in the denominator of Equation \eqref{IneqRationIrrelevantUpdates}; the existence of this bijection proves Equation \eqref{IneqRationIrrelevantUpdates}. We similarly observe that if $w_{1} \in \mathcal{B}_{100}(v)$ and  $w_{2} \in G \backslash \mathcal{B}_{108}(v)$, then 
\be  \label{IneqRationSlightlyRelevantUpdates}
\frac{\P[v_{s} = w_{1} | \S_t, \mathcal{G}_{v,r}]}{\P[v_{s} = w_{2} |  \S_t, \mathcal{G}_{v,r}]} \leq 1,
\ee 
since certain updates within $\mathcal{B}_{100}(v)$ may be forbidden by the conditioning $ \mathcal{G}_{v,r}$. This can be made formal in essentially the same way as the argument for Equation  \eqref{IneqRationIrrelevantUpdates}. Inequalities \eqref{IneqRationIrrelevantUpdates} and \eqref{IneqRationSlightlyRelevantUpdates} imply 
\be 
\P[v_{s} \in \mathcal{B}_{100}(v) |  \S_t, \mathcal{G}_{v,r}] &\leq \frac{| \mathcal{B}_{100}(v)|}{n - | \mathcal{B}_{108}(v) | }   \\
&\leq \frac{(2d)^{101}}{n} + o(n^{-1}),
\ee 
where the second line follows from noting that $| \mathcal{B}_{\ell}(v)| \leq (2d)^{\ell + 1}$. Combining this with inequality \eqref{IneqCollSizeBdType3Set1}, 
\be 
\T_3 = \P[\EvReg{v,r, C_{1} \log(n)^{4}} |  \S_t] \geq 1 - \frac{ (2d)^{101} + (2d)^{101} \log(n)^{4}}{C_{1} \log(n)^{4}} + o(n^{-1}).
\ee 
Since $C_1 = 25(2d)^{101}$, for $n$ sufficiently large,
\be \label{IneqCollSizeBdType3SetFin}
\T_3 \geq {23 \over 25}. 
\ee 
 Next we bound $\T_2$.  For $H \subset \LL$, let $\phi_{H} = \{ \phi_{H}(0), \phi_{H}(1), \ldots \} = \{ s \geq t \, : \, v_{s} \in H \}$ be the ordered list of times at which the update vertex $v_{s}$ falls in $H$. Then let  $\mathcal{H}_{v,r}$ be the $\sigma$-algebra generated by the random variables $\{ v_{s} \, : \, s \in \phi_{\mathcal{B}_{100}(v)} \cap \{r -T, \ldots, r-1 \} \}$. Unlike $\mathcal{G}_{v,r}$, the update times $\phi_{\mathcal{B}_{100}(v)}$ are not included in this $\sigma$-algebra, only the update locations $v_{s}$. \par

For notational convenience, 
let $\S'_t =  \{  \mathcal{F}_{t}, \EvCo{v,r}, \EvSp{v,r, 100},\EvReg{v,r, C_{1} \log(n)^{4}} \}$. The gap condition $\EvGap{v,r}$ is fulfilled if there are no updates to the region $\mathcal{B}_{100}(v)$ in the time interval $\{r-B_{1}n, \ldots, r-1\}$, and so 
\be \label{IneqSimpleTBound}
\T_2 &= \P[\EvGap{v,r}  | \S'_t] \\
&\geq \E[ \P[ \psi_{\mathcal{B}_{100}(v,r)} \cap \{r-B_{1} n, \ldots, r-1 \} = \emptyset | \S'_t, \mathcal{H}_{v,r} ]  | \S'_t].
\ee 
The indices $\phi_{\mathcal{B}_{100}(v)} \cap \{ r- T, \ldots, r - 1 \}$ are, conditional on $\mathcal{H}_{v,r}$, a uniformly-generated size-$|\phi_{\mathcal{B}_{100}(v)} \cap \{ r- T, \ldots, r - 1 \}|$ subset of $\{ r-T, \ldots, r-1\}$. Since $|\phi_{\mathcal{B}_{100}(v)} \cap \{ r- T, \ldots, r - 1 \}| \leq C_{1} \log(n)^{4}$ by the condition $ \EvReg{v,r, C_{1} \log(n)^{4}}$, we have for $n$ sufficiently large
\be 
\P[ &\psi_{\mathcal{B}_{100}(v,r)} \cap \{r-B_{1}n, \ldots, r-1 \} = \emptyset | \S'_t, \mathcal{H}_{v,r} ]\\
&=  \sum_{0 \leq k \leq C_{1} \log(n)^{4} } \frac{ { T- B_{1}n \choose k } }{ { T \choose k } }  \P[ |\phi_{\mathcal{B}_{100}(v)} \cap \{ r- T, \ldots, r - 1 \}| = k | \S'_t, \mathcal{H}_{v,r} ] \\
&= \sum_{0 \leq k \leq C_{1} \log(n)^{4} } \prod_{i=0}^{B_{1}n-1} \frac{T-k-i}{T-i} \P[ |\phi_{\mathcal{B}_{100}(v)} \cap \{ r- T, \ldots, r - 1 \}| = k | \S'_t, \mathcal{H}_{v,r} ] \\
&\geq \sum_{0 \leq k \leq C_{1} \log(n)^{4} } \prod_{i=0}^{B_{1}n-1} \big( 1 - \frac{2k}{T} \big) \P[ |\phi_{\mathcal{B}_{100}(v)} \cap \{ r- T, \ldots, r - 1 \}| = k |\S'_t, \mathcal{H}_{v,r} ] \\
&\geq \prod_{i=0}^{B_{1}n-1} \big( 1 - \frac{2 C_{1} \log(n)^{4}}{n \log(n)^{4}} \big) \geq e^{-4 B_{1} C_{1}}.
\ee 
Combining this with inequality \eqref{IneqSimpleTBound} gives 
\be \label{IneqCollSizeBdType2SetFin}
\T_2 \geq  e^{-4 B_{1} C_{1}}.
\ee 
Finally, we bound the term $\T_1$. Let  
\be
\EvCov{v,r} = \Big \{ \mathcal{B}_{100}(v) \subset \bigcup_{r- B_{1}n \leq s \leq r-1} \{ v_{s} \}  \Big \}
\ee be the event that every element of $\mathcal{B}_{100}(v)$ is updated during the time interval $\{ r - B_{1} n, \ldots, r - 1 \}$. Roughly speaking, for any fixed configuration $X \in \{0,1\}^{\LL}$, we will denote by $\mathcal{E}_{v,r,X}^{(\mathrm{coll-up})}$ all of the updates `allowed' by $\EvSp{v,r, 100}, \EvGap{v,r}$ and $\EvCo{v,r}$. More precisely,  $\mathcal{E}_{v,r,X}^{(\mathrm{coll-up})}$ is the set of updates $\{ (v_{s}, p_{s}) \}_{r- B_{1} n \leq s \leq r-1}$ that have the properties:
\begin{itemize}
\item If $p_{s} \leq \frac{c}{n}$, then $v_{s} \notin \mathcal{B}_{100}(v)$. 
\item If $X_{r-B_{1}n} = X$ and this KCIP process is updated using the dynamics \eqref{EqCiRep} with update variables $\{ (v_{s}, p_{s}) \}_{r-B_{1}n \leq s \leq r-1}$, then $X_{r-1}[v] = 0$ and also there exist two neighbours $w_{1}, w_{2}$ of $v$ that are in distinct components of $G_{r-1}$ and satisfy $X_{r-1}[w_{1}] = X_{r-1}[w_{2}] = 1$.
\end{itemize}
Similarly, we will denote by $\mathcal{E}_{v,r,X}^{(\mathrm{cov-up})}$ the set of updates `allowed' by $\EvSp{v,r, 100}, \EvGap{v,r}$ and $\EvCov{v,r}$. More precisely, this is the set of updates $\{ (v_{s}, p_{s}) \}_{r-B_{1}n \leq s \leq r-1}$ that have the properties:
\begin{itemize}
\item If $p_{s} \leq \frac{c}{n}$, then $v_{s} \notin \mathcal{B}_{100}(v)$. 
\item $\mathcal{B}_{100}(v) \subset \cup_{s=r-B_{1}n}^{r-1} \{ v_{s} \}$.
\end{itemize}
We claim that if $v,r,X$ are such that $|\mathcal{E}_{v,r,X}^{(\mathrm{coll-up})} \cap \mathcal{E}_{v,r,X}^{(\mathrm{cov-up})}| \geq 1$, then in fact 
\be \label{IneqToBeExplainedInWords}
\frac{|\mathcal{E}_{v,r,X}^{(\mathrm{coll-up})} \cap \mathcal{E}_{v,r,X}^{(\mathrm{cov-up})}|}{| \mathcal{E}_{v,r,X}^{(\mathrm{cov-up})}|} \geq (2d+1)^{-2d}.
\ee 
To see this, denote by $\{ w_{i} \}_{i=1}^{2d}$ the neighbours of $v$ and by $\{ w_{i,j} \}_{j=1}^{2d}$ the neighbours of $w_{i}$.  Call an element of $\mathcal{E}_{v,r,X}^{(\mathrm{cov-up})}$ \textit{good}, if for all $1 \leq i \leq 2d$,  
\be 
\inf \{ s \, : \, s \geq r - B_{1}n, v_{s} = w_{i} \} \geq \inf_{1 \leq j \leq 2d} \inf \{ s \, : \, s \geq r - B_{1}n, v_{s} = w_{i,j} \}.
\ee
In other words, a sequence is \textit{good} if vertex $w_{i}$ is not updated until after all of the vertices $\{ w_{i,j} \}_{j=1}^{2d}$ have been updated at least once. A good sequence will not ever remove a particle from a neighbour of $v$, and the configuration $X_{r-1}|_{\mathcal{B}_{100}(v)}$ resulting from a good sequence will have only singletons. Thus, any good sequence will also be in $\mathcal{E}_{v,r,X}^{(\mathrm{coll-up})}$ if $\mathcal{E}_{v,r,X}^{(\mathrm{coll-up})}$ is not empty. Since at least one out of every $(2d+1)^{-2d}$ elements of $\mathcal{E}_{v,r,X}^{(\mathrm{cov-up})}$ is good, this observation implies inequality \eqref{IneqToBeExplainedInWords}. \par
Inequality \eqref{IneqToBeExplainedInWords} can be expressed as
\be 
\P[\EvCo{v,r} |\mathcal{F}_{t},  \EvSp{v,r, 100}, \EvGap{v,r}, \EvCov{v,r}] \geq (2d+1)^{-(2d+1)} \P[\EvCo{v,r} |\mathcal{F}_{t},  \EvSp{v,r, 100}, \EvGap{v,r}].
\ee 
We thus have 
\be \label{IneqCollSizeBdType1SetFin}
\T_1 &= \P[ \colM{r} = 1 | \mathcal{F}_{t}, \EvCo{v,r}, \EvSp{v,r, 100}, \EvGap{v,r} ] \\
&\geq \P[\colM{r} = 1 | \mathcal{F}_{t}, \EvCo{v,r}, \EvSp{v,r, 100}, \EvGap{v,r}, \EvCov{v,r} ] \P[ \EvCov{v,r} | \mathcal{F}_{t}, \EvCo{v,r}, \EvSp{v,r, 100}, \EvGap{v,r}] \\
&= 1 \times \P[ \EvCov{v,r} | \mathcal{F}_{t}, \EvCo{v,r}, \EvSp{v,r, 100}, \EvGap{v,r}] \\
&= {\P[ \EvCo{v,r} | \mathcal{F}_{t}, \EvSp{v,r, 100}, \EvGap{v,r}, \EvCov{v,r} ]}
\frac{\P[ \EvCov{v,r} | \mathcal{F}_{t}, \EvSp{v,r, 100}, \EvGap{v,r}]}{\P[\EvCo{v,r} |\mathcal{F}_{t}, \EvSp{v,r, 100}, \EvGap{v,r}]} \\
&\geq (2d+1)^{-2d} {\P[ \EvCo{v,r} | \mathcal{F}_{t}, \EvSp{v,r, 100}, \EvGap{v,r} ]}
\frac{\P[ \EvCov{v,r} | \mathcal{F}_{t}, \EvSp{v,r, 100}, \EvGap{v,r}]}{\P[\EvCo{v,r} |\mathcal{F}_{t}, \EvSp{v,r, 100}, \EvGap{v,r}]} \\
&= (2d+1)^{-2d} \P[ \EvCov{v,r} | \mathcal{F}_{t}, \EvSp{v,r, 100}, \EvGap{v,r}].
\ee 
By a monotonicity argument essentially identical to that used to prove Equation \eqref{IneqRationIrrelevantUpdates} followed by the standard `coupon-collector' bound, there exists some $B > 0$ so that
\be \label{IneqCollSizeBdType1SetFinAux}
\P[ \EvCov{v,r} | \mathcal{F}_{t}, \EvSp{v,r, 100}, \EvGap{v,r}] \geq \frac{1}{2} 
\ee
for all $B_{1} > B$ and all $n > N(B_{1})$ sufficiently large. Choosing $B_{1} = B+1$, and combining equality \eqref{IneqLemmaEL1MainDecomp} with inequalities \eqref{IneqCollSizeBdType3SetFin}, \eqref{IneqCollSizeBdType2SetFin}, \eqref{IneqCollSizeBdType1SetFin} and \eqref{IneqCollSizeBdType1SetFinAux} gives the bound
\be 
\P[\colM{r} = 1 | \mathcal{F}_{t}, \EvCo{v,r}, \EvSp{v,r, 100}] \geq \frac{23}{50} e^{-4(B+1)C_{1}} (2d+1)^{-2d} + o(n^{-1}) .
\ee 
Combining this with inequality \eqref{eqn:Mcolsparsebd}, we have 
\be 
\P[\colM{r} = 1 | \mathcal{F}_{t}, \EvCo{v,r}] \geq \delta  + o(n^{-1})
\ee 
with $\delta = \frac{23}{50} e^{-4(B+1)C_{1}} (2d+1)^{-2d}$, verifying \eqref{eqn:colllemrealbd}. By the observations made in \eqref{eqn:colMrcond}
and \eqref{eqn:Mcolsparsebd}, the conclusion in \eqref{IneqTypCompSizeMainIneq} follows immediately and the proof is finished.
\end{proof}

Recall the definition of the number of collisions $\cC_{s}$ from formula \eqref{EqDefCCs}. We will use Lemma \ref{LemmaEL1} to bound the expected change $\tilde{Y}_{t + \epsilon n^{3}} - \tilde{Y}_{t}$ in terms of  $\cC_{s}$: 
\begin{lemma} [Comparison of Number of Components to Number of Collisions] \label{LemmaCompNumCompsNumCols}
Fix $\delta = \delta(c,d) > 0$ as given by Lemma \ref{LemmaEL1}. Then for all $0 < \epsilon < \epsilon(c,d)$ sufficiently small, 
\be
 \E[\tilde{Y}_{t + \epsilon n^{3}} - \tilde{Y}_{t} \vert \mathcal{F}_{t}] &\leq  - \frac{2\delta}{3} \, \E \big[ \vert \big\{ t \leq u \leq t + \epsilon n^{3} \, : \, \colM{u} \neq 0 \big\} \vert \, \vert \mathcal{F}_{t} \big] \\
 &\hspace{2cm}+ \tilde{Y}_{t}  \frac{128 c^{2} (d+1)^{3} \epsilon}{3}(1 + o(1))  + O(1),
\ee
where the implied constants do not depend on $\epsilon$ or $n$.
\end{lemma}

\begin{proof}
Throughout the proof of this lemma, we use `update variables' $\{ p_{s} \}_{s \in \mathbb{N}}$, $\{ v_{s} \}_{s \in \mathbb{N}}$ from formula \eqref{EqCiRep}. We have
\be \label{IneqCompNumCompMainDec}
\E \big[\tilde{Y}_{t + \epsilon n^{3}} - \tilde{Y}_{t} \vert \mathcal{F}_{t}\big] &=  \E\big[\sum_{t \leq u \leq t + \epsilon n^{3} \, :  \, \colM{u} = 1} ( \tilde{Y}_{u+1} - \tilde{Y}_{u})\vert \mathcal{F}_{t}\big] \\
&\hspace{2cm}+ \E\big[\sum_{t \leq u \leq t + \epsilon n^{3} \, : \, \colM{u} \neq 1} ( \tilde{Y}_{u+1} - \tilde{Y}_{u})\vert \mathcal{F}_{t}\big]\\
&\equiv \T_1 + \T_2.
\ee
We first estimate the term $\T_2$. From Definition \ref{DefCorrCompCount}, we have that $\E[(\tilde{Y}_{u+1} - \tilde{Y}_{u}) \vert p_{u} > \frac{c}{n}, \mathcal{F}_{u}] = 0$, and so
\be \label{IneqMessySumEasyTerms}
\T_2 = \E[\sum_{t \leq u \leq t + \epsilon n^{3} \, : \, \colM{u} \neq 1}  ( \tilde{Y}_{u+1} - \tilde{Y}_{u}) \textbf{1}_{\colM{u} \neq 1} \textbf{1}_{p_{u} < \frac{c}{n}} \vert \mathcal{F}_{t}].
\ee 
By Lemma \ref{LemCalc2}, the corrected component count $\tilde{Y}_{u}$ cannot change by more than one when vertices are added and no collisions occur:
\be 
\vert (\tilde{Y}_{u+1} - \tilde{Y}_{u})  \textbf{1}_{\colM{u} \neq 1} \textbf{1}_{p_{u} < \frac{c}{n}} \vert \leq 1.
\ee 
Finally, if $\vert H \vert \leq 2$, then $\cN_{H} = 1$. In particular, adding a vertex to $X_{u}$ can only increase $\tilde{Y}_{u}$ if it connects two components or if it is added to a component that already has at least two vertices. Since there are at most $4 d \delta_{u}$ vertices adjacent to such a component,
\be 
\E[ (\tilde{Y}_{u+1} - \tilde{Y}_{u})\textbf{1}_{\colM{u} \neq 1} \textbf{1}_{p_{u} < \frac{c}{n}} \vert \mathcal{F}_{t} ] \leq \frac{4cd}{n^{2}} \E[\delta_{u} \vert \mathcal{F}_{t}].
\ee 
Combining this bound with inequality \eqref{IneqMessySumEasyTerms},
\be \label{IneqMessySumEasyTermsAppl2}
\T_2 \leq \frac{4cd}{n^{2}} \E\big[\sum_{u=t}^{t+\epsilon n^{3}} \delta_{u} \vert \mathcal{F}_{t}\big]. 
\ee 
By inequality \eqref{LongCalcOfCorrComp} and the calculation in inequality \eqref{EqLastLineLemmaCalc3}, we have   
\be 
\frac{4cd}{n^{2}} \E\big[\sum_{u=t}^{t + \epsilon n^{3}} \delta_{u} \vert \mathcal{F}_{t}\big] \leq  \tilde{Y}_{t}  \frac{128 c^{2} (d+1)^{3}}{n^{3}} \big( 8n \big(1 - \frac{1}{8n} \big)^{\epsilon n^{3}} + s \big)  +  \frac{32 cd}{n} \delta_{t} + \frac{4}{3} \big(1 - \frac{1}{8n} \big)^{\epsilon n^{3}} \delta_{t}.
\ee 
Combining this with inequality \eqref{IneqMessySumEasyTermsAppl2}, this implies
\be 
\T_2 &\leq  \tilde{Y}_{t}  \frac{128 c^{2} (d+1)^{3}}{3 n^{3}} \big( 8n \big(1 - \frac{1}{8n} \big)^{\epsilon n^{3}} + \epsilon n^{3} \big) 
+ \frac{32 cd}{n} \delta_{t} + \big(1 - \frac{1}{8n}\big)^{\epsilon n^{3}} \delta_{t}\\
&\leq   \tilde{Y}_{t}  \frac{128 c^{2} (d+1)^{3} \epsilon}{3}(1 + o(1)) + O(1). \label{IneqCombDescTermOne}
\ee 

Next, we turn to bounding $\T_1$. By Lemma \ref{LemCalc1}, it follows that
\be 
\T_1 &= \E\big[\sum_{t \leq u \leq t + \epsilon n^{3} \, : \, \colM{u} = 1} ( \tilde{Y}_{u+1} - \tilde{Y}_{u}) \vert \mathcal{F}_{t}\big] \\
&\leq -\frac{2}{3} \E \big[ \vert \big\{ t  \leq u \leq t + \epsilon n^{3} \, : \, \colM{u} = 1 \big\} \vert \, | \, \mathcal{F}_{t} \big] + 0.
\ee 
By Lemma \ref{LemmaEL1}, this implies  
\be \label{IneqCombDescTermTwoP1}
\T_1 \leq - \frac{2\delta}{3} \, \E \big[ \vert \big\{ t + n \log(n)^{4} \leq u \leq t + \epsilon n^{3} \, : \, \colM{u} \neq 0 \big\} \vert \, | \mathcal{F}_{t} \big] + o(n^{-1}).
\ee
Thus
\be 
\E[\vert  & \colT{\epsilon n^{3}}  \cap \{ t, t+1, \ldots, t + n \log(n)^{4} \} \vert \, \vert \mathcal{F}_{t}] \\
&\leq \E[\vert \{ t \leq u \leq t + n \log(n)^{4} \, : \, V_{u+1} > V_{u} \} \vert \, \vert \mathcal{F}_{t}] \\
&\leq \frac{2cd}{n^{2}} \E[\sum_{u=t}^{t+ n \log(n)^{4}} V_{u} \vert \mathcal{F}_{t}] \\
&\leq \frac{4cd(d+1)}{n^{2}} \E[\sum_{u=t}^{t+ n \log(n)^{4}} \tilde{Y}_{u} \vert \mathcal{F}_{t}] \\
&\leq \frac{4cd(d+1)}{n^{2}} \sum_{u=t}^{t + n \log(n)^{4}} \Big( \tilde{Y}_{t} \big( 1 + \frac{96 c^{2} (d+1)^{3}}{n^{3}} \big( 8n \big(1 - \frac{1}{8n} \big)^{u} + u \big) \big) +  \frac{24 cd}{n} \delta_{t} + \big(1 - \frac{1}{8n} \big)^{u} \delta_{t} \Big) \\
&\leq \tilde{Y}_{t}(1 + o(1)) \frac{\log(n)^{4}}{n} + O(1) 
\label{IneqCombDescTermTwoP2}
\ee 
where the third inequality is due to inequality \eqref{IneqSimpCorrCompCountVsPart} and the fourth inequality is due to Lemma \ref{LemmaCalc3}. Combining inequalities \eqref{IneqCombDescTermTwoP1} and \eqref{IneqCombDescTermTwoP2}, we have  
\be 
\T_1  &\leq  -\frac{2\delta}{3} \, \E \big[ \vert \big\{ t \leq u \leq t + \epsilon n^{3} \, : \, \colM{u} \neq 0 \big\} \vert \, | \, \mathcal{F}_{t}  \big] \\
&\hspace{2cm}+ \frac{2 \delta}{3} \tilde{Y}_{t}(1 + o(1)) \frac{\log(n)^{4}}{n} + O(1).
\ee 
Combining this inequality with inequalities \eqref{IneqCombDescTermOne} and \eqref{IneqCompNumCompMainDec}, we have 
\be 
\E[\tilde{Y}_{t + \epsilon n^{3}} - \tilde{Y}_{t} \vert \mathcal{F}_{t}] &\leq - \frac{2\delta}{3} \, \E \big[ \vert \big\{ t \leq u \leq t + \epsilon n^{3} \, : \, \colM{u} \neq 0 \big\} \vert \, \vert \mathcal{F}_{t} \big] \\
&\hspace{2cm}+ \frac{2 \delta}{3} \tilde{Y}_{t}(1 + o(1)) \frac{\log(n)^{4}}{n}  +  \tilde{Y}_{t}  \frac{128 c^{2} (d+1)^{3} \epsilon}{3}(1 + o(1)) + O(1)  \\
&=  - \frac{2\delta}{3} \, \E \big[ \vert \big\{ t \leq u \leq t + \epsilon n^{3} \, : \, \colM{u} \neq 0 \big\} \vert \, \vert \mathcal{F}_{t} \big] \\
&\hspace{2cm}+ \tilde{Y}_{t}  \frac{128 c^{2} (d+1)^{3} \epsilon}{3}(1 + o(1))  + O(1),
\ee 
and the proof is finished.
\end{proof}
\subsection{Color constrained Ising process}
Our next goal is to prove a lower bound on the expected number of collisions. To this end,  we define a `colored' version of the KCIP on a general finite graph $G$, which allows us to make rigorous the notion of a single particle moving and branching over time.

\begin{defn} [Color constrained Ising process]\label{DefColConIs}
Fix $t \in \mathbb{N}$, $1 \leq k \leq n$ and $x \in \{0,1\}^{G}$ so that the subgraph of $G$ induced by the vertices $\{v \in G \, : \, x[v] = 1 \}$ has exactly $k$ connected components. We define a Markov chain $\{\widehat{X}_{s}\}_{s \geq t}$ on the state space $\{0,1,\ldots,k\}^{n}$ that is closely coupled to the KCIP $\{ X_{s} \}_{s \geq t}$ started at $X_{t} = x$; in fact, we will have $X_{s}[v] = \textbf{1}_{\widehat{X}_{s}[v] \neq 0}$ for all $v \in G$ and $s \geq t$. We begin by setting the initial condition $\widehat{X}_{t}$. For a fixed $v \in G$, if $X_{t}[v] = 0$, we also set $\widehat{X}_{t}[v] = 0$. Fix an ordering of the $k$ connected components $c_{t}[1], \ldots, c_{t}[k]$ of $G_{t}$, and set $\widehat{X}_{t}[v] = i$ for all $v \in c_{t}[i]$ and $\widehat{X}_{t}[v] = 0$ for all $v \notin \cup_{i=1}^{k} c_{t}[i]$. Note that this arbitrary ordering and labelling of the components is done once, at time $t$. We do not reorder components at times $s > t$, and we will always have $\widehat{X}_{s}[v]$ in the set of labels $\{0,1,\ldots, k\}$, even if the number of components at time $s$ is not equal to $k$. Indeed, it will turn out that with probability one there exists a (random) index $i \in \{1,2,\ldots, k\}$ and time $S \geq t$ so that $\widehat{X}_{s}[v] \in \{0,i\}$ for all $s > S$.\par
 To evolve $\widehat{X}_{s}$, recall that $\{ X_{s} \}_{s \geq t}$ evolves by selecting at every time $s$ a vertex to update, and sometimes changing the label of that vertex. Whenever the labelling of a vertex $v$ is changed from 1 to 0 in $X_{s}$, the labelling of $v$ should also be changed to 0 in $\widehat{X}_{s}$. Whenever the labelling of a vertex $v$ is changed from 0 to 1 at time $s$ in $X_{s}$, choose a vertex $u_{s} \sim \mathrm{Unif} \{ w \, : \, (w,v) \in E,  \, X_{s}[w] = 1\}$ uniformly at random from the neighbours of $v$ that have a non-zero label in $\widehat{X}_{s}$, set $\widehat{X}_{s+1}[v] = \widehat{X}_{s}[u_{s}]$, and then set $\widehat{X}_{s+1}[w] = \widehat{X}_{s}[u_{s}]$ for all $w \in \mathrm{Comp}_{s}(u_{s})$. All other labels of $\widehat{X}_{s+1}$ should be the same as that of $\widehat{X}_s$. Since entire components can `flip' colours, the process $\{ \widehat{X}_{s} \}_{s \geq t}$ will sometimes have several labels change at once. 
\end{defn}
 
We now give some definitions related to the colored KCIP.  For all $s \geq t$ define $\mathrm{Comp}_{s}^{(i)} $ to be the connected components of $G_{s}$ containing only vertices $u$ satisfying $\widehat{X}_{s}[u] = i$. Define the \textit{number of vertices with color $i$} by
\be 
V_{s}^{(i)} = \sum_{u \in \LL} \textbf{1}_{\widehat{X}_{s}[u] = i}
\ee 
and the associated \textit{interference time} by
\be
\zint^{(i)} = \inf \left \{ s > t \, : \, \{ | V_{s}^{(i)} - V_{s-1}^{(i)} | > 1\} \cup \{ \big| \, | \mathrm{Comp}_{s}^{(i)} | - | \mathrm{Comp}_{s-1}^{(i)} | \, \big| > 1  \}  \right \}.
\ee
We also provide a generalization of the definition of the \textit{triple} time in formula \eqref{EqTripleTime}:
\be \label{EqTripleTimeGen} 
\ztr^{(i)} &= \inf \{ s > t \, : \, V_{s}^{(i)} \geq 3 \}.
\ee  

This allows us to state the following corollary of Lemma \ref{LemmaNateshLowerBoundOrig}: 

\begin{cor}  \label{LemmaNateshLowerBound}
Fix $i, t \in \mathbb{N}$ and $\epsilon > 0$ and assume that $V_{t}^{(i)} = 1$. Then
\be \label{IneqCompGrowthLowerBound2}
\P\big[\ztr^{(i)} - t < \min \big( \epsilon \frac{n^{3}}{6 c^{2} d (2d-1)}, \zint^{(i)} \big) \big] = O \big(\epsilon \big). 
\ee 
\end{cor}
\begin{proof}
We reduce this to the case of Lemma \ref{LemmaNateshLowerBoundOrig}. Let $v \in \LL$ be the unique vertex with $\widehat{X}_{t}[v] = i$. 
If $\sum_{u \in \LL} X_{t}[u] > 1$, we can couple $\{ X_{s} \}_{s \geq t}$ to a second KCIP $\{ X_{s}' \}_{s \geq t}$, with initial condition $X_{t}'[u] = \textbf{1}_{u = v}$ in the following way. Let $p_{s}, v_{s}$ be as in representation \eqref{EqCiRep}. It is possible to obtain a third vertex with label $i$ only by either growing the single component of color $i$ to size 3 without having a collision with any other component, or by having a collision with another component. Thus, for any constant $0 < C < \infty$, if $\ztr^{(i)} < \min \big( C, \zint^{(i)} \big)$ holds for the process $\{ X_{s} \}_{s \geq t}$, it must also hold for the process $\{ X_{s}' \}_{s \geq t}$. Thus, for the purposes of proving inequality \eqref{IneqCompGrowthLowerBound2}, it is enough to prove it in the case that $X_{t}[u] = \textbf{1}_{u=v}$; but in that case, the conclusion follows by Lemma \ref{LemmaNateshLowerBoundOrig}. 
\end{proof}
We next recall the Coalescence process (\cite{ClSu73, HoLi75}): 
\begin{defn}[Coalescence Process] \label{defCoalProcSimp}
Fix a graph $G$, $k \in \mathbb{N}$ and $q \in[0, \frac{1}{k}]$. A \textit{coalescence process on graph $G$ with $k$ initial particles and moving rate $q$} is a Markov chain $\{ Z_{s} \}_{s \in \mathbb{N}}$ on $G^{k}$. Let $O_{s} = \{ v \in G \, : \, \exists \, \, 1 \leq i \leq k \,\, \text{such that} \, Z_{s}[i] = v\}$ be the \textit{occupied sites} of $Z_{s}$. To evolve $Z_{s}$, we first choose $u_{s} \sim \mathrm{Unif}([0,1])$, $v_{s} \sim \mathrm{Unif}([O_{s}])$ and $u_{w} \sim \mathrm{Unif}([\mathcal{B}_{1}(v_{s}) \backslash \{v_{s}\}])$ and set $I_{s} = \{i \, : \, Z_{s}[i] = v_{s} \}$. If $u_{s} \leq q | O_{s} |$, then set $Z_{s+1}[i] = u_{w}$ for all $i \in I_{s}$ and set $Z_{s+1}[j] = Z_{s}[j]$ for all $j \notin I_{s}$; otherwise, set $Z_{s+1}[j] = Z_{s}[j]$ for all $j$. 
\end{defn}
\begin{remark}The coalescence process has many other names and descriptions. The construction of the coalescence process as the  `dual' to the voter process is well known (see \cite{HoLi75}). We can also view the coalescence process with $k$ initial particles as $k$ random walkers that take turns making independent simple random walk steps until a collision occurs, at which point the colliding particles `merge' into a single particle. After this collision, the coalescence process resumes with $k-1$ particles.
\end{remark}
\subsection{Coupling KCIP with coalescence process} Next we couple the KCIP with the coalescence process to obtain a lower
bound on the number of collisions $\cC_s$ (see formula \eqref{EqDefCCs}). 
\begin{lemma}\label{LemmLbNumCol}
With notation as above, there exist constants $\kappa = \kappa(\epsilon, c,d) > 0$ and $C = C(\epsilon, c, d) < \infty$ that do not depend on $n$ so that
\be 
\E[\cC_{t + \epsilon n^{3}} \vert \mathcal{F}_{t}] \geq \kappa Y_{t} - C. 
\ee 
\end{lemma}
\begin{proof}
Roughly speaking, our strategy is to show that the KCIP has `almost as many' collisions as an associated coalescence process on $\LL$. Recall the `color constrained' KCIP $\widehat{X}_{s}$ from Definition \ref{DefColConIs}. We define $\{ Q_{s}[i] \}_{s \geq t + n^{2.5}}\, 1 \leq i \leq Y_{t}$ to be a collection of $(1-\frac{cd}{n^{2}})$-lazy simple random walks on $\LL$, with starting points $\widehat{X}_{t+ n^{2.5}}[Q_{t + n^{2.5}}[i]] = i$. These walks will not be independent; we will later define a coupling of these walks to $\widehat{X}_{s}$. We will use the properties of this coupling to show that the following all hold:
\begin{enumerate}
\item For $1 \leq i \leq Y_{t}$, let $T(i)$ be the largest integer so that $Q_{s}[i] \in \{ v \, : \, \widehat{X}_{s}[v] = i \}$ holds for all $t + n^{2.5} \leq s \leq T(i)$. We show, with high probability, that $T(i)$ is `quite large' for `many' values of $i$.
\item The expected number of `near-collisions' between processes $Q_{s}[i], Q_{s}[j]$, $i \neq j$ is at least some fixed fraction of the expected number of collisions in a coalescence process $\{R_s\}_{s \geq t + n^{2.5}}$ with initial particles at  $(Q_{t +n^{2.5} }[1], \ldots, Q_{t + n^{2.5} }[Y_{t}] )$.  See Figure \ref{PicEmbeddedCoal}, showing the particles of the colored KCIP $\widehat{X}_{t}$ as well as `crossed-out' vertices representing $R_{t}$.
\item Theorem 5 of \cite{Cox89} implies that this last number is `almost' as large as $Y_{t}$. 
\item The expected number of collisions in the KCIP is at least some fixed fraction of the expected number of near-collisions.
\end{enumerate}
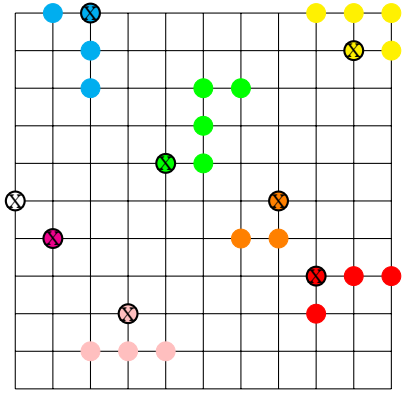
\begin{figure} 
\begin{tikzpicture}[scale=0.5]
    \draw (0, 0) grid (10, 10);
    
    \draw [thick] [fill = white] (0,5) circle [radius=.25];
    \draw (0,5) node{x};
    \draw [thick] [fill = magenta] (1,4) circle [radius=.25];
    \draw (1,4) node{x};
   \draw [cyan] [fill = cyan] (1,10) circle [radius=.25];
     \draw [thick] [fill=cyan] (2,10) circle [radius=.25];
     \draw (2,10) node{x};
\draw [cyan] [fill = cyan] (2,9) circle [radius=.25];
     \draw [cyan] [fill=cyan] (2,8) circle [radius=.25];

    \draw [pink] [fill=pink] (2,1) circle [radius=.25];
    \draw [pink] [fill=pink] (4,1) circle [radius=.25];
    \draw [pink] [fill = pink] (3,1) circle [radius=.25];
     \draw [thick] [fill=pink] (3,2) circle [radius=.25];
      \draw (3,2) node{x};
    \draw [thick][fill = green] (4,6) circle [radius=.25];
     \draw (4,6) node{x};
    \draw [green] [fill=green] (5,6) circle [radius=.25];
     \draw[green] [fill=green] (5,7) circle [radius=.25];
     \draw[green] [fill=green] (5,8) circle [radius=.25];
    \draw [green] [fill =green] (6,8) circle [radius=.25];
     \draw [orange][fill=orange] (6,4) circle [radius=.25];
     \draw [orange] [fill = orange] (7,4) circle [radius=.25];
      \draw [thick][fill=orange] (7,5) circle [radius=.25];
       \draw (7,5) node{x};
        \draw [red] [fill = red] (8,2) circle [radius=.25];
     \draw [thick] [fill = red] (8,3) circle [radius=.25];
           \draw (8,3) node{x};
      \draw [red] [fill = red] (9,3) circle [radius=.25]; 
      \draw [red] [fill = red] (10,3) circle [radius=.25];
      
      \draw [yellow] [fill = yellow] (10,10) circle [radius=.25]; 
      \draw [yellow] [fill = yellow] (9,10) circle [radius=.25];
      \draw [yellow] [fill = yellow] (8,10) circle [radius=.25]; 
      \draw [yellow] [fill = yellow] (10,9) circle [radius=.25];
      \draw [thick] [fill = yellow] (9,9) circle [radius=.25];
      \draw (9,9) node{x};
\end{tikzpicture}
\caption{Embedded Coalesecence Process}
\label{PicEmbeddedCoal}
\end{figure}

We now give the formal argument. Assume $n > c$ and fix an index $1 \leq i \leq Y_{t}$. Let 
\be
\tS{i} = \inf \{ s \geq t \, : \, \forall \, w_{1}, w_{2} \in \LL, \,  \{ \widehat{X}_{s}[w_{1}] = \widehat{X}_{s}[w_{2}] = i \} \implies \{ (w_{1}, w_{2}) \notin E \} \}
\ee
 be the first time that no two particles colored $i$ are adjacent. If the set $\{ v \, : \, \widehat{X}_{\tS{i}}[v] = i \}$ is non-empty, choose a vertex $v$ from that set and let $Q_{\tS{i}}[i] = v$. Although we have not yet described the evolution of $\{ Q_{s}[i] \}_{s \geq \tS{i}}$, we define the associated random time 
\be 
\tTr{i} = \inf \{ s > \tS{i} \, : \, \text{the component of } G_{s} \text{ containing } Q_{s}[i] \text{ has at least three elements.} \}.
\ee

Define the \textit{lifetime} $\ell_{s}[v]$ of a particle at $v$ with $\widehat{X}_{s}[v]=i$ to be 
\be
\ell_{s}[v] = \sup \{ T \geq s \, : \, \forall \, s < u \leq T, \, \, \widehat{X}_{u}[v] = i \, \, \mathrm{and} \, \, v_{u} \neq v  \}
\ee
when that set is non-empty, and $\ell_{s}[v] = 0$ otherwise. 
 Again, without having defined the evolution of $\{ Q_{s}[i] \}_{s \geq \tS{i}}$, we define the decoupling time of colour $i$ to be the minimum of the first time that the particle at $Q_{s}[i]$ does not have colour $i$ and the first time that there is a size-three component of colour $i$: 
 \be \label{EqDecoupDefMain}
 \tDe{i} = \min \big( \tTr{i}, \inf \{ s \geq \tS{i} \, : \, \widehat{X}_{s}[Q_{s}[i]] \neq i \}  ).
 \ee We now describe our coupling of the processes $\{ Q_{s}[i] \}$ to $\{ \widehat{X}_{s} \}$. For $\tS{i}  \leq s < \tDe{i} $, we evolve $Q_{s}[i]$ by always setting \be \label{EqDefFunnyQProc}
Q_{s+1}[i] = \mathrm{argmax}_{ v \in \mathrm{Comp}_{s}(Q_{s}[i])} \ell_{s}[v].
\ee 

When $s \geq \tDe{i}$, we are not interested in $Q_{s+1}$ and do not make use of its evolution beyond this time. Finally, define the `near-collision' time 
\be 
\tau_{\mathrm{near}}^{(i)} = \inf \{ s \geq \tS{i} \, : \, \text{there exists an element of } X_{s} \text{ at distance 2 from } Q_{s}[i] \}. 
\ee 
First, we claim that $\tDe{i} \geq \min(\tTr{i},\tau_{\mathrm{near}}^{(i)})$.  This is because, for all of the color-$i$ vertices near $Q_{s}[i]$ to be removed, there would have to be a collision involving color-$i$ vertices near $Q_{s}[i]$, and thus also a near-collision. Observe that, conditioned on $\tS{i}$, the distribution of the path 
\be 
\{ Q_{s}[i] \}_{s = \tS{i}}^{\min(\tTr{i} , \tau_{\mathrm{near}}^{(i)})-1}
\ee is equal to that of a simple random walk on $G$ with holding probability $1 - \frac{cd}{n^{2}}$.  \par
We now couple a coalescence process $\{R_{s} \}_{s \geq t + n^{2.5}}$ with $Y_{t}$ initial particles and moving rate $\frac{cd}{n^{2}}$ to $\{Q_{s} \}_{s \geq t + n^{2.5}}$. We begin by setting the initial conditions $R_{t + n^{2.5}}$. If $\tS{i} \leq n^{2.5}$, we set $R_{t+n^{2.5}}[i] = Q_{t+n^{2.5}}[i]$. Otherwise, choose $R_{t+n^{2.5}}[i]$ uniformly at random from among all vertices not already containing a particle. Let $\mathcal{D}$ denote the pairs of indices $(i,s)$ that satisfy 
\be \label{IneqGoodIndicesCoalProc}
\tau_{\mathrm{start}}^{(i)} \leq n^{2.5} \leq s \leq \min(\tau_{\mathrm{triple}}^{(i)}, \tau_{\mathrm{near}}^{(i)}).
\ee 
For $(i,s) \in \mathcal{D}$, we set $R_{s}[i] = Q_{s}[i]$. For $(i,s) \notin \mathcal{D}$, we choose $R_{s+1}[i]$ conditional on $R_{s}$ and $\{ R_{s+1}[j] \}_{ (j,s) \in \mathcal{D}}$ and independently of all other random variables being discussed. We now argue that such an extension from the indices $(i,s) \in \mathcal{D}$ to all indices $1 \leq i \leq Y_{t}$, $t + n^{2.5} \leq s \leq t + \epsilon n^{3}$ is possible. To see this, it is sufficient to check that at each $(i,s) \in \mathcal{D}$, we have:
\begin{itemize}
\item $Q_{s}[i]$ evolves as a simple random walk with holding probability $1 - \frac{cd}{n^{2}}$ (\textit{i.e.}, each particle moves as a simple random walk). 
\item If $(j,s) \in \mathcal{D}$ for some $j \neq i$, at least one of $Q_{s+1}[i] = Q_{s}[i]$ or $Q_{s+1}[j] = Q_{s}[j]$ must hold (\textit{i.e.}, two particles can't move at once).
\item If $(j,s) \in \mathcal{D}$ for some $j \neq i$, $Q_{s}[i] = Q_{s}[j]$ cannot hold (\textit{i.e.}, no two particles that are being conditioned on have collided). 
\end{itemize}

These three conditions, taken together, characterize the one-step dynamics of a subset of particles evolving as a coalescence process.

 This completes the definition of our coupling of $\{ R_{s} \}_{t + n^{2.5} \leq s \leq \epsilon n^{3}}$ to  $\{ \widehat{X}_{s} \}_{t + n^{2.5} \leq s \leq \epsilon n^{3}}$. It has the critical property that, for $t + n^{2.5} \leq s < \tau_{\mathrm{decoupling}}^{(i)}$, we have  $\widehat{X}_{s}[R_{s}[i]] = i$. \par
Say that color $1 \leq i \leq Y_{t}$ has a \textit{near-collision} between times $t_{1}$ and $t_{2}$ if there exists a time $t_{1} \leq s \leq t_{2}$ and a pair of vertices $u, v$ so that $\widehat{X}_{s}[u] = i$, $X_{s}[v] = 1$, $u,v$ are not in the same component of $G_{s}$ and $|u-v| = 2$. Say that color $i$ has \textit{coalesced by time $t_{1}$} if $R_{s}[i] = R_{s}[j]$ for some $j \neq i$ and $s \leq t_{1}$. The color $i$ has a near-collision between times $t$ and $t + \epsilon n^{3}$ unless at least one of the following four events occurs:
\begin{enumerate}
\item $\mathcal{A}_{1,n^{2.5}}^{(i)}$, where for $0 \leq u \leq n^{2.5}$ we define $\mathcal{A}_{1,u}^{(i)}$ to be the event that $\tS{i} > t + u$ and $i$ has no near-collisions between time $t$ and $t + u$. 
\item $\mathcal{A}_{2}^{(i)}$: $\tS{i} \leq t + n^{2.5}$ but color $i$ has not coalesced by time $t + \epsilon n^{3}$. 
\item $\mathcal{A}_{3}^{(i)}$: $\tTr{i} < \tau_{\mathrm{near}}^{(i)} < \epsilon n^{3}$.
\item $\mathcal{A}_{4}^{(i)}$: $\tS{i} \leq t + n^{2.5}$ and color $i$ has coalesced, but for all pairs $s,j$ so that $t \leq s \leq \epsilon n^{3}$ and $R_{s}^{(i)} = R_{s}^{(j)}$, we also have $\tau_{\mathrm{decoupling}}^{(j)} < \tau_{\mathrm{near}}^{(i)}$.
\end{enumerate}
We now bound the probabilities of these four events. Noting that color $i$ has a near-collision if there are ever two components of color $i$ (since the two components must be at distance exactly 2 when they are first separated), we have by essentially the same calculation as inequality \eqref{IneqDiffDec}:
\be 
\E[V_{t+s+1}^{(i)} \textbf{1}_{\mathcal{A}_{1,s+1}^{(i)}} | \mathcal{F}_{t+s} ] &\leq V_{t+s+1}^{(i)} - \frac{V_{t+s}^{(i)}}{n} (1 - \frac{c}{n}) + \frac{2d V_{t+s}^{(i)}}{n} \frac{c}{n} \\
& \leq V_{t+s}^{(i)} (1 - \frac{1}{n}(1 - \frac{c(2d+1)}{n}) ).
\ee 
Thus, for $n > 2c(2d+1)$, we iterate and find
\be 
\E[V_{t +  n^{2.5}}^{(i)} \textbf{1}_{\mathcal{A}_{1,n^{2.5}}^{(i)}} | \mathcal{F}_{t}] \leq n e^{-n}.
\ee 
Since $V_{t+s}^{(i)}$ is at least 1 for all $t+s \leq \tau_{\mathrm{near}}^{(i)}$, applying Markov's inequality gives
\be \label{IneqManyCollFinal1}
\P[\mathcal{A}_{1,n^{2.5}}^{(i)} | \mathcal{F}_{t}] \leq \E[V_{t + n^{2.5}}^{(i)} \textbf{1}_{\mathcal{A}_{1,n^{2.5}}^{(i)}} | \mathcal{F}_{t}] \leq  n e^{-n}.
\ee 
Next, by Theorem 5 of \cite{Cox89}, there exists a constant $C = C(\epsilon, c,d)$ that does not depend on $n$ so that
\be \label{IneqManyCollFinal2}
\E[\sum_{i} \textbf{1}_{\mathcal{A}_{2}^{(i)}}] \leq C
\ee 
for all $n$ sufficiently large. By Corollary \ref{LemmaNateshLowerBound},
\be \label{IneqManyCollFinal3}
\E[\sum_{i} \textbf{1}_{\mathcal{A}_{3}^{(i)}}] = O(\epsilon) Y_{t}.
\ee 
To bound $\mathcal{A}_{4}^{(i)}$, we consider the collection of colors $j$ that are given in the definition of $\mathcal{A}_{4}^{(i)}$. We say that these colors are involved in an `unrecorded collision,' as the collisions described in event $\mathcal{A}_{4}^{(i)}$ do not contribute to our count of the total number of near-collisions of the KCIP. Observe that any particle $R_{s}[j]$ that is involved in such an unrecorded collision at time $s$ necessarily coalesces with another particle $R_{s}[j']$ during the course of the collision, and also must have $\tDe{j} \leq s$. Thus, each decoupled particle can be involved in only one unrecorded collision before being merged with another particle, and so 
\be \label{IneqManyCollFinal4}
\sum_{i} \textbf{1}_{\mathcal{A}_{4}^{(i)}} \leq 2(\sum_{i} \textbf{1}_{\mathcal{A}_{1, n^{2.5}}^{(i)}}  + \sum_{i} \textbf{1}_{\mathcal{A}_{2}^{(i)}}  + \sum_{i} \textbf{1}_{\mathcal{A}_{3}^{(i)}}).
\ee 
 Denote by $\near_{s}$ the total number of near-collisions between times $t$ and $s$. Combining inequalities \eqref{IneqManyCollFinal1}, \eqref{IneqManyCollFinal2}, \eqref{IneqManyCollFinal3} and \eqref{IneqManyCollFinal4}, we have that
\be \label{IneqLbNumColPrettyMuchDone}
\E[\near_{t + \epsilon n^{3}}] \geq Y_{t}(1 - O(\epsilon)) - O(1).
\ee 
Finally, we must relate the number of near-collisions $\near_{t + \epsilon n^{3}}$ to the number of collisions $\cC_{t + \epsilon n^{3}}$: 

\begin{prop}  \label{propNearColImpCol}
Fix $\epsilon > 0$. There exists $\kappa = \kappa(c,d) > 0$ so that for all $n$ sufficiently large,
\be 
\E[\cC_{t + \epsilon n^{3} + 4 n^{2.5}}] \geq \kappa \E[\near_{t + \epsilon n^{3}}].
\ee 
\end{prop}
\begin{proof}[Proof of Proposition \ref{propNearColImpCol}]
This amounts to checking that, once two components $C_{1}, C_{2}$ of $G_{s}$ have vertices $w_{1} \in C_{1}, \,w_{2} \in C_{2}$ at distance 1, there is a positive probability that all three of the following (purely local) events occur:
\begin{itemize}
\item The particles at $w_{i}$ survive longer than any of their neighbours.
\item A particle is added to the common neighbour $v$ of $w_{1}, w_{2}$ before any other particle is added to the set $\mathcal{B}_{1}(w_{1}) \cup \mathcal{B}_{1}(w_{2})$.
\item Some particle is added to the set $\mathcal{B}_{1}(w_{1}) \cup \mathcal{B}_{1}(w_{2})$ before time $s + 4 n^{2.5}$.
\end{itemize}
The first two events are purely local, and the last occurs with probability at least 
\newline ${1 - (1 - \frac{c}{n^{2}})^{n^{2.5}} \approx 1 - e^{-c \sqrt{n}}}$. Thus all of these events have probability bounded away from zero, and the proof is finished.
\end{proof}

Combining Proposition \ref{propNearColImpCol} with inequality \eqref{IneqLbNumColPrettyMuchDone}, and noting that $n^{2.5} = o(n^{3})$, completes the proof.
\end{proof}
Now, we are finally ready to give the proof of Theorem \ref{LemmaContractionEstimate},
establishing a drift condition for $V_t$. 
\begin{proof}[Proof of Theorem \ref{LemmaContractionEstimate}]
Recall the definitions of $\delta > 0$ from Lemma \ref{LemmaCompNumCompsNumCols} and $\kappa > 0$ from Lemma \ref{LemmLbNumCol}. Then
\be 
\E[V_{t + \epsilon n^{3}} | \mathcal{F}_{t}] &= \E[V_{t+ \epsilon n^{3}} - {Y}_{t+\epsilon n^{3}} | \mathcal{F}_{t}] + \E[{Y}_{t+\epsilon n^{3}} - \tilde{Y}_{t + \epsilon n^{3}}] + \E[\tilde{Y}_{t + \epsilon n^{3}} - \tilde{Y}_{t} | \mathcal{F}_{t}] +  \tilde{Y}_{t}\\
&\leq 4cd(1 + o(1)) + \E[{Y}_{t+\epsilon n^{3}} - \tilde{Y}_{t + \epsilon n^{3}} | \mathcal{F}_{t}] + \E[\tilde{Y}_{t + \epsilon n^{3}} - \tilde{Y}_{t} | \mathcal{F}_{t}] +  \tilde{Y}_{t}\\ 
&\leq 8cd(1 + o(1)) + \E[\tilde{Y}_{t + \epsilon n^{3}} - \tilde{Y}_{t} | \mathcal{F}_{t}] +  \tilde{Y}_{t}\\ 
&\leq -\frac{2 \delta}{3} \E[\cC_{t + \epsilon n^{3}}  \vert \mathcal{F}_{t}] +  O(\epsilon \tilde{Y}_{t}) + O(1) + Y_{t} \\
&\leq \big( 1 - \frac{2 \delta \kappa}{3} \big) Y_{t} + O( \epsilon \tilde{Y}_{t}) + O(1) \\
&\leq  \big( 1 - \frac{2 \delta \kappa}{3} + O(\epsilon) \big) V_{t} + O(1),
\ee 
where the first inequality comes from Lemma \ref{LemmCompPartNumComp}, the second inequality comes from inequality \eqref{IneqRev1Star1DeltaBd} and a second application of Lemma \ref{LemmCompPartNumComp},  the third inequality comes from Lemma \ref{LemmaCompNumCompsNumCols}, the bound in the fourth inequality comes from 
Lemma \ref{LemmLbNumCol}, and the final inequality comes from the  bound $Y_{s} \leq \tilde{Y}_s \leq V_{s}$ (see \eqref{IneqRev1Star1DeltaBd}).  \par 
Fixing $\epsilon > 0$ sufficiently small, then, there exists some constant $C > 0$ so that 
\be 
\E[V_{t + \epsilon n^{3}} \vert \mathcal{F}_{t}] \leq  \big( 1 - \frac{ \delta \kappa}{2} \big) V_{t} + C.
\ee 
Write $\alpha = \frac{\delta \kappa}{2}$. Iterating, we have
\be 
\E[V_{t+ k \epsilon n^{3}}] &\leq \E[\E[\ldots \E[ V_{t+ k \epsilon n^{3}} \vert \mathcal{F}_{t + (k-1) \epsilon n^{3}}] \ldots \vert \mathcal{F}_{t + \epsilon n^{3}} ] \vert \mathcal{F}_{t}] \\
&\leq \E[\E[\ldots \E[  \big( 1 - \alpha \big) V_{t+ (k-1) \epsilon n^{3}}  + C \vert {t + (k-2) \epsilon n^{3}} ] \ldots \vert {t +  \epsilon n^{3}} ] \vert \mathcal{F}_{t}] \\
&\leq  \big( 1 - \alpha \big)^{k} V_t + C_G
\ee 
for some constant $C_G$ and the proof is finished.
\end{proof}
\section{Excursion Lengths of KCIP} \label{SecExcLength}
Fix $\epsilon_0$ small enough so that Theorem \ref{LemmaContractionEstimate}
applies and set 
\be \label{eqn:kmax}
k_{\mathrm{max}} = 4 {C_G \over \alpha}
\ee
where $C_G,\alpha$ are as defined in inequality \ref{eqn:conddrift} of Theorem \ref{LemmaContractionEstimate}.
 In this section, the drift condition for $V_t$ obtained in Theorem \ref{LemmaContractionEstimate} will be used to show:
\begin{enumerate}
\item The distribution of the first hitting time of $\cup_{1 \leq k \leq k_\mathrm{max}}\Omega_{k}$ is $O(n^{3} \log(n))$, uniformly in the starting point $X_{0}$ (see inequality \eqref{IneqLemmaTechLem2Start}).
\item  $T \gg n^{3} \log(n)$ implies that 
\be
\sum_{t \leq T} \textbf{1}_{X_{t} \in \cup_{1 \leq k \leq k_\mathrm{max}} \Omega_{k}} \gg n^3 \log(n)
\ee
with high probability, uniformly in the starting point $X_{0}$ (see Corollary \ref{CorMainDriftResult}).
\end{enumerate}
Items (1) and (2) above thus provide strong bounds for the occupation times of 
KCIP on $\Omega_k$, uniformly in $k\leq k_{\max}$.
We start with the following elementary technical lemma.
\begin{lemma} \label{LemmaTechLem1}
Fix $0 < \beta < 1$ and $0 < \gamma < \infty$. Consider a stochastic process $\{J_{t} \}_{t \in \mathbb{N}}$ on $\mathbb{N}$ with associated filtration $\mathcal{J}_{t}$ that satisfies the drift condition 
\be
\E[J_{s+1} \vert \mathcal{J}_{s}] \leq (1 - \beta) J_{s} + \gamma  
\ee
for all $s \in \mathbb{N}$. Let $Z_{1}, Z_{2}, \ldots$ be an i.i.d. sequence of random variables with geometric distribution and mean $\frac{2}{\beta}$. If $J_{0} \leq \frac{4 \gamma}{\beta}$, for all $T \in \mathbb{N}$ we have
\be 
\P \big[\sum_{s=0}^{T} \mathbf{1}_{J_{s} < \frac{4 \gamma}{\beta}} < C \big] \leq \P \big[\sum_{i=1}^{C} Z_{i} > T \big].
\ee 
\end{lemma}

\begin{proof}
Assume $J_{0} < \frac{4 \gamma}{\beta}$ and let $\tRe = \inf \{ s > 0 \, : \, J_{s} < \frac{4 \gamma}{\beta} \}$. We have
\be 
\E[J_{s} \mathbf{1}_{\tRe \geq s}] \leq  \big(1 - \frac{\beta}{2} \big)^{s} \frac{4 \gamma}{\beta}, \\
\ee 
and so by Markov's inequality
\be  \label{IneqExpDom0}
\P[\tRe > s ] &\leq \P[J_{s} \textbf{1}_{\tRe \geq s} > \frac{4 \gamma}{\beta}]  
\leq \big(1 - \frac{\beta}{2} \big)^{s}.
\ee
Define $t_{0} = 0$ and $t_{i+1} = \inf \{ s > t_{i} \, : \, J_{s} \leq \frac{4 \gamma}{\beta} \}$. By inequality \eqref{IneqExpDom0},
\be \label{IneqExpDom}
\P[t_{i+1} - t_{i} > s | \{ t_{j} \}_{j \leq i} ] \leq  \big(1 - \frac{\beta}{2} \big)^{s}.
\ee  
We have
\be 
\P[\sum_{s=0}^{T} \textbf{1}_{J_{s} < \frac{4 \gamma}{\beta}} < C] &\leq \P[t_{C} > T] 
= \P[\sum_{i=1}^{C} (t_{i} - t_{i-1}) > T]. 
\ee 
Inequality \eqref{IneqExpDom} implies that the distribution of $(t_{i+1} - t_{i})$ is (conditionally on $\{ t_{j} \}_{j \leq i }$) stochastically dominated by a geometric distribution with mean $\frac{2}{\beta}$; this completes the proof.
\end{proof}

Define the set  
\be \label{eqn:CalG}
\ck = \big \{x \in \{0,1\}^{\LL} \, : \, \sum_{v \in \LL} x[v] \leq k_\mathrm{max} \big \}.
\ee
We apply Lemma \ref{LemmaTechLem1} to the KCIP on $\LL$, with $J_s = V_s$ and $\mathcal{J}_t = \mathcal{F}_t$ to obtain the following corollary.
\begin{cor} \label{CorMainDriftResult}
Fix $\epsilon$ sufficiently small so that Theorem \ref{LemmaContractionEstimate} applies, and let $\alpha,C_G$ be as in Theorem \ref{LemmaContractionEstimate}. For fixed $C_{2}$ and $C_{1} > \frac{16}{\alpha} C_{2}$ sufficiently large, all $n > N(c,d)$ sufficiently large, and any starting point $X_{0} \in \Omega$,
\be 
\P[\sum_{t=0}^{ C_{1} n^{3} \log(n)} \mathbf{1}_{X_{t} \in \ck} < C_{2} n^{3} \log(n)] = O(n^{-5}).
\ee 
\end{cor}
\begin{proof}

Let $\tau_{\mathrm{start}} = \inf \{ t > 0 \, : \, X_{t} \in \ck \}$ and fix $k \in \mathbb{N}$. By Theorem \ref{LemmaContractionEstimate},
\be 
\E[V_{k \epsilon n^{3}} \textbf{1}_{\tau_{\mathrm{start}} > k \epsilon n^{3}} ] \leq \big(1 - \frac{1}{2} \alpha \big)^{k} V_{0},
\ee 
and so by Markov's inequality and the trivial bound that $V_{t} \leq n$ for all $t$,
\be \label{IneqLemmaTechLem2Start}
\P[\tau_{\mathrm{start}} > k \epsilon n^{3}] &\leq \P[V_{k \epsilon n^{3}} \textbf{1}_{\tau_{\mathrm{start}} > k \epsilon n^{3}} > 1 ]\\
&\leq n \big(1 - \frac{1}{2} \alpha \big)^{k}. 
\ee 
Fix $T \in \mathbb{N}$ and let  $\{ Z_{i}\}_{i \in \mathbb{N}}$ be an i.i.d. sequence of random variables with geometric distribution and mean $\frac{2}{\alpha}$ . By inequality \eqref{IneqLemmaTechLem2Start}, the Markov property and Lemma \ref{LemmaTechLem1},
\be 
\P[ & \sum_{t=0}^{ C_{1} n^{3} \log(n)} \textbf{1}_{X_{t} \in \ck} > C_{2} n^{3} \log(n)] \geq \P[\sum_{t=0}^{ C_{1} n^{3} \log(n)} \textbf{1}_{X_{t} \in \ck} > C_{2}  n^{3} \log(n) | \tau_{\mathrm{start}} < T] \P[\tau_{\mathrm{start}} < T] \\
& = \P[\tau_{\mathrm{start}} < T]  \sum_{t=0}^{T} \P[\sum_{t=0}^{ C_{1} n^{3} \log(n)} \textbf{1}_{X_{t} \in \ck} > C_{2} n^{3} \log(n) | \tau_{\mathrm{start}} = t] \P[\tau_{\mathrm{start}} = t | \tau_{\mathrm{start}} \leq T]\\
&\geq \big(1 - n \big(1 - \frac{1}{2} \alpha \big)^{\lfloor\frac{T}{\epsilon n^{3} } \rfloor} \big)  \sum_{t=0}^{T} \P \big[\sum_{i=1}^{C_{2} n^{3} \log(n)} Z_{i} \leq C_{1} n^{3} \log(n) - t \big] \P[\tau_{\mathrm{start}} = t | \tau_{\mathrm{start}} \leq T] \\
&\geq \big(1 - n \big(1 - \frac{1}{2} \alpha \big)^{\lfloor\frac{T}{\epsilon n^{3}} \rfloor} \big)   \P \big[\sum_{i=1}^{C_{2} n^{3} \log(n)} Z_{i} \leq C_{1} n^{3} \log(n) - T \big].
\ee 
Choosing $T = \lfloor \frac{C_{1}}{2} n^{3} \log(n) \rfloor$, we have for $C_{1}$ sufficiently large (\textit{e.g.}, larger than $200 \alpha$ suffices) that
\be 
\P[\sum_{t=0}^{ C_{1} n^{3} \log(n)} \textbf{1}_{X_{t} \in \ck} > C_{2} n^{3} \log(n)] \geq \big(1 - n^{-5}  \big) \big( 1 - e^{- \frac{C_{1} \alpha}{16} n^{3} \log(n)} \big),
\ee  
where the second part of the above inequality follows from a standard concentration inequality for geometric random variables. This completes the proof. 
\end{proof}
\subsection{Bounds for the collision times of coalescent process}
Recall $k_{\mathrm{max}}$ from Equation \eqref{eqn:kmax} and the set $\Omega_{k}$ from formula \eqref{EqDefOmegaK}. In this section we obtain estimates for the collision times for a coalescence process (see Definition \ref{defCoalProcSimp}) started in $\Omega_k$. All of our bounds are based on soft arguments and are immediate consequences of results from \cite{Cox89}. These bounds will be used for obtaining estimates of occupation times of KCIP in $\Omega_k$ in the next section. \par
Let $\{ Z_{t} \}_{t \in \mathbb{N}}$ be a coalesence process on $\Lambda(L,d)$ with $k_{\max}$ particles and moving rate $q = {cd \over n^2}$. Define the process $W_t$ on $\{0,1\}^{\LL}$:
\be \label{EqDefAltCoalProcRep}
W_{t}[v] = \textbf{1}_{\exists i: Z_{t}[i] = v}.
\ee
$W_{t}$ will often be referred to as `the' coalescence process, as it is a Markov chain and $\{ Z_{t} \}_{t \in \mathbb{N}}$ can be reconstructed (up to permutation of labels) from $\{ W_{t} \}_{t \in \mathbb{N}}$.  We give some notation related to the `skeletons' of our processes of interest. Define the sequence of times $\phi_{0} = 0$ and 
\be \label{EqDefCoalProcMoveTimes}
\phi_{i+1} = \inf \{ t > \phi_{i} \, : \, W_{t} \neq W_{\phi_{i}} \}. 
\ee 
These are the times that $\{ W_{t} \}_{t \in \mathbb{N}}$ changes. \begin{remark} \label{RmkFastCoalProc}
Let $\{ \phi_{i}'\}_{i \geq 0 }$ be a sequence of i.i.d. geometric random variables with mean 4 and define $\{W_{t}'\}_{t \geq 0}$ by
\be 
W_{t}' = W_{\phi_{i}}
\ee 
for $t$ satisfying $\sum_{j=0}^{i} \phi_{j}' \leq t < \sum_{j=0}^{i+1} \phi_{j}'$. The process $\{ W_{t}' \}_{t \geq 0}$ 
is still a coalescence process in the sense of \cite{Cox89}. Furthermore, 
\be \label{EqFastCoalRel}
\{ W_{\phi_{i}} \}_{i \in \mathbb{N}} \stackrel{D}{=} \{ W_{\phi_{i}'}' \}_{i \in \mathbb{N}}
\ee 
and 
\be 
\E[\phi_{1}] = \frac{n^{2}}{4 k_{\mathrm{max}} cd} \E[\phi_{1}'].
\ee 
\end{remark}
Define the first collision time as 
\be \label{eqn:colltime}
\tCo = \inf \{ t \, : \, \vert W_{t} \vert < |W_0|\}.
\ee
For the remainder of this section, define for $\zeta > 0$:
\be \label{DefNiceStarts}
\mathcal{G}_{\zeta}^{(n)} = \{ w \in \{0,1\}^{\Lambda(L,d)} \, : \, \inf_{u,v \, : \, w[u] = w[v] = 1} | u - v| > \zeta \}.
\ee  

We show that, for $\zeta$ sufficiently large, collision times are not `too small' when started from $\mathcal{G}_{\zeta}^{(n)} \cap \Omega_k$ : 
\begin{lemma} \label{LemSoftCoalBound1}
Let $\{Z_{t} \}_{t \in \mathbb{N}}$ be a coalescence process with $1 < k \leq k_{\mathrm{max}}$ initial particles on graph $G = \Lambda(L,d)$ and let $\{ W_{t} \}_{t \in \mathbb{N}}$ be defined as in formula \eqref{EqDefAltCoalProcRep}. Then, for all $0 < \delta < 1$, there exist $\epsilon = \epsilon(c,d, k_{\mathrm{max}}, \delta) > 0$ and $C = C(c,d,k_{\mathrm{max}}, \delta)$ so that, for all $n > N(c,d,k_{\mathrm{max}}, \delta)$,
\be [IneqSoftCoalBound1Res]
\P[\tCo < \epsilon n^{3} \vert W_{0} = w  ] &< 1 - \delta, \\
\P[\tCo < \phi_{\epsilon n} | W_{0} = w] &< 1 - \delta
\ee 
uniformly in $1 < k \leq k_{\mathrm{max}}$ and $w \in \mathcal{G}_{C}^{(n)} \cap \Omega_{k}$.
\end{lemma}
\begin{proof}
We begin by proving the first half of inequality \eqref{IneqSoftCoalBound1Res}.
By Theorem 5 of \cite{Cox89}, for any fixed $\epsilon > 0$, any sequence $\alpha_{n}$ satisfying $\lim_{n \rightarrow \infty} \alpha_{n} = \infty$, and any sequence $w^{(n)} \in \mathcal{G}_{\alpha_{n}}^{(n)}$, we have 
\be  
\lim_{n \rightarrow \infty} \P[\tCo < \epsilon n^{3} \vert W_{0} = w^{(n)} ] = f(\epsilon, c, d, k) < 1
\ee 
for some explicit function $f$ that satisfies $\lim_{\epsilon \rightarrow 0 } f(\epsilon, c, d, k) = 0$ for $d \geq 3$ and all $c,k$. {Defining $f(\epsilon,c,d) = \max_{2 \leq k \leq k_{\max}} f(\epsilon,c,d,k)$, this implies }
\be  \label{IneqCoxSimpBound1}
\lim_{n \rightarrow \infty} \P[\tCo < \epsilon n^{3} \vert W_{0} = w^{(n)} ] \leq f(\epsilon, c, d) < 1
\ee 
{and that $\lim_{\epsilon \rightarrow 0 } f(\epsilon, c, d) = 0$ for all  $d \geq 3$ and all $c$.}
\par
The remainder of the argument is a proof by contradiction.  Fix $\delta > 0$ and choose $\epsilon > 0$ so that $f(\epsilon,c,d) \leq \frac{1-\delta}{2}$. Assume that inequality \eqref{IneqSoftCoalBound1Res} is false. Then for all $C> 0$, there exists a strictly increasing sequence of integers $\{ n_{i} = n_{i}(C) \}_{i \in \mathbb{N}}$  so that
\be
\sup_{w^{(n_{i})} \in \mathcal{G}_{C}^{(n_{i})} } \P[\tCo < \epsilon n_{i}^{3} \vert W_{0} = w^{(n_{i})}] \geq 1 - \delta.
\ee
Let $\{ C_{j} \}_{j \in \mathbb{N}}$ be a sequence satisfying $\lim_{j \rightarrow \infty} C_{j} = \infty$, and for each $j$ let the increasing sequence of integers $\{ n_{i,j} = n_{i}(C_{j}) \}_{i \in \mathbb{N}}$ satisfy
\be
\sup_{w^{(n_{i,j})} \in \mathcal{G}_{C_{j}}^{(n_{i,j})}} \P[\tCo < \epsilon n_{i,j}^{3} \vert W_{0} = w^{(n_{i,j})}] \geq 1 - \delta.
\ee
Then the diagonal sequence $n_{i,i}$ satisfies
\be 
\mathrm{liminf}_{i \rightarrow \infty} \sup_{w^{(n_{i,i})} \in \mathcal{G}_{C_{i}}^{(n_{i,i})}} \P[\tCo < \epsilon n_{i,i}^{3} \vert W_{0} = w^{(n_{i,i})}] \geq 1 - \delta.
\ee 
Since {$f(\epsilon,c,d) \leq \frac{1-\delta}{2}$} by assumption, this contradicts equality \eqref{IneqCoxSimpBound1}, completing the proof of the first half of inequality \eqref{IneqSoftCoalBound1Res}. The proof of the second half of inequality \eqref{IneqSoftCoalBound1Res} is essentially identical; simply follow the same steps for the sped-up process $\{W_{t}'\}_{t \in \mathbb{N}}$ described in Remark \ref{RmkFastCoalProc} and use formula \eqref{EqFastCoalRel} to relate this back to $\{W_{t}\}_{t \in \mathbb{N}}$. 
\end{proof}
Lemma \ref{LemSoftCoalBound1} has the following strengthening as an immediate corollary: 

\begin{cor} \label{CorSoftCoalBound2}
There exist $\epsilon = \epsilon(c,d, k_{\mathrm{max}}) > 0$ and $\delta = \delta(c,d, k_{\mathrm{max}}) > 0$  so that, for all $n$ sufficiently large,
\be [IneqSoftCoalBound2MainIneq]
\P[\tCo < \epsilon n^{3} \vert W_{0} = w] \leq 1 - \delta,\\
\P[\tCo < \phi_{\epsilon n}\vert W_{0} = w] \leq 1 - \delta\\
\ee 
uniformly in $1 < k \leq k_{\mathrm{max}}$ and $w \in \Omega_{k}$.
\end{cor}
\begin{proof}
Define the directed graph $\mathcal{X}$ to have vertex and directed edge sets
\be \label{EqDefGraphForCoalProc}
V(\mathcal{X}) &=  \ck, \\
E(\mathcal{X}) &= \{ (x,y) \, : \, x,y \in \ck, \, \P[W_{1} = y | W_{0} = x] > 0 \}.
\ee 
Our key observation here is that it is possible to change any initial configuration $w \in \ck$ into a configuration $w' \in \mathcal{G}_{C}^{(n)}$ by making some number of moves along the edges of $\mathcal{X}$, where the number of moves required is uniformly bounded in $n$ (but may depend on $C,d$ or $k_{\max}$). Since each move takes $O(n^{2})$ steps in the coalescence process, and each move has probability bounded from below, this allows us to apply Lemma \ref{LemSoftCoalBound1} to any initial configuration.  \par 

More formally, we begin by proving the first half of inequality \eqref{IneqSoftCoalBound2MainIneq}. Fix $w \in \Omega_k$, any $0 < \gamma < 1$ and let $\epsilon = \epsilon(c,d,k_{\max},\gamma),  C = C(c,d,k_{\max},\gamma)$ be constants given by Lemma \ref{LemSoftCoalBound1}. 
Overloading notation slightly, if $x,y \in \mathcal{X}$, we denote by $|x-y|$ the length of the shortest path from $x$ to $y$ in $\mathcal{X}$; see Figure \ref{PicShortPath} for an example of such a path when $d=2$ and $C = 2$.

\begin{figure} 
\begin{tikzpicture}[scale=0.35]
    \draw (0, 0) grid (8, 8);
   \draw (16,0) grid (24,8);
   \draw [fill] (3,3) circle [radius=.25];
   \draw [fill] (4,3) circle [radius=.25];
   \draw [fill] (4,4) circle [radius=.25];
   \draw [blue] [->] (3,3) .. controls (2,2).. (1,3); 
    \draw [blue] [->] (4,4) .. controls (5,5).. (4,6); 
    \draw[->] (9,4) -- (15,4);
    \draw [fill] (17,3) circle [radius=.25];
   \draw [fill] (20,3) circle [radius=.25];
   \draw [fill] (20,6) circle [radius=.25];
\end{tikzpicture}
\caption{Short Path}
\label{PicShortPath}
\end{figure}
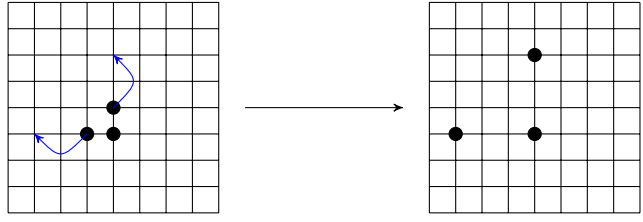

Set
\be
w' = \mathrm{argmin} \{ |w-u| \, : \, u \in \mathcal{G}_{C}^{(n)} {\cap \Omega_k }\}.
\ee 
Let $\Gamma = (w_0 = w, w_{1}, \ldots, w_{|w - w'|} = w')$ be a path from $w$ to $w'$ satisfying $| w_{i} - w_{i+1} | = 1$ for all $i$.  We have
\be \label{IneqSimpleGammaCalc}
\P[(W_{\phi_{0}}, W_{\phi_{1}}, \ldots, W_{\phi_{|w-w'|}}) = \Gamma \, | W_{0} = w] &\geq (2d k_{\mathrm{max}})^{-| w - w'|}\\
 &\geq (2d k_{\mathrm{max}})^{- k_{\mathrm{max}}^{2} (C+1)}.
\ee
 Thus, applying Lemma \ref{LemSoftCoalBound1},
\be 
\P\big[\tCo > \epsilon n^{3} \vert W_{0} = w \big] &\geq \P \big[\tCo > \epsilon n^{3} \vert W_{0} = w' \in \mathcal{G}_{C}^{(n)} \big]   (2d k_{\mathrm{max}})^{- k_{\mathrm{max}}^{2} (C+1)} \\
&\geq {\gamma}  (2d k_{\mathrm{max}})^{- k_{\mathrm{max}}^{2} (C+1)}.
\ee 
This completes the proof of the first half of inequality \eqref{IneqSoftCoalBound2MainIneq}, {with $\delta = \gamma (2d k_{\mathrm{max}})^{- k_{\mathrm{max}}^{2} (C+1)}$}. As with Lemma \ref{LemSoftCoalBound1}, the proof of the second half is essentially identical. 
\end{proof}

We further strengthen this to bound the probability that particles are close to colliding after a small number of steps. Let 
\be
\tNe{i} = \inf \{ t \, : \, \inf_{u,v \, : \, W_{t}[u] = W_{t}[v] = 1} |u-v| \leq i \}
\ee
be the first time that two particles in the coalescence process are within distance $i$.
\begin{lemma} \label{LemmaSoftCoalBound3}
For all $i \in \mathbb{N}$ and all $0 < \gamma < 1$, there exist $\epsilon = \epsilon(c,d, k_{\mathrm{max}}, \gamma) > 0$ and $C = C(c,d,k_{\mathrm{max}}, \gamma) < \infty$  so that, for all $n > N(c,d,k_{\mathrm{max}}, \gamma)$ sufficiently large, 
\be [ineqSoftCoalMainBound3]
\P[\tNe{i} < \epsilon n^{3} \vert W_{0} = w ] \leq 1 - \gamma, \\
\P[\tNe{i} < \phi_{\epsilon n} \vert W_{0} = w ] \leq 1 - \gamma 
\ee 
uniformly in $1 < k \leq k_{\mathrm{max}}$ and $w \in \mathcal{G}_{C}^{(n)}$.
\end{lemma}

\begin{proof}
We begin by proving the first half of inequality \eqref{ineqSoftCoalMainBound3}. Fix $w \in \mathcal{G}_{C}^{(n)}$, fix $0 < \gamma < 1$, let $\delta = 1 - \frac{\gamma}{(4d k_{\mathrm{max}})^{i}}$ and let $C,\epsilon$ be the constants associated with $\delta$ as given by Lemma \ref{LemSoftCoalBound1}. By the definition of $\tNe{i}$, there is a sequence $\Gamma = (w_0 \equiv W_{\tNe{i}}, w_{1}, \ldots, w_{i})$ so that $|w_{j+1} - w_{j}| = 1$ and so that a collision occurs during the transition from $w_{i-1}$ to $w_{i}$. Let $I \in \mathbb{N}$ be such that $\phi_{I} = \tNe{i}$. As in inequality \eqref{IneqSimpleGammaCalc},
\be 
\P[(W_{\phi_{I}}, W_{\phi_{I+1}}, \ldots, W_{\phi_{I+i}}) = \Gamma \, | W_{\phi_{I}} = w_{0}] &\geq (2d k_{\mathrm{max}})^{-i}.
\ee
Since 
\be 
\P[\phi_{I + i} - \phi_{I} \geq n^{2} \log(n)] = o(1), 
\ee 
this implies
\be 
\P[\tCo < \tNe{i} + n^{2} \log(n)] \geq \frac{1}{(2d k_{\mathrm{max}})^{i}} (1 - o(1))
\ee 
uniformly in starting position $W_{0}$. Thus, for all $\epsilon > 0$ and $w \in \mathcal{G}_{C}^{(n)}$,

\be 
 \frac{\gamma}{(4d k_{\mathrm{max}})^{i}} &\geq \P[\tCo < \epsilon n^{3} | W_{0} = w] \\
&\geq \P[\tCo < \epsilon n^{3}, \tNe{i} < \frac{\epsilon}{2} n^{3} | W_{0} = w] \\
&\geq (1 - o(1)) (2dk_{\mathrm{max}})^{-i}  \P[\tNe{i} < \frac{\epsilon}{2} n^{3} | W_{0} = w].
\ee 
We conclude that 
\be 
\P[\tNe{i} < \frac{\epsilon}{2} n^{3} | W_{0} = w] \leq \frac{\gamma}{2} (1 + o(1)),
\ee 
completing the proof of the first half of inequality \eqref{ineqSoftCoalMainBound3}. As with Lemma \ref{LemSoftCoalBound1}, the proof of the second half is essentially identical.
\end{proof}

We then have the following Corollary to Lemma \ref{LemmaSoftCoalBound3}.

\begin{cor}   \label{CorSoftCoalBound4}
Fix $i \in \mathbb{N}$. There exist $\epsilon = \epsilon(c,d, k_{\mathrm{max}},i) > 0$ and $\delta = \delta(c,d, k_{\mathrm{max}},i) > 0$  so that, for all $n$ sufficiently large,
\be 
\P[\tNe{i} < \epsilon n^{3} \vert W_{0} = w] \leq 1 - \delta, \\
\P[\tNe{i} < \phi_{\epsilon n} \vert W_{0} = w] \leq 1 - \delta
\ee 
uniformly in $1 < k \leq k_{\mathrm{max}}$ and $w \in \Omega_{k} \cap \mathcal{G}_{i}^{(n)}$.
\end{cor}
\begin{proof}
This follows from Lemma \ref{LemmaSoftCoalBound3} in essentially the same way that Corollary \ref{CorSoftCoalBound2} followed from Lemma \ref{LemSoftCoalBound1}.
\end{proof}
\subsection{Comparing KCIP with the Coalescence process}
Recall $k_{\mathrm{max}}$ from formula \eqref{eqn:kmax}, the set $\Omega_{k}$ from formula \eqref{EqDefOmegaK} and the sets $\mathcal{G}_{\zeta}^{(n)}$ from formula \eqref{DefNiceStarts}. In this section, we obtain bounds on the occupation measure of $X_{t}$ in $\Omega_{k}$, {uniformly in} $k {\leq} k_{\mathrm{max}}$. 
Recall the definition of the exit times $L_{k}(x)$ from Equation \eqref{EqExitTimeDef}. Our intermediate steps are to show, uniformly in $1 \leq k \leq k_{\mathrm{max}}$, 
\begin{enumerate}
\item Uniformly in $x \in \Omega_{k}$, $\P[L_{k}(x) < \epsilon n^{3}] < 1 - \delta < 1$ for some $\epsilon =\epsilon(c,d,k_{\max})$ and  $\delta =\delta(c,d,k_{\max})$.
\item Uniformly in $x \in \Omega_{k}$, $\P[L_{k}(x) > \epsilon^{-1} n^{3}] < 1 - \delta$.
\item Uniformly in $x \in \Omega_{k}$, $\P[\sum_{t=0}^{\epsilon n^{3}} \textbf{1}_{X_{t} \in \Omega_{{k}}} < \frac{1}{2} \epsilon n^{3}] < 1 - \delta$.
\item  Uniformly in $x \in \Omega_{k}$, $\P[X_{L_{k}(X_{0})} \in \Omega_{k-1} | X_{0}=x] > \delta, \P[X_{L_{k}(X_{0})} \in \Omega_{k+1} | X_{0}=x] > \delta.$
\end{enumerate}

 We begin by proving item (4):
\begin{lemma} \label{LemmaLowerBoundOnExcLength}
There exist $\epsilon = \epsilon(c,d, k_{\mathrm{max}}) > 0$, $\delta = \delta(c,d, k_{\mathrm{max}}) > 0$ and $\alpha = \alpha(c,d,k_{\mathrm{max}}) > 0$  so that, uniformly in $1 \leq k \leq k_{\mathrm{max}}$ and $x \in \Omega_{k}$, we have
\be 
\P[\sum_{t=0}^{\epsilon n^{3}} \mathbf{1}_{X_t \in \Omega_{k}} \geq \alpha n^{3} \vert X_{0} = x ] \geq  \delta.
\ee 
\end{lemma}

\begin{proof}
Assume first that $x \in \Omega_{k} \cap \mathcal{G}_{2}^{(n)}$. We couple $\{ X_{t} \}_{t \in \mathbb{N}}$ to a coalesence process $\{ Z_{t} \}_{t \in \mathbb{N}}$ using the same coupling as in  Lemma \ref{LemmLbNumCol} and define $\{W_{t} \}_{t \in \mathbb{N}}$ as in formula \eqref{EqDefAltCoalProcRep}; since $X_{0} \in \Omega_{k}$, we have $\tS{i} = 0$ for all $1 \leq i \leq k$. Recalling the definition of $\{ \phi_{i} \}_{i \in \mathbb{N}}$ from formula \eqref{EqDefCoalProcMoveTimes}, we fix $T \in \mathbb{N}$ and consider the two events:
\begin{enumerate}
\item $\mathcal{A}_{1,T}$: The coalescence process has $\tNe{2} < \phi_{T}$. 
\item $\mathcal{A}_{2,T}$: The KCIP has $\inf \{ t \, : \, \delta_{t} > 1 \} < \phi_{T}$.
\end{enumerate}
By Lemma \ref{LemmaSoftCoalBound3}, there exists some $\epsilon_{1}, \gamma > 0$ so that 
\be \label{Ineq1LowerBoundOnExcLength}
\P[\mathcal{A}_{1,\epsilon n}] \leq 1 - \gamma + o(1)
\ee 
for $0 < \epsilon < \epsilon_{1}$. By Corollary \ref{LemmaNateshLowerBound}, there exists some $\epsilon_{2}$ so that for $0 < \epsilon < \epsilon_{2}$,
\be \label{Ineq2LowerBoundOnExcLength}
\P[\mathcal{A}_{2,\epsilon n}] \leq \frac{\gamma}{2} + o(1).
\ee 
Choose $\epsilon = \frac{1}{2} \min(\epsilon_{1}, \epsilon_{2})$ and fix $\alpha, \beta > 0$.
Set $\mathcal{A}_{\epsilon n} =  \mathcal{A}_{1,\epsilon n} \cup \mathcal{A}_{2,\epsilon n}$. By inequalities \eqref{Ineq1LowerBoundOnExcLength} and \eqref{Ineq2LowerBoundOnExcLength}, 
\be 
\P[\sum_{t=0}^{\beta n^{3}} \textbf{1}_{X_{t} \in \Omega_{k}} \geq \alpha n^{3} \vert X_{0} = x ] &\geq \P[\sum_{t=0}^{\beta n^{3}} \textbf{1}_{X_{t} \in \Omega_{k}} \geq \alpha n^{3} \vert X_{0} = x, \mathcal{A}_{\epsilon n}^{c} ]  \P[ \mathcal{A}_{\epsilon n}^{c} ] \\
&\geq \frac{\gamma}{2}\P[\sum_{t=0}^{\beta n^{3}} \textbf{1}_{X_{t} \in \Omega_{k}} \geq \alpha n^{3} \vert X_{0} = x, \mathcal{A}_{\epsilon n}^{c} ] + o(1).
\ee 
Conditionally on $W_{\phi_{i}}$, $\phi_{i+1} - \phi_{i}$ has geometric distribution with mean between $\frac{n^{2}}{cdk_{\mathrm{max}}}$ and $\frac{n^{2}}{cd}$. Thus, for all $\beta$ sufficiently large, 
\be \label{NearEndTrivIneq1Really}
\P[\beta^{-1} n^{3} \leq \phi_{\epsilon n} \leq \beta n^{3}] = 1 - o(1).
\ee
Since $\P[\mathcal{A}_{\epsilon n}^{c}] \geq \frac{\gamma}{4} + o(1)$ is bounded away from 0, this implies that, for $\beta$ sufficiently large,
\be \label{NearEndTrivIneq2Really}
\P[\sum_{t=0}^{\beta n^{3}} \textbf{1}_{X_{t} \in \Omega_{k}} \geq \alpha n^{3} \vert X_{0} = x ] &\geq \frac{\gamma}{4} \P[\sum_{t=0}^{\phi_{\epsilon n}} \textbf{1}_{X_{t} \in \Omega_{k}} \geq \alpha n^{3} \vert X_{0} = x,\mathcal{A}_{\epsilon n}^{c} ] + o(1).
\ee 
By checking allowed sequences of update variables $\{ v_{s}, p_{s} \}_{s \in \mathbb{N}}$, if $V_{t} \leq k_{\mathrm{max}}$ and $\delta_{t} = 0$, we have 
\be 
\P[\delta_{t+1} = 1 | \delta_{t} = 0, V_{t} \leq k_{\max}, t \leq \phi_{\epsilon n}, \mathcal{A}_{\epsilon n}^{c}] \leq \P[\delta_{t+1} = 1 | \delta_{t} = 0, V_{t} \leq k_{\max}] \leq   \frac{2cd k_{\max}}{n^{2}},
\ee 
while if $V_{t} \leq k_{\mathrm{max}} + 1$ and $\delta_{t} = 1$,
\be 
\P[\delta_{t+1} = 0 |  \delta_{t} = 1, &V_{t} \leq k_{\max}+1, t \leq \phi_{\epsilon n}, \mathcal{A}_{\epsilon n}^{c}] \\
&\geq \P[\delta_{t+1} = 0 |  \delta_{t} = 1, V_{t} \leq k_{\max}+1] \\
& \geq  \frac{1}{2n}(1 + o(1)).
\ee 
Since the latter rate is $\Theta(n)$ times larger than the former, we have
\be 
\P[\sum_{t=0}^{\phi_{\epsilon n}} \textbf{1}_{X_{t} \in \Omega_{k}} \geq \alpha \phi_{\epsilon n} \vert X_{0} = x, \mathcal{A}_{\epsilon n}^{c} ] = 1 + o(1)
\ee 
for all $\alpha^{-1}$ sufficiently large. Combining this with inequalities \eqref{NearEndTrivIneq1Really} and \eqref{NearEndTrivIneq2Really}, we have for $\beta, \alpha^{-1}$ sufficiently large that 
\be 
\P[\sum_{t=0}^{\beta n^{3}} \textbf{1}_{X_{t} \in \Omega_{k}} \geq \alpha \beta^{-1} n^{3} \vert X_{0} = x ] \geq \frac{\gamma}{8} + o(1).
\ee 
This completes the proof of the Lemma when  $x \in \Omega_{k} \cap \mathcal{G}_{2}^{(n)}$.\par
We give a short argument reducing the case $x \in \Omega_{k}$ to the case $x \in \Omega_{k} \cap \mathcal{G}_{2}^{(n)}$; it is essentially the same argument given in Corollary \ref{CorSoftCoalBound2}. Fix $x \in \Omega_{k} \backslash \mathcal{G}_{2}^{(n)}$. There exists $x' \in \Omega_{k} \cap \mathcal{G}_{2}^{(n)}$ and a sequence of configurations $\Gamma = (x = x_{0}, x_{1}, \ldots, x_{\ell} = x')$ so that $|x_{i+1} - x_{i} | \leq 1$ and $\ell \leq 16 k_{\mathrm{max}}^{2}$. The same argument as given in inequality \eqref{IneqSimpleGammaCalc} implies that there exists some $\gamma' = \gamma'(c,d,k_{\mathrm{max}}) > 0$ so that
\be 
\P[\inf\{t \, : \, X_{t} = x' \} < n^{2} \log(n) | X_{0} = x] > \gamma'.
\ee 
This bound reduces the case $x \in \Omega_{k}$ to the case  $x \in \Omega_{k} \cap \mathcal{G}_{2}^{(n)}$, at the cost only of replacing $\gamma$ with $\gamma \gamma'$ and replacing $\epsilon n^{3}$ with $\epsilon n^{3} + n^{2} \log(n) = (1 + o(1))\epsilon n^{3}$. This finishes the proof.
\end{proof}

Define 
\be \label{eqn:rk}
\rho_{k} = \inf \{ t \, : \, X_{t} \in \Omega_{k} \}.
\ee 
 \begin{lemma} \label{LemmaIneqTransProbKNearK}
For all $1 \leq k \leq k_{\mathrm{max}}$, there exists some $\delta, \epsilon > 0$ so that 
\be[IneqTransProbKNearK]
\P[\rho_{k-1} < \epsilon n^{3} | X_{0} = x] > \delta, \\
\P[\rho_{k+1} < \epsilon n^{3} | X_{0} = x] > \delta 
\ee 
holds uniformly in $x \in \Omega_{k}$ (where the first part of the inequality obviously requires $k \geq 2$).
\end{lemma}
\begin{proof}
As argued in Corollary \ref{CorSoftCoalBound2} (and again in the last paragraph of the proof of Lemma \ref{LemmaLowerBoundOnExcLength}), we can assume without loss of generality that $x \in \Omega_{k} \cap \mathcal{G}_{4}^{(n)}$, since for any fixed $\epsilon > 0$, the ratio 
\be 
\frac{\inf_{x \in \Omega_{k}} \P[\rho_{k-1} < \epsilon n^{3} | X_{0} = x] }{\inf_{x \in \Omega_{k} \cap \mathcal{G}_{4}^{(n)} } \P[\rho_{k-1} < \epsilon n^{3}(1 + o(1)) | X_{0} = x] }
\ee is bounded away from 0 uniformly in $n$. Thus, it is sufficient to consider only starting positions $X_{0} \in \Omega_{k} \cap \mathcal{G}_{4}^{(n)}$.  \par
 Fix $1 \leq k \leq k_{\mathrm{max}}$ and $X_0 = x \in \Omega_{k} \cap \mathcal{G}_{4}^{(n)}$. We now couple $\{ X_{t} \}_{t \in \mathbb{N}}$ to a coalesence process $\{ Z_{t} \}_{t \in \mathbb{N}}$ using the same coupling as in  Lemma \ref{LemmLbNumCol} and define $\{W_{t} \}_{t \in \mathbb{N}}$ according to formula \eqref{EqDefAltCoalProcRep}. Since $X_{0} \in \Omega_{k}$, we have $\tS{i} = 0$ for all $1 \leq i \leq k$. Define the sequence of movement times $\{ \phi_{j} \}_{j \geq 0}$ as in formula \eqref{EqDefCoalProcMoveTimes} and the graph $\mathcal{X}$ as in formula \eqref{EqDefGraphForCoalProc}. Fix $\epsilon > 0$; we will define $\Psi$ to be the collection of paths through the graph $\mathcal{X}$ for which the associated coalescence process has $\tNe{4} \leq \phi_{\epsilon n}$. More precisely, we say that a path $\Gamma = (w_{0}, w_{1}, \ldots, w_{m})$ through the graph $\mathcal{X}$ is in $\Psi$ if it satisfies:
\begin{enumerate}
\item $m \leq \epsilon n$. 
\item $| w_{i+1} - w_{i} | = 1$ for all $0 \leq i < m$, where distance is measured according to the graph distance on $\mathcal{X}$.
\item $w_{0} = x$. 
\item There exist $u,v \in G$ so that $|u-v| = 4$ and $w_{m}[u] = w_{m}[v] = 1$. 
\item For all $u,v \in G$ with $|u - v | = 4$ and all $0 \leq i < m$, we have $w_{i}[u] w_{i}[v] = 0$.
\end{enumerate}
For any path $\Gamma \in \Psi$ of length $m$ and any $A > 0$,   we have by essentially the same calculation as in Lemma \ref{LemmaNateshLowerBoundOrig} that 
\be 
\P[  \sup_{0 \leq t \leq \tNe{4}} \delta_{t} \leq 1 | \{\phi_{m} > A n^{3} \}, \{W_{\phi_{i}} \}_{i=0}^{m} = \Gamma ] \geq g(A) + o(1)
\ee 
for some function $g(A) > 0$. Combining this with inequality \eqref{NearEndTrivIneq1Really}, we have for all $A$ so that $\frac{A}{\epsilon}$ is sufficiently large that
\be 
\P[ \{ \sup_{0 \leq t \leq \tNe{4}} \delta_{t} \leq 1 \} \cap \{ \phi_{m} \leq A n^{3} \} |  \{W_{\phi_{i}} \}_{i=0}^{m} = \Gamma ] \geq g(A) + o(1).
\ee 
In particular, there exist $\epsilon, A, B >0$ so that 
\be \label{IneqTranProbNearKGoodPaths}
\P[ \{ \sup_{0 \leq t \leq \tNe{4}} \delta_{t} \leq 1 \} \cap \{ \phi_{m} \leq A n^{3} \} |  \{W_{\phi_{i}} \}_{i=0}^{m} = \Gamma ] \geq B + o(1).
\ee 
We choose such a triple $\epsilon, A,B$ for the remainder of the proof of the first half of inequality \eqref{IneqTransProbKNearK}.  Define $\rho_{k-1}' = \inf \{t \, : \,  Y_{t} = k-1 \}$. By the arguments in Proposition \ref{propNearColImpCol}, there exists $\gamma' > 0$ so that
\be 
\P[\rho_{k-1}' < 4 n^{2.5} | X_{0} = y] \geq \gamma' (1 + o(1))
\ee 
holds uniformly in $y \in \Omega_{k} \backslash \mathcal{G}_{2}^{(n)}$. By the same reduction argument used in Corollary \ref{CorSoftCoalBound2}, this implies that there exists some $\gamma > 0$ so that
\be 
\P[\rho_{k-1}' < 4 n^{2.5} | X_{0} = y] \geq \gamma (1 + o(1))
\ee 
holds uniformly in $y \in \Omega_{k} \backslash \mathcal{G}_{4}^{(n)}$. 
Recall the definition of the update variables $(p_{t}, v_{t})$ from Equation \eqref{EqCiRep}. Denote the event
 \be \label{EqDefChiCond}
\chi_n(t) =  \{ \forall \, s \, : \, t \leq s \leq  t+  n \log(n)^{2}, \, p_{s} > \frac{c}{n} \, \, \mathrm{ or } \sum_{(u,v_{s}) \in E(\LL) }X_{s}[u]=0 \}.
 \ee
We then have
\be \label{IneqToFix}
\P &[\rho_{k-1} < 4 n^{2.5} + n \log(n)^{2} | X_{0} = y] \\ 
&\geq \sum_{t = 0}^{4 n^{2.5}} \P[\rho_{k-1} < t + n \log(n)^{2} | \rho_{k-1}' = t] \P[\rho_{k-1}' = t] \\
&\geq \frac{1}{3} \sum_{t = 0}^{4 n^{2.5}} \P[ \{ \cup_{s=t}^{t+n \log(n)^{2}} \{ v_{s} \} = {\LL} \} \cap \chi_n(t) | \rho_{k-1}' = t] \P[\rho_{k-1}' = t] \\
&\geq \frac{1}{3}(1 + o(1)) \sum_{t = 0}^{4 n^{2.5}} \P[\rho_{k-1}' = t] \\
&\geq \frac{\gamma}{3}(1 + o(1)).
\ee 
Combining this with inequality \eqref{IneqTranProbNearKGoodPaths} and the observation that $w_{m} \in \Omega_{k} \backslash \mathcal{G}_{4}^{(n)}$ for all $\Gamma = (w_{0}, \ldots, w_{m}) \in \Psi$, we have 
\be 
\P[\rho_{k-1} < A n^{3} + n^{2.5} + n\log(n)^{2} |  \{W_{\phi_{i}} \}_{i=0}^{m} = \Gamma ] \geq \frac{ B\gamma}{3}  + o(1).
\ee 
This implies that 
\be 
\P[& \rho_{k-1} < A n^{3} + n^{2.5} + n \log(n)^{2} | X_{0} = x] \\
&\geq \sum_{\Gamma \in \Psi} \P[\rho_{k-1} < A n^{3} + n^{2.5} + n \log(n)^{2} |  \{W_{\phi_{i}} \}_{i=0}^{m} = \Gamma ] \P[ \{W_{\phi_{i}} \}_{i=0}^{m} = \Gamma | X_{0} = x ] \\
&\geq \frac{B \gamma}{3} \sum_{\Gamma \in \Psi} \P[ \{W_{\phi_{i}} \}_{i=0}^{m} = \Gamma | X_{0} = x ] +o(1)\\
&= \frac{B \gamma}{3} \P[\tNe{4} \leq \phi_{\epsilon n} | X_{0} = x ] +o(1)\\
&\geq \frac{B \gamma}{3} \P[\tCo \leq \phi_{\epsilon n} | W_{0} = x ] +o(1).
\ee 
The last quantity, $\P[\tCo \leq \phi_{\epsilon n} | W_{0} = x ]$, is bounded away from 0 by Theorem 5 of \cite{Cox89}. This completes the proof of the first half of inequality \eqref{IneqTransProbKNearK}. \par
 We now prove the second half of inequality \eqref{IneqTransProbKNearK}.  Fix $x \in \mathcal{G}_{4}^{(n)}$ and $\epsilon > 0$. We consider a new collection of paths $\Psi$ on $\mathcal{X}$. Roughly speaking, these will be the paths for which there are no near-collisions for many steps. More precisely, we say that $\Gamma = (x = w_{0}, w_{1},\ldots, w_{m})$ is in $\Psi$ if
\begin{itemize} 
\item $m \geq \frac{1}{2} \epsilon n$.
\item $| w_{i+1} - w_{i} | = 1$ for all $0 \leq i < m$. 
\item $w_{0} = x$. 
\item For all $u,v \in G$ with $|u - v | \leq 4$ and all $0 \leq i \leq m$, we have $w_{i}[u] w_{i}[v] = 0$.
\end{itemize}
Recall that $\tDe{j}$ is defined in formula \eqref{EqDecoupDefMain} and $\ztr^{(j)}$ is defined in formula \eqref{EqTripleTimeGen}. If $ \{W_{\phi_{i}} \}_{i=0}^{m} = \Gamma \in \Psi$, then no near-collisions of $W_{t}$ have occurred by time $\phi_{m}$, and so conditioned on this event we have
\be \label{EndLemmaObs1}
\min_{1 \leq j \leq k} \tDe{j} \geq \min( \phi_{m}, \min_{1 \leq j \leq k}\ztr^{(j)}).
\ee 
 Also, for any $t \in \mathbb{N}$ and any starting point $X_{0} = x \in \Omega_{k}$ with $\delta_{0} = 1$, we have 
\be \label{EndLemmaObs2}
\P[\min_{0 \leq s \leq t} \delta_{s} \geq 1 | X_{0} = x] \geq (1 - \frac{2}{n})^{t}. 
\ee 
In particular, the indicator function of the event $\{\min_{0 \leq s \leq t} \delta_{s} \geq 1\}$ is stochastically dominated by a geometric random variable with mean $\frac{n}{2}$. Combining inequalities \eqref{EndLemmaObs1} and \eqref{EndLemmaObs2}, we have 
\be \label{IneqEndLemmaNormalPath}
 \P[\{ \sum_{t = 0}^{\phi_{m}} \textbf{1}_{\delta_{t} \geq 1} \geq \frac{ \epsilon}{32} n^{2} \} \cup \{ \min_{1 \leq j \leq k}\ztr^{(j)} \leq \phi_{m} \}  | \{W_{\phi_{i}} \}_{i=0}^{m} = \Gamma] = 1 - o(1).
\ee 
Noting that 
\be 
\P[\min_{1 \leq j \leq k}\ztr^{(j)} = t + 1 | \{W_{\phi_{i}} \}_{i=0}^{m} = \Gamma, \delta_{t} = 1, \min_{1 \leq j \leq k}\ztr^{(j)} \geq t, t \leq \phi_{i}] \geq \frac{c}{n^{2}},
\ee
we obtain 
\be 
\P[\min_{1 \leq j \leq k}\ztr^{(j)} >  \phi_{m} |  \{W_{\phi_{i}} \}_{i=0}^{m} = \Gamma, \sum_{t = 0}^{\phi_{m}} \textbf{1}_{\delta_{t} \geq 1} \geq \frac{ \epsilon}{32} n^{2}] &\leq \big(1 - \frac{c}{n^{2}} \big)^{\frac{\epsilon}{32} n^{2}} \\
&\leq e^{-\frac{\epsilon c}{32}}.
\ee 
Combining this with inequality \eqref{IneqEndLemmaNormalPath} and rearranging terms, this implies
\be 
\P[\min_{1 \leq j \leq k}\ztr^{(j)} &\leq \phi_{m}   | \{W_{\phi_{i}} \}_{i=0}^{m} = \Gamma] \\
&\geq 1 - o(1) - \P[\min_{1 \leq j \leq k}\ztr^{(j)} >  \phi_{m} |  \{W_{\phi_{i}} \}_{i=0}^{m} = \Gamma, \sum_{t = 0}^{\phi_{m}} \textbf{1}_{\delta_{t} \geq 1} \geq \frac{ \epsilon}{32} n^{2}] \\
&\geq 1 - e^{-\frac{\epsilon c}{32}} - o(1).
\ee 
Summing over $\Gamma \in \mathcal{X}$,
\be 
\P[\min_{1 \leq j \leq k}\ztr^{(j)} \leq \min(\phi_{m}, \tNe{4})] &\geq \sum_{\Gamma \in \mathcal{X}} \P[\min_{1 \leq j \leq k}\ztr^{(j)} \leq \phi_{m}   | \{W_{\phi_{i}} \}_{i=0}^{m} = \Gamma]  \P[\{W_{\phi_{i}} \}_{i=0}^{m} = \Gamma] \\
&\geq (1 - e^{-\frac{\epsilon c}{32}} - o(1)) \sum_{\Gamma \in \mathcal{X}} \P[\{W_{\phi_{i}} \}_{i=0}^{m} = \Gamma] \\
&\geq (1 - e^{-\frac{\epsilon c}{32}} - o(1)) \P[ \tNe{4} \geq \phi_{\frac{\epsilon}{2}}].
\ee 
Write $\ztr^{(\min)} = \min_{1 \leq j \leq k} \ztr^{(j)}$. By Corollary \ref{CorSoftCoalBound4}, we conclude that there exists some $\gamma > 0$ so that 
\be 
\P[\ztr^{(\min)} \leq \min(\phi_{m}, \tNe{4})] \geq \gamma + o(1).
\ee 
Combining this with the calculation in inequality \eqref{IneqToFix},
\be 
\P[\rho_{k+1} &< \phi_{\frac{\epsilon n}{2}} + n \log(n)^{2} | X_{0} = y] \\
&\geq \P[\{ \cup_{s=\ztr^{(\min)} }^{\ztr^{(\min)} +n \log(n)^{2}} \{ v_{s} \} = {\LL} \} \cap \chi_n(\ztr^{(\min)}) | \ztr^{(\min)} \leq \min(\phi_{\frac{\epsilon}{2} n}, \tNe{4}), X_{0} = y] \\
&\times \P[\ztr^{(\min)} \leq \min(\phi_{\frac{\epsilon}{2} n}, \tNe{4}) | X_{0} = y] \\
&\geq (1 - o(1))(\gamma - o(1)).
\ee 
Combining this with inequality \eqref{NearEndTrivIneq1Really} completes the proof. 
\end{proof}

\section{Proof of Theorem \ref{ThmMainResult}} \label{SecProofThm}

In this section, we find bounds on the occupation times of $\Omega_{k}$ and use these bounds to finish the proof of Theorem \ref{ThmMainResult}. 
Recall from \eqref{eqn:kappak} that the quantity $\kappa_k$ applied to $X_t$ is 
\be 
\kappa_{k}(T) = \sum_{t=0}^{T} \textbf{1}_{X_{t} \in \Omega_{k}}.
\ee  
\begin{lemma} [Bound on Occupation Measures] \label{LemmaOccMeasureBound}
There exists $0 < \gamma < \infty$ so that, for all $A$ sufficiently large and all $B > 0$,
\be 
\P[\kappa_{k}(A n^{3} \log(n)) < B n^{3} \log(n)] \leq \gamma \frac{B}{A} + o(1)
\ee
uniformly in $X_{0} = x \in \Omega$ and $1 \leq k \leq k_{\mathrm{max}}$.
\end{lemma} 

\begin{proof}
Recall the set $\ck$ from \eqref{eqn:CalG} and $k_\mathrm{max}$ from \eqref{eqn:kmax}.
By Corollary \ref{CorMainDriftResult}, there exists $c_{1} > 0$ so that for all $A$ sufficiently large,
\be \label{IneqLemmBoundOcc1}
\P[\sum_{t=0}^{A n^{3} \log(n)} \textbf{1}_{X_{t} \in \ck} > c_{2} A n^{3} \log(n)|X_0 = x\in \Omega] = 1 + o(1).
\ee 
 Recall the definition of the update variables $(p_{t}, v_{t})$ from formula \eqref{EqCiRep} and let $\chi_n \equiv \chi_{n}(0)$ from formula \eqref{EqDefChiCond}. For any $x \in \ck$, we have that 
\be 
\P\Big[\inf \{t >0 \, : \, X_{t} \in \ck \cap \cup_{k} \Omega_{k} \} &< n \log(n)^{2} | X_{0} = {x\in \ck }\Big] \\
&\geq \P \Big[ \{ \cup_{s=0}^{n \log(n)^{2}} \{ v_{s} \} = {\LL} \} \cap \chi_n|X_0=x \in \ck \Big] \\
&= 1 + o(1). \label{IneqLemmBoundOcc2}
\ee 
Combining inequalities \eqref{IneqLemmBoundOcc1} and \eqref{IneqLemmBoundOcc2} with Lemma \ref{LemmaLowerBoundOnExcLength}, there exists some $c_{2} > 0$ and $1 \leq k' \leq k_{\max}$ so that, for all $A$  sufficiently large,
\be \label{IneqBaseCase}
\P[ \sum_{t=0}^{A n^{3} \log(n)} \textbf{1}_{X_{t} \in \Omega_{k'}} > c_{2} A n^{3} \log(n)|X_0 = x\in \Omega] = 1 + o(1).
\ee 
Inequality \eqref{IneqBaseCase}  implies the existence of some $c_{3} > 0$ so that, for any $X_0 = x \in \Omega$,
\be \label{IneqBaseCase2}
\E[\sum_{t=0}^{A n^{3} \log(n)} \textbf{1}_{X_{t} \in \Omega_{k'}}] \geq c_{3} A n^{3} \log(n)
\ee 
for all $A$ sufficiently large. We claim that inequality \eqref{IneqBaseCase2} also holds with $k'$ replaced by $k'+1$ and also (if $k' \geq 2$) with $k'-1$.  By Lemmas \ref{LemmaIneqTransProbKNearK} and \ref{LemmaLowerBoundOnExcLength}, there exists some $c_{4} > 0$ so that 
\be[IneqBoundOccFirstConc]
\E[\sum_{t=0}^{A n^{3} \log(n)} \textbf{1}_{X_{t} \in \Omega_{k'+1}}] &\geq c_{4} \E[\sum_{t=0}^{A n^{3} \log(n)} \textbf{1}_{X_{t} \in \Omega_{k'}}] \geq c_{4} c_{3} A n^{3} \log(n), \\
\E[\sum_{t=0}^{A n^{3} \log(n)} \textbf{1}_{X_{t} \in \Omega_{k'-1}}] &\geq c_{4} \E[\sum_{t=0}^{A n^{3} \log(n)} \textbf{1}_{X_{t} \in \Omega_{k'}}] \geq c_{4} c_{3} A n^{3} \log(n).
\ee 
This is exactly inequality \eqref{IneqBaseCase2} with $k'$ replaced by $k'+1$ and $k'-1$ respectively. Since the set $\{1,2,\ldots,k_{\mathrm{max}} \}$ is finite, the argument between inequalities \eqref{IneqBaseCase2} and \eqref{IneqBoundOccFirstConc} implies that there exists some $c_{5} > 0$ so that 
\be \label{IneqPrettyMuchLastInequality}
\E[\sum_{t=0}^{A n^{3} \log(n)} \textbf{1}_{X_{t} \in \Omega_{j}}] \geq c_{5} A n^{3} \log(n)
\ee 
for all $1 \leq j \leq k_{\mathrm{max}}$. Since $\kappa_{k}(T) = \sum_{t=0}^{T} \textbf{1}_{X_{t} \in \Omega_{k}}$,
the result now follows from an application of Markov's inequality to inequality \eqref{IneqPrettyMuchLastInequality}.\end{proof}

Finally, we prove our main result: 

\begin{proof}[Proof of Theorem \ref{ThmMainResult}]
The lower bound in Theorem \ref{ThmMainResult} is given in Theorem \ref{ThmLowerBoundOnMixingTime}. We now  show the proof of the upper bound in Theorem \ref{ThmMainResult} by applying the bounds in 
Lemma \ref{CorMixingTimeRestWalk} and Lemma \ref{LemmaOccMeasureBound} to Lemma \ref{LemmaBasicMixing}.\par
 Lemma \ref{LemmaOccMeasureBound} implies that there exist constants $0 < A,B < \infty$, and a function $N = N(A,B,c,d)$ so that for all $k \leq k_{\max}$ (see formula \eqref{eqn:kmax}) and all $n > N$,
\be \label{eqn:MT1}
\sup_{x \in \Omega} \P[\kappa_{k}(A n^{3} \log(n)) < B n^{3} \log(n) | X_{0} = x] \leq {1\over 16}.
\ee
Similarly,  Lemma \ref{CorMixingTimeRestWalk} implies that there exists a constant $0 < C_1 < \infty$ so that
\be \label{eqn:MT2}
\max_{1 \leq k \leq k_{\max}} \tau_{n,k} \leq C_1 n^{2 + \frac{2}{d}} \log(n). 
\ee

In the notation of  Lemma \ref{LemmaBasicMixing}, let $\Theta = \Omega$, let $\Theta_{k} = \Omega_{k}$ for $1 \leq k \leq {n \over 2}$ and let $\Theta_{{{n \over 2} +1}} = \Omega \backslash \cup_{k=1}^{{n\over 2}} \Omega_{k}$. Set $I = \{1,2,\ldots,k_{\max}\}$. By Theorem \ref{LemmaContractionEstimate}, if the KCIP is started at the stationary distribution $\pi$, there exists $0<\alpha,\epsilon<1$ so that
\be 
\E[V_{0}] = \E[V_{{\epsilon n^3}}] \leq (1-\alpha) \E(V_{0}) + C_{G}.
\ee 
Thus we get $\E(V_0) \leq {C_{G} \over \alpha}  = {k_{\max} \over 4}$. Markov's inequality yields that
\be
\P[V_{0} \leq k_{\max}] \geq \frac{3}{4}.
\ee
Since $V_0$ is distributed according to $\pi$, this immediately yields
\be
\pi(\cup_{k \in I} \Omega_{k}) \geq \frac{3}{4} - o(1). 
\ee
Thus, in the notation of Lemma \ref{LemmaBasicMixing}, we can fix for all $n> N = N(c,d,A,B)$ sufficiently large,
\be 
a &= \frac{3}{8},\, \quad
\beta = \frac{11}{16}, \, \quad
\gamma = \frac{1}{11}, \\
t &= 16 \,c_{\gamma}'\, C_1 n^{2 + \frac{2}{d}} \log(n) \leq Bn^{3} \log(n), \\
T &= \lceil A n^{3} \log(n) \rceil
\ee 
where $C_1$ is the constant from \eqref{eqn:MT2}.
By \eqref{eqn:MT2}, we have 
\be \label{IneqReallyEnd1}
\max_{1 \leq k \leq k_{\max}} \frac{\tau_{n,k} c_{\gamma}'}{t} \leq \frac{1}{16}.
\ee 
By \eqref{eqn:MT1}, we have that for ${1 \leq k \leq k_{\max}}$,
\be \label{IneqReallyEnd2}
\sup_{x \in \Omega} \P[\kappa_{k}(T) < t | X_{0} = x] \leq {1\over 16}.
\ee
Combining inequalities \eqref{IneqReallyEnd1} and \eqref{IneqReallyEnd2} immediately implies that, in the notation of Lemma \ref{LemmaBasicMixing}, $\mathcal{T} \leq T = O(n^{3} \log(n))$. Thus by Lemma \ref{LemmaBasicMixing} we have that $\tmix = O(n^{3} \log(n))$ and the proof of Theorem \ref{ThmMainResult} is complete.
\end{proof}
\section{Conclusion and future work} \label{SecConclusion}

We have resolved only some special cases in Aldous' conjecture. In particular, we do not have any results for: 
\begin{itemize}
\item Graphs other than the torus.
\item Density regimes other than $p = \frac{c}{n}$.
\end{itemize}

In this section, we give conjectures for other graphs and regimes, and mention cases for which our methods work well. Before doing so, there is the question as to whether the adjustment to Aldous' conjecture required by Theorem  \ref{ThmLowerBoundOnMixingTime} is essentially the only required correction. If so, this would suggest the conjecture:
\begin{conj} \label{ConjNotReal}
Fix $d \in \mathbb{N}$ and let $G$ be a $d$-regular graph. The mixing time $\tmix$ of the KCIP with parameter $p$ on graph $G$ is $O(p^{-1} d | G | \tmixR + \frac{|G|}{d^{3} p^{2}})$ as $|G|$ goes to infinity, where $\tmixR$ is the mixing time of the $\frac{1}{2}$-lazy simple random walk on the graph $G$.
\end{conj}

We believe that this gives the correct answer for the torus in dimension $d \geq 3$ (that is, we conjecture that the correct mixing time for the process described in Theorem \ref{ThmMainResult} is $O \big( n^{3} \big)$). However, we do not believe that Conjecture \ref{ConjNotReal} is true in general. In particular, define the coalescence time $\tau_{\mathrm{Coal}}$ of a graph to be the expected time for the coalescence process on $G$ to go from $| G |$ particles to 1 particle. We suspect that the mixing time of the KCIP is bounded from below by the coalesence time of the associated coalesence process on the same graph. By \cite{Cox89}, this suggests that the mixing time for the torus in dimension $d = 2$ is at least $n^{3} \log(n)$, while Conjectures  \ref{ConjAldous} and \ref{ConjNotReal} both suggest the mixing time is $O( n^{3})$. In general, we conjecture that: 
\begin{conj} \label{ConjMaybeBetter}
Fix $d \in \mathbb{N}$ and let $G$ be a $d$-regular graph. The mixing time $\tmix$ of the KCIP with parameter $p$ on graph $G$ is $O(p^{-1} d | G | \tmixR + \frac{|G|}{d^{3} p^{2}} + p^{-1} d |G| \tau_{\mathrm{Coal}})$ as $|G|$ goes to infinity, where $\tmixR$ is the mixing time of the $\frac{1}{2}$-lazy simple random walk on the graph $G$ and $\tau_{\mathrm{Coal}}$ is the coalesence time of the graph $G$.
\end{conj}
Next, we discuss when our strategy outlined in Section \ref{sec:proofsketch} may be applicable for proving Conjecture \ref{ConjMaybeBetter}. When restricted to densities in the regime $p = \frac{c}{n}$, our proof strategies are likely to work well for many other sequences of bounded-degree graphs. In particular, for random triangle-free $d$-regular graphs, our argument goes through with only two major changes. The first change is to replace all bounds on the coalesence process from \cite{Cox89} with analogous bounds from \cite{Oliv12b} and \cite{Oliv13}. These bounds are substantially looser, but strong enough for our arguments to go through. Next, random $d$-regular graphs are expanders with high probability (see, \textit{e.g.}, \cite{HLW06}), and in particular have spectral gaps that are uniformly bounded below in $n$ and mixing times that grow like $\log(n)$. This allows us to make the second change, replacing the bound on the log-Sobolev constant from \cite{Yau97} with a bound on the spectral gap from \cite{HLW06}. Besides the invocation of \cite{Yau97}, all comparison arguments for the log-Sobolev constant of $\LL$ given in Section \ref{SecCompToExc} also apply as written to bounding the spectral gap of general $d$-regular graphs. These graphs have such large spectral gaps that the bounds obtained this way are sufficient. Beyond expanders, we generally expect our strategy to succeed for families of bounded-degree graphs with
\be \label{sec:conjrel}
\max(\tau_{\mathrm{Coal}}, \tmixR) = O(n^{3}).
\ee
In particular a similar approach works for the lattice in $d \geq 3$ dimensions. When \eqref{sec:conjrel} fails to hold, as with $\LL$ in dimension $d=2$, the arguments in Section \ref{SecDriftCond} must be substantially changed. \par 

For sequences of $m = m(n)$-regular graphs (with $m(n)$ very slowly growing with $n$), our strategy could be refined to give nontrivial bounds, but our results are not very useful as written. For graphs with very large degrees $d \approx n$, a straightforward comparison argument to the KCIP on the complete graph analogous to that given in \cite{DiSa93b} for the simple exclusion process gives useful bounds. \par 

Our strategies are unlikely to work well for $p = p_n \gg \frac{c}{n}$, as the arguments this paper rely quite strongly on the stationary measure of the constrained Ising process being concentrated on configurations with few particle. Despite technical difficulties, we believe that the$O(n^{3})$ bound will hold up to $p_n \approx n^{-\frac{1}{2} - \frac{1}{d}}$. Again, for $p$ sufficiently large, a straightforward comparison argument to the constrained Ising process on the complete graph analogous to that given in \cite{DiSa93b} for the simple exclusion process gives useful bounds. \par

\section*{Acknowledgements}
NSP is partially supported by an ONR grant. We thank Daniel Jerison and Anastasia Raymer for helpful conversations.
\bibliographystyle{alpha}
\bibliography{CIBib}
\end{document}